\documentclass[12pt,a4paper,oneside,openright]{book} 

\usepackage{amsmath}

\usepackage{amssymb, amsthm, paralist, mathrsfs} 

\newcommand{\vect}[2]{\ensuremath{{#1}_1, \dots, {#1}_{#2}}} 
\newcommand{\ol}[1]{\ensuremath{\overline{#1}}} 
\newcommand{\kin}[1]{\ensuremath{\mathrm{Kin}(#1)}}

\usepackage{tikz} 
\usepackage[subrefformat=parens]{subfig}
\usetikzlibrary{calc, shapes.geometric}

\theoremstyle{plain}

\newtheorem{thm}{Theorem}[chapter]
\newtheorem{lemma}[thm]{Lemma}
\newtheorem{conj}[thm]{Conjecture}

\theoremstyle{definition}
\newtheorem{dfn}[thm]{Definition}

\numberwithin{equation}{thm}

\usepackage{graphicx} 

\usepackage{hyperref} 

\usepackage{setspace} 
\usepackage{msor_thesis} 

\usepackage[notref, notcite]{showkeys} 

\usepackage{mathptmx} 

\begin{document}
\frontmatter 

\author{Amanda Cameron}
\title{Kinser inequalities and related matroids}
\subject{Mathematics}
\abstract{ Kinser developed a hierarchy of inequalities dealing with the dimensions of certain spaces constructed from a given quantity of subspaces. These inequalities can be applied to the rank function of a matroid, a geometric object concerned with dependencies of subsets of a ground set. A matroid which is representable by a matrix with entries from some finite field must satisfy each of the Kinser inequalities. We provide results on the matroids which satisfy each inequality and the structure of the hierarchy of such matroids.}

\ack{I would like to thank Dillon Mayhew for his advice and supervision.

Thanks also go to Michael Welsh for technical support and for proof-reading this thesis.}

\mscthesisonly 
\pagenumbering{alph}
\maketitle

\pagenumbering{roman}
\tableofcontents
\listoffigures 

\mainmatter 

\chapter{Introduction}
A fundamental question in matroid theory is whether it is possible to find a characterisation of the class of representable matroids. In particular, we wish to know whether this can be achieved with a finite number of axioms, by adding additional rank axioms to the exisiting three. This was first alluded to by Whitney \cite{whitney}, in the paper which initiated the area of matroid theory, and the problem remains open today. Ingleton \cite{ingleton} introduced one new axiom which a matroid must satisfy in order to be representable. 

\begin{dfn}
Let $M=(E,r)$ be a matroid. For subsets \vect{X}{4} of $E$, the \emph{Ingleton inequality} is:
\begin{align*}
r(X_3)+r(X_4)+r(X_1\cup X_2)+ r(X_1\cup X_3\cup X_4)+r(X_2\cup X_3\cup X_4)\\
\leq r(X_1\cup X_3)+r(X_1\cup X_4)+r(X_2\cup X_3) +r(X_2\cup X_4)+r(X_3\cup X_4)
\end{align*}
\end{dfn}

This new condition, while necessary, is not sufficient to characterise representability. For instance, the direct sum of the Fano and the non-Fano matroids satisfies the Ingleton condition but is not representable, as later proven in Lemma \ref{containment}.

\begin{figure}[h]
\centering
\subfloat[Fano matroid, $F_7$ ]{
\includegraphics[scale=0.6]{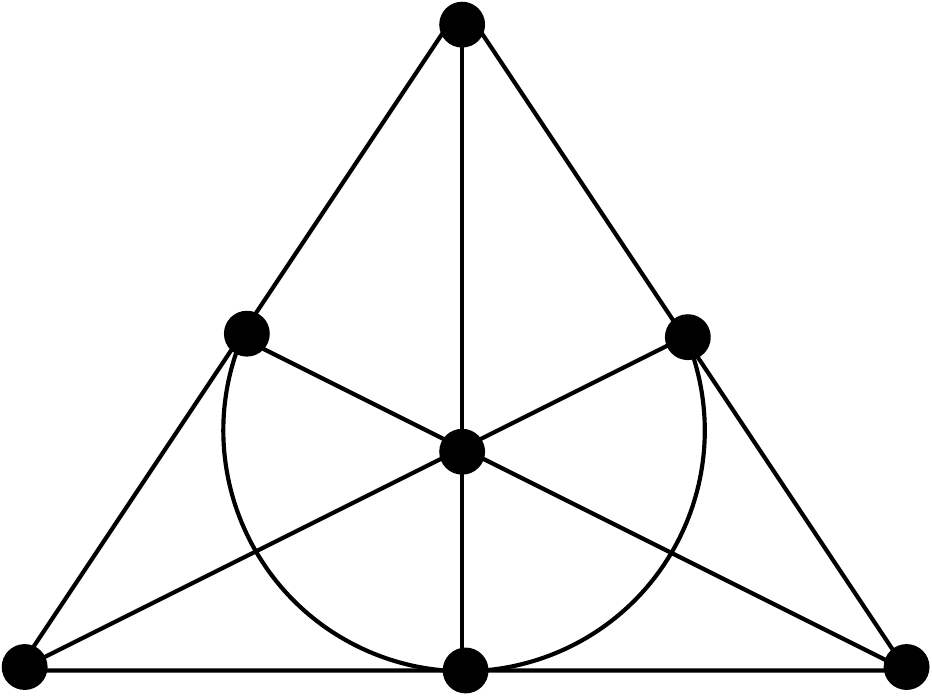}
}
\qquad
\subfloat[Non-Fano matroid, $F_7^-$]{
\includegraphics[scale=0.6]{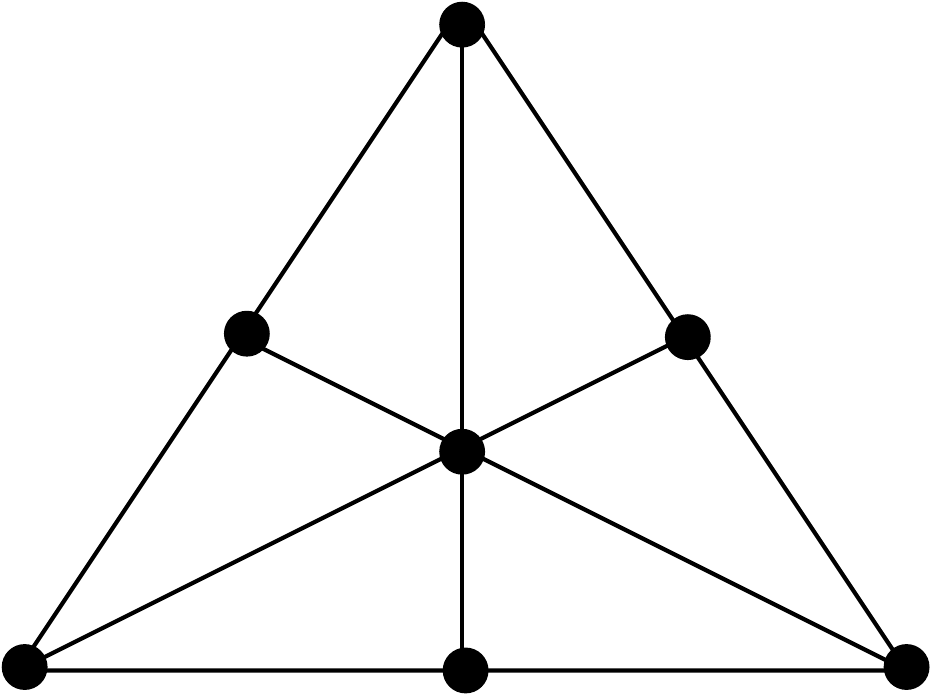}
}
\end{figure}

\pagebreak

Recently, Kinser \cite{kinser} introduced an infinite family of new representability conditions, the first of which is equivalent to the Ingleton condition. 

\begin{dfn}
Let $M$ be a matroid, and let \vect{X}{n} be any collection of subsets of $E(M)$. The \emph{$n$-th Kinser inequality}, where $n\geq 4$, is 
\begin{align*}
\sum_{i=3}^n r(X_i)+r(X_1\cup X_2)+r(X_1\cup X_3\cup X_n)+\sum_{i=4}^n r(X_2\cup X_{i-1}\cup X_i) \\
\leq r(X_1\cup X_3)+r(X_1\cup X_n)+\sum_{i=3}^n r(X_2\cup X_i)+\sum_{i=4}^n r(X_{i-1}\cup X_i)
\end{align*}
\end{dfn}

This hierarchy of inequalities is also not sufficient to guarantee representability of a matroid -- the direct sum of the Fano and the non-Fano is again a counter-example to this. Briefly putting aside the use of an infinite list of axioms, we have the following conjecture, which is due to Mayhew, Newman, and Whittle \cite{missingaxiom}. 

\begin{conj}
\label{insuff}
It is impossible to characterise the class of representable matroids with a finite number of rank axioms.
\end{conj}

Note that \cite{missingaxiom} is a response to an paper of V\'{a}mos' \cite{vamos} dealing with the same question. In this paper, V\'{a}mos introduced the following geometric construction, which we call a $V$-matroid:

\begin{dfn}
\label{vmatroid}
A \emph{V-matroid} consists of a (possibly infinite) set $E$ and a collection of finite subsets $\mathcal{I}\subseteq E$ such that:
\begin{itemize}
\item[{\rm I1}.] $\varnothing\in\mathcal{I}$
\item[{\rm I2}.] If $I\in\mathcal{I}$ and $J\subseteq I$, then $J\in\mathcal{I}$
\item[{\rm I3}.] If $I,J\in\mathcal{I}$ and $|I|=|J|+1$, there exists $x\in I-J$ such that $J\cup x\in\mathcal{I}$
\end{itemize}
\end{dfn}

Instead of using rank axioms, V\'{a}mos describes $V$-matroids with an infinite list of first-order axioms. In \cite{vamos}, V\'{a}mos proved that it is not possible to characterise representable $V$-matroids by adding a further first-order axiom to this infinite list. Note that a first-order axiom is not equivalent to a rank axiom. However, as every finite $V$-matroid is a matroid, Conjecture \ref{insuff} could be regarded as a strengthening of this result.

This thesis is dedicated to investigating the classes of matroids which satisfy each of the Kinser inequalities. We will cover invariant properties of the classes, such as being minor closed and direct sum closed, and, more importantly, we will provide results on how the classes interact with each other to form an infinite hierarchy. We will touch on the complexity of verifying a matroid satisfies a given Kinser inequality, which will show that gaining certain information on the Kinser classes, such as which classes are closed under duality, could involve a great amount of computational work.

This thesis will conclude by considering a question which arises naturally in conjunction with representability, that of excluded minors. The following theorem was proven by Mayhew, Newman, and Whittle in 2008 \cite{eminors}, settling a conjecture by J. Geelen.

\begin{thm}
For any infinite field $\mathbb{K}$ and any matroid $N$ representable over $\mathbb{K}$, there is an excluded minor for $\mathbb{K}$-representability that has $N$ as a minor.
\end{thm}

We will provide a strengthening of this result, showing that there is in fact an infinite number of such excluded minors. Specifically, we will show that for each layer of the Kinser class hierarchy, we can find an excluded minor which is contained inside that layer.

\chapter{Fundamentals}

To begin with, we will cover the basic concepts in matroid theory which will be used throughout this thesis. All of the following concepts and results can be found in \cite{Oxley}.

\begin{dfn}
\label{matroid}
A \emph{matroid} $M=(E,\mathcal{I})$ consists of a finite ground set $E$ and a collection of subsets $\mathcal{I}\subseteq E$ such that:
\begin{itemize}
\item[{\rm I1}.] $\varnothing\in\mathcal{I}$
\item[{\rm I2}.] If $I\in\mathcal{I}$ and $J\subseteq I$, then $J\in\mathcal{I}$
\item[{\rm I3}.] If $I,J\in\mathcal{I}$ and $|I|<|J|$, there exists $x\in J-I$ such that $I\cup x\in\mathcal{I}$
\end{itemize}
\end{dfn}

Any subset of $E$ contained in $\mathcal{I}$ is referred to as an \emph{independent set}, while any subset of $E$ which is not contained in $\mathcal{I}$ is called \emph{dependent}. A dependent set of cardinality one is called a \emph{loop}. We may use $E(M)$ in the place of $E$ at times, in order to make it clear which matroid is being referred to.

\begin{dfn}
\label{rank}
Take a matroid $M$ with ground set $E$. The \emph{rank} of a subset $X$ of $E$, denoted by $r(X)$, is the cardinality of the largest independent subset of $X$.
\end{dfn}

\begin{lemma}
\label{rankaxioms}
A matroid $M$ can be described by the ground set $E$ and a \emph{rank function} $r:\mathcal{P}(E)\rightarrow \mathbb{Z}\cup\{0\}$ such that, for $X,Y\in\mathcal{P}(E)$, the following conditions hold:
\begin{itemize}
\item[{\rm R1}.] $r(X)\leq |X|$
\item[{\rm R2}.] If $Y\subseteq X$, $r(Y)\leq r(X)$
\item[{\rm R3}.] $r(X\cup Y)+r(X\cap Y)\leq r(X)+r(Y)$
\end{itemize}
\end{lemma}

A set $X$ is independent if and only if $r(X)=|X|$. If $r(X)=r(M)$ we call $X$ a \emph{basis} of $M$. If a set contains a basis, it is called \emph{spanning}.

\section{Dependencies}

\begin{dfn}
\label{closure} The \emph{closure} of a set $X$ is denoted by $cl(X)$, where \\ $cl(X)=X\cup \{e\in E-X \ | \ r(X\cup e)=r(X)\}$
\end{dfn}

\begin{lemma}
\label{closureaxioms}
The closure function of a matroid satisfies the following conditions:
\begin{itemize}
\item[{\rm CL1}.] If $X\subseteq E$, then $X\subseteq cl(X)$.
\item[{\rm CL2}.] If $X\subseteq Y$, then $cl(X)\subseteq cl(Y)$.
\item[{\rm CL3}.] If $X\subseteq E$, then $cl(cl(X))=cl(X)$.
\item[{\rm CL4}.] If $X\subseteq E$ and $x\in E$, and $y\in cl(X\cup x)-cl(X)$, then $x\in cl(X\cup y)$.
\end{itemize}
\end{lemma}

The closure function corresponds to the notion of span of a vector space, and is sometimes referred to as such. A \emph{flat} is a set whose closure is equal to the set itself, i.e. $cl(X)=X$. If a flat has rank $r(M)-1$, it is called a \emph{hyperplane}. 

A minimally dependent set, i.e. a dependent set whose every proper subset is independent, is called a \emph{circuit}. A matroid can be described entirely by its set of circuits $\mathcal{C}$.

\begin{lemma}
\label{circuitaxioms} $(E,\mathcal{C})$ describes a matroid when the following conditions hold.
\begin{itemize}
\item[{\rm C1}.] $\varnothing\notin\mathcal{C}$
\item[{\rm C2}.] If $C,D\in\mathcal{C}$ and $C\subseteq D$, then $C=D$
\item[{\rm C3}.] If $C,D$ are distinct elements of $\mathcal{C}$ amd $e\in C\cup D$, then $(C\cup D)-e$ contains a circuit
\end{itemize}
\end{lemma}

A \emph{circuit-hyperplane} is a set which is both a circuit and a hyperplane.

\begin{dfn}
\label{relaxation}
Let $M$ be a matroid and let $H$ be a circuit-hyperplane of $M$. $H$ has rank equal to $r(M)-1$. We say that we \emph{relax} $H$ when we make it independent, i.e. $r(H)=r(M)$. When we reverse this operation, we say that we \emph{tighten} $H$.
\end{dfn}

\section{Representability}

\begin{dfn}
If $V$ is a set of vectors in a vector space, and for every subset $X$ of $V$, we define $r(X)$ to be the linear rank of $X$, then $(V,r)$ is a matroid, which we say is \emph{representable}. 
\end{dfn}
If these vectors come from a finite field $\mathbb{K}$, we say that $M$ is \emph{$\mathbb{K}$-representable.} 
\section{Minors}

\begin{dfn}
\label{delete}We can remove an element $e$ of a matroid $M=(E,r)$ by \emph{deleting} it. This yields a matroid $M\backslash e=(E-e,r_{M\backslash e})$, where $r_{M\backslash e}(X)=r_M(X)$ for all $X\subseteq E-\{e\}$.
\end{dfn}

\begin{dfn}
\label{contract}
We can also remove an element $e$ of a matroid $M=(E,r)$ by \emph{contracting} it. This gives a matroid $M/e=(E-e,r_{M/e})$ where $r_{M/e}(X)=r_M(X\cup \{e\})-r(\{e\})$ for all $X\subseteq E-\{e\}$.
\end{dfn}

Any matroid producted by a sequence of deletions and contractions is called a \emph{minor} of $M$.

We say that a class of matroids $\mathcal{M}$ is \emph{minor closed} if, for every matroid $M$ in $\mathcal{M}$, each of its minors is also in $\mathcal{M}$.
A matroid $M$ is an \emph{excluded minor} for a minor closed class of matroids $\mathcal{M}$ if $M\notin\mathcal{M}$ but deleting or contracting any element from $M$ produces a matroid in $\mathcal{M}$. A matroid $M$ is contained in $\mathcal{M}$ if and only if $M$ does not contain an excluded minor for $\mathcal{M}$.

\section{Duality}

\begin{dfn}
\label{dual}
From $M$ we can construct the \emph{dual matroid} $M^*$. This has ground set equal to the ground set $E$ of $M$, and the rank of any subset is found using the function $r^*(X)=|X|+r(E^*-X)-r(M)$.
\end{dfn}

A basis of $M^*$ is is called a \emph{cobasis} of $M$. Note that if $B$ is a basis of $M$, then $E-B$ is a cobasis of $M$. Similarly, the rank function, circuits and independent sets of $M^*$ are called the \emph{corank} function, \emph{cocircuits} and \emph{coindependent} sets of $M$. 

\begin{lemma}
\label{relax}\emph{(\cite[Proposition 2.1.7]{Oxley})}
Let $M$ be a matroid. Relax a circuit-hyperplane $H$ of $M$ to yield the matroid $M'$. Then $(M')^*$ is identical to the matroid yielded from $M^*$ by relaxing the circuit-hyperplane $E-H$ of $M^*$.
\end{lemma}

\begin{lemma}
\label{relax2}\emph{(\cite[Proposition 3.3.5]{Oxley})}
Let $H$ be a cicuit-hyperplane of a matroid $M$, and let $M'$ be the matriod obtained from $M$ by relaxing $H$.
\begin{itemize}
\item[{\rm i}.] When $e\in E(M)-H$, $M/e=M'/e$, and, unless $e$ is a coloop of $M$, $M'\backslash e$ is obtained from $M\backslash e$ by relaxing the circuit-hyperplane $H$ of $M\backslash e$.
\item[{\rm ii}.] Dually, when $f\in H$, $M\backslash f=M'\backslash f$ and, unless $f$ is a loop of $M$, $M'/f$ is obtained from $M/f$ by relaxing the circuit-hyperplane $X-f$ of $M/f$.
\end{itemize}
\end{lemma}

\section{Transversals}

\begin{dfn}
\label{transversal}
Let $S$ be any set. Take a family of subsets $\mathcal{A}=(A_1,\ldots,A_k)$ of $S$. A \emph{transversal} or \emph{system of distinct representatives} of $\mathcal{A}$ is a subset $\{s_1,\ldots,s_m\}$ of $S$ such that $s_i\in A_i$ for all $i\in\{1,\ldots,m\}$ and $s_1,\ldots,s_m$ are distinct. 
\end{dfn}

\begin{dfn}
\label{partial}
Let $S$ be any set. $X\subseteq S$ is a \emph{partial transversal} of a family of subsets $\mathcal{A}=(A_1,\ldots,A_j)$ if $X$ is a transversal of $(A_1,\ldots,A_k)$ for some $A_1,\ldots,A_k\subseteq S$.
\end{dfn}

\begin{lemma}
\label{transversals}
Let $\mathcal{A}=(A_1,\ldots,A_m)$ be a family of subsets of a set $S$. When $\mathcal{A}$ is a partition of $S$, the collection of partial transversals of $\mathcal{A}$ is the collection of independent sets of a matroid on $S$. This matroid is denoted by $M[\mathcal{A}]$.
\end{lemma}

If a matroid $M$ is isomorphic to $M[\mathcal{A}]$ for some family of subsets $\mathcal{A}$, we say that M is a \emph{transversal matroid} and that $\mathcal{A}$ is a \emph{presentation} of $M$. Every transversal matroid is representable over all sufficiently large fields, as proven in \cite[Proposition 11.2.16]{Oxley}.

A transversal matroid can be represented by a bipartite graph. Let \\$\mathcal{A}=(A_1,\ldots,A_m)$ be a family of subsets of $S$, and let $J=\{1,\ldots,m\}$. Construct the graph $G[\mathcal{A}]$ which has vertex set $S\cup J$ and edge set $\{xj \ | \ x\in S, j\in J, x\in A_j\}$. Recall that a matching of a graph is a collection of edges such that no two share a common endpoint. A subset $X$ is a partial transversal of $\mathcal{A}$ if and only if there is a matching in $G[\mathcal{A}]$ in which every edge has an endpoint in $X$, i.e. $X$ is matched into $J$.

\begin{dfn}
\label{truncation}
Take a matroid $M=(E,r)$ with independent sets $\mathcal{I}$. Let $J=\{I\in\mathcal{I} \ | \ |I|=r(M)\}$. The \emph{truncation} of $M$ is a matroid $T(M)=(E,r)$ with independent sets $\mathcal{I}-J$.
\end{dfn}

\chapter{Kinser Inequalities}

We will now introduce the Kinser inequalities, developed by Kinser in 2009 in \cite{kinser}. An example of a matroid which exemplifies inequality $n$ for all $n\geq 4$ will be described, as the V\'{a}mos matroid exemplifies the Ingleton inequality. If a single circuit-hyperplane of this matroid is relaxed, it no longer satisfies the inequality. These matroids will be used in further results in this thesis. We will show that the class of matroids which satisfy Kinser inequality $n$ for all $n\geq 4$ is minor-closed. These classes are also closed under direct sums.

\section{Inequalities}

\begin{dfn}
\label{inequality}
Let $M$ be a matroid, and let \vect{X}{n} be any collection of subsets of $E(M)$. The \emph{$n$-th Kinser inequality}, where $n\geq 4$, is 
\begin{align*}
\sum_{i=3}^n r(X_i)+r(X_1\cup X_2)+r(X_1\cup X_3\cup X_n)+\sum_{i=4}^n r(X_2\cup X_{i-1}\cup X_i) \\
\leq r(X_1\cup X_3)+r(X_1\cup X_n)+\sum_{i=3}^n r(X_2\cup X_i)+\sum_{i=4}^n r(X_{i-1}\cup X_i)
\end{align*}
\end{dfn}

Note that inequality $n$ has $2n-3$ terms on each side.

\begin{center}
\includegraphics[scale=0.6]{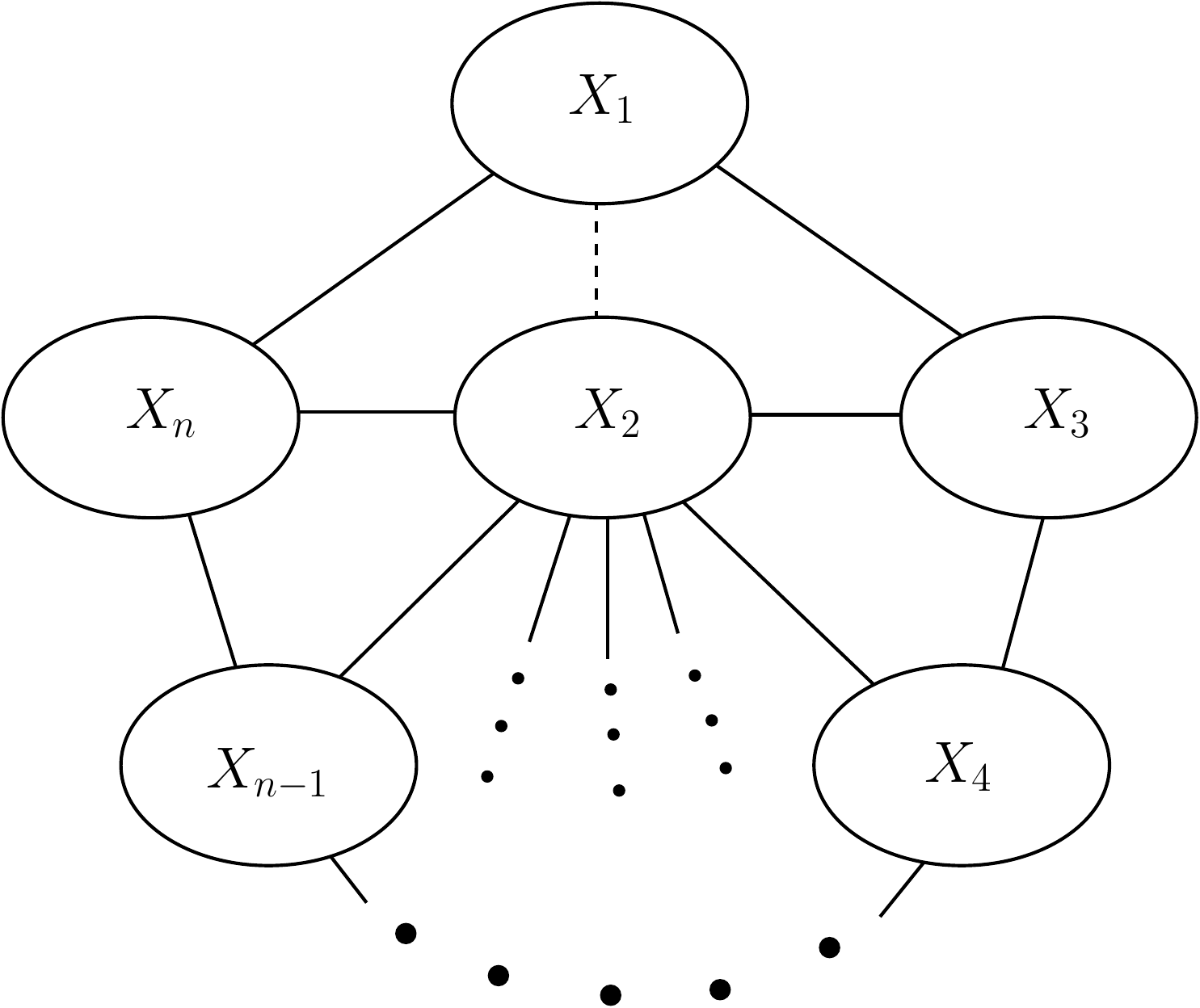}
\end{center}

The above diagram gives a representation of Kinser inequality $n$. The ovals represent the $n$ subsets of $E(M)$, and each edge aside from the dotted one between $X_1$ and $X_2$ represents a term on the right-hand side of the inequality. On the left-hand side, we have the singleton sets starting from $X_3$, the triple $X_1\cup X_3\cup X_n$ at the very top, the dashed $X_1\cup X_2$ edge, and every triangle of edges involving $X_2$, excluding those using $X_1$.

When $n=4$, this yields the \emph{Ingleton inequality} [1971], which holds for any four subspaces $X_1,\ldots,X_4$ of a vector space:
\begin{align*}
\text{dim}(V_3)+\text{dim}(V_4)+\text{dim}(V_1+ V_2)+ \text{dim}(V_1+V_3+V_4)+\text{dim}(V_2+V_3+V_4)\\
\leq \text{dim}(V_1+V_3)+\text{dim}(V_1+V_4)+\text{dim}(V_2+V_3) +\text{dim}(V_2+V_4)+\text{dim}(V_3+V_4)
\end{align*}
As a representable matroid can be embedded inside a vector space, this inequality clearly holds for such matroids. In fact, in order for a matroid to be representable, it must satisfy each Kinser inequality for all choices of families \vect{X}{n}.

Recall that if $X$ and $Y$ are subspaces of some vector space $\mathcal{V}$, then $$X+Y=\{ {\bf x}+{\bf y} \ | \ {\bf x}\in X, {\bf y}\in Y\}$$ is a subspace of $\mathcal{V}$ as well.

The following proof is adapted from that of \cite[Theorem 1]{kinser}, which was stated in terms of an arrangement of $n$ subspaces.

\begin{lemma}
\label{rep}
A representable matroid $M$ satisfies each Kinser inequality.
\end{lemma}

\begin{proof}
Let $M$ be a representable matroid, and let \vect{V}{n} be subsets of $E(M)$. Embed $M$ in the projective geometry $PG(r-1,\mathcal{K})$ and replace each $V_i$ with its closure, $\langle V_i\rangle$, in the projective geometry. Let $W=\langle V_3\rangle\cap\ldots\cap\langle V_n\rangle$. Let $|\langle V_i\rangle|$ denote the dimension of $\langle V_i\rangle$. Using submodularity, we have that
\begin{align*}
|\langle W\rangle+ \langle V_1\rangle|+|\langle W\rangle+ \langle V_2\rangle| & \geq |(\langle W\rangle+\langle V_1\rangle )\cap(\langle W\rangle+\langle V_2\rangle )|\\
& \qquad+|\langle W\rangle+ \langle V_1\rangle+ \langle V_2\rangle |\\
& \geq |\langle W\rangle+ (\langle V_1\rangle\cap \langle V_2\rangle ) |+|\langle W\rangle+ \langle V_1\rangle+ \langle V_2\rangle |\\
& \geq |\langle W\rangle |+|\langle W\rangle+ \langle V_1\rangle+ \langle V_2\rangle |
\end{align*}

Rearranging this, we get that
\begin{equation}
|\langle W\rangle+ \langle V_1\rangle+ \langle V_2 \rangle | -|\langle W\rangle+ \langle V_1\rangle | \leq |\langle W\rangle+ \langle V_2\rangle | -|\langle W\rangle |
\end{equation}
We will give a bound on each side of this inequality.

Note that $|\langle W\rangle+ \langle V_1\rangle+ \langle V_2\rangle |\geq |\langle V_1\rangle+ \langle V_2\rangle |$. \\
As $\langle W\rangle+ \langle V_1\rangle\subseteq (\langle V_1\rangle+ \langle V_3\rangle)\cap (\langle V_1\rangle+ \langle V_n\rangle )$, we have by submodularity that
$$|\langle W\rangle+ \langle V_1\rangle |\leq |\langle V_1\rangle + \langle V_3\rangle |+|\langle V_1\rangle+ \langle V_n\rangle |-|\langle V_1\rangle +\langle V_3\rangle + \langle V_n\rangle |$$
This gives us a lower bound for the left-hand side of (3.2.1):
\begin{align*}
|\langle V_1\rangle+ \langle V_2\rangle |-|\langle V_1\rangle + \langle V_3\rangle |-|\langle V_1\rangle +\langle V_n\rangle |+|\langle V_1\rangle+ \langle V_3\rangle + \langle V_n\rangle| \\
\leq |\langle W\rangle+ \langle V_1\rangle + \langle V_2\rangle |-|\langle W\rangle+ \langle V_1\rangle |
\end{align*}

Now take the right-hand side. We have that $$|\langle W\rangle+ \langle V_2\rangle |-|\langle W\rangle |= |\langle V_2\rangle |-|\langle V_2\rangle\cap\langle W\rangle |$$
Note that $V_2\supseteq V_2\cap V_3\supseteq\ldots\supseteq V_2\cap\ldots\cap V_n=V_2\cap W$.
This gives that
\begin{equation}
|\langle V_2\rangle |-|\langle V_2\rangle\cap\langle W\rangle |=\sum_{i=3}^n (|\langle V_2\rangle\cap\ldots\cap\langle V_{i-1}\rangle |-|\langle V_2\rangle\cap\ldots\cap\langle V_i\rangle |)
\end{equation}
For each summand above, we give an upper bound: for $3\leq i\leq n$, submodularity gives that
$$|\langle V_2\rangle\cap\ldots\cap\langle V_{i-1}\rangle |-|\langle V_2\rangle\cap\ldots\cap\langle V_i\rangle |= |\langle V_i\rangle+(\langle V_2\rangle\cap\ldots\cap\langle V_{i-1})\rangle |-|\langle V_i\rangle |$$
As $\langle V_i\rangle+(\langle V_2\rangle\cap\ldots\cap \langle V_{i-1}\rangle)\subseteq (\langle V_i\rangle+ \langle V_2\rangle)\cap (\langle V_i\rangle+\langle V_{i-1}\rangle)$, we have
\begin{align*}
|\langle V_i\rangle+ (\langle V_2\rangle\cap\ldots\cap\langle V_{i-1}\rangle ) |-|\langle V_i\rangle | &\leq |(\langle V_i\rangle+ \langle V_2\rangle )\cap (\langle V_i\rangle+ \langle V_{i-1})\rangle |-|\langle V_i\rangle | \\
&= |\langle V_i\rangle+ \langle V_2\rangle |+|\langle V_i\rangle+\langle V_{i-1}\rangle |\\
&\qquad -|\langle V_2\rangle+ \langle V_{i-1}\rangle+ \langle V_i\rangle |-|\langle V_i\rangle |
\end{align*}

Note that when $i=3$ this simplifies to
\begin{align*}
|\langle V_3\rangle+ \langle V_2\rangle |-|\langle V_3\rangle | & \leq
|\langle V_3\rangle+ \langle V_2\rangle |+|\langle V_3\rangle+\langle V_2\rangle |-|\langle V_2\rangle+ \langle V_3\rangle |-|\langle V_3\rangle |\\
& = |\langle V_2\rangle+ \langle V_3\rangle |-|\langle V_3\rangle |
\end{align*}

Plugging this into (3.2.2) then (3.2.1) gives, after rearranging,
\begin{align*}
\sum_{i=3}^n |\langle V_i\rangle |+|\langle V_1\rangle+ \langle V_2\rangle |+|\langle V_1\rangle+ \langle V_3\rangle+ \langle V_n\rangle |+\sum_{i=4}^n |\langle V_2\rangle + \langle V_{i-1}\rangle+\langle X_i\rangle | \\
\leq |\langle V_1\rangle+\langle V_3\rangle |+|\langle V_1\rangle+\langle V_n\rangle |+\sum_{i=3}^n |\langle V_2\rangle+\langle V_i\rangle |+\sum_{i=4}^n |\langle V_{i-1}\rangle+\langle V_i\rangle |
\end{align*}
Note that $|\langle V_i\rangle|=r(V_i)$. In order to show that inequality $n$ holds, we must show that $|\langle V_i\rangle+\langle V_j\rangle|=r(V_i\cup V_j)$. We have that $$r(V_i\cup V_j)=|\langle V_i\cup V_j\rangle |$$ We will show that this is equal to $|\langle V_i\rangle+\langle V_j\rangle |$. 

Let $x\in\langle V_i\rangle+\langle V_j\rangle$. This means that $x=x_1+x_2$ where $x_1\in\langle V_i\rangle$ and $x_2\in\langle V_j\rangle$. We have that $x_1\in\langle V_i\cup V_j\rangle$ and $x_2\in\langle V_i\cup V_j\rangle$, so $x\in\langle V_i\cup V_j\rangle$. Now take $x\in\langle V_i\cup V_j\rangle$. We can write $x$ as a linear combination of elements $S_i$ from $V_i$ and elements $S_j$ from $V_j$. We have that $S_i\subseteq \langle V_i\rangle$ and that $S_j\subseteq \langle V_j\rangle$, so $x\in\langle V_i\rangle +\langle V_j\rangle$. Thus $|\langle V_i\cup V_j\rangle |=|\langle V_i\rangle+\langle V_j\rangle |$. We can therefore replace every term $|\langle V_i\rangle+\langle V_j\rangle |$ with $r(V_i\cup V_j)$. Similarly, $|\langle V_i\cup V_j\cup V_k\rangle |=|\langle V_i\rangle+\langle V_j\rangle+\langle V_k\rangle |$. Making all such replacements yields inequality $n$.

\end{proof}
We say that a \emph{bad family} for a matroid $M$, relative to $n$, is a family of subsets \vect{X}{n} which does not satisfy Kinser inequality $n$. 

We can also represent an inequality as applied to a specific matroid with a graph. Let \vect{X}{n} be a family of subsets of a matroid $M$. Take a graph $G$ on vertices $V=\{X_1,\ldots,X_{n}\}$ with adjacency relation $a$ such that $$a(X_i)=\{X_j \ | \ X_i\cup X_j \ \mathrm{is} \ \mathrm{a} \ \mathrm{term} \ \mathrm{on} \ \mathrm{the} \ \text{right-hand side} \ \mathrm{of} \ \mathrm{inequality} \ n\}$$ In other words, two vertices are joined by an edge if the union of the two vertices is a term in inequality $n$. Recall that when $G[V,E]$ is any graph with vertex set $V$ and edge set $E$, an induced subgraph $G[E']$, has edge set $E'$ and vertex set equal to the vertices incident with edges in $E'$. We will use this construction to show that certain subgraph structures cannot exist, when attempting to find a bad family in a matroid.

\begin{dfn}
\label{classs}
\emph{ Kinser class} $n$, denoted by $\mathcal{K}_n$, is the set of matroids which satisfy Kinser inequality $n$ for all families of subsets \vect{X}{n} of the ground set. We define $\mathcal{K}_\infty=\bigcap_{i\geq 4}\mathcal{K}_i$. 
\end{dfn}

A matroid $M$ has a bad family relative to $n$ if and only if $M\notin\mathcal{K}_n$.

\begin{dfn}
\label{dual class}
The \emph{dual Kinser class} $n$ is $\mathcal{K}_n^*=\{ M^* \ | \ M\in\mathcal{K}_n\}$
\end{dfn}

\section{Kinser matroids}

Next we will construct a class of matroids called \emph{Kinser matroids}, relating to the Kinser inequalities as the V\'{a}mos matroid relates to the Ingleton inequality. In fact, the V\'{a}mos matroid is the fourth Kinser matroid after a circuit-hyperplane having been relaxed. The rank $r$ Kinser matroid, for $r\geq 4$, is denoted by \kin{r}, and has a ground set of size $r^2-3r+4$.

First, we will define a rank $r+1$ tranversal matroid, $M_{r+1}$. Let $\mathcal{A}=(A_1,\ldots,A_{r-1},A,A')$. Also let $V_1,\ldots, V_r$ be pairwise disjoint sets such that $$|V_1|=|V_3|=\cdots=|V_r|=r-2$$ and $V_2=\{e,f\}$. The ground set of $M_{r+1}$ is $V_1\cup\cdots\cup V_r$. Let $A=E(M_{r+1})$ and let $A'=V_2$. Let 
\begin{align*}
A_1 &= (V_1\cup V_3\cup\cdots\cup V_r)-(V_1\cup V_r) \\
A_3 &= (V_1\cup V_3\cup\cdots\cup V_r)-(V_1\cup V_3)
\end{align*}
For $i\in\{4,\ldots,r\}$, let
$$A_i=(V_1\cup V_3\cup\cdots\cup V_r)-(V_{i-1}\cup V_i)$$
Then $M_{r+1}$ is the tranversal matroid $M[\mathcal{A}]$. 

Note $\{e,f\}$ is a series pair in $M_{r+1}$.

Define \kin{r} to be the truncation of $M_{r+1}$. 

\pagebreak

\begin{figure}[ht]
\begin{center}
\begin{tikzpicture}[thin,line join=round]
\coordinate (o) at (0,0);
\coordinate[label=$V_1$] (v1) at (150:3);
\coordinate[label=$V_2$] (v2) at (88:3);
\coordinate[label=$V_3$] (v3) at (5:3);
\coordinate[label=0:$V_4$] (v4) at (-105:2);
\coordinate (a) at ($(o)!0.67!(v1)$);
\coordinate (b) at ($(o)!0.33!(v1)$);
\coordinate (c) at ($(o)!0.67!(v4)$);
\coordinate (d) at ($(o)!0.33!(v4)$);
\coordinate[label=0:$e$] (e) at ($(o)!0.67!(v2)$);
\coordinate[label=0:$f$] (f) at ($(o)!0.33!(v2)$);
\coordinate (g) at ($(o)!0.67!(v3)$);
\coordinate (h) at ($(o)!0.33!(v3)$);
\draw (v1) -- (o) -- (v3);
\draw (v2) -- (o) -- (v4);
\foreach \p in {a, b, c, d, e, f, g, h} \fill[black] (\p) circle (3pt);
\end{tikzpicture}
\caption{\kin{4}}\label{drunk-plus}
\end{center}
\end{figure}

\begin{figure}[ht]
\begin{center}
\begin{tikzpicture}[thin,line join=round]
\coordinate (a) at (0,0);
\coordinate (b) at (5,0);
\coordinate (c) at (0,3);
\coordinate (d) at (5,3);
\coordinate (o) at (intersection of c--b and a--d);
\coordinate[label=$V_2$] (v2) at ($(o) + (1.5,2.5)$);
\coordinate[label=0:$e$] (e) at ($(o)!0.8!(v2)$);
\coordinate[label=0:$f$] (f) at ($(o)!0.45!(v2)$);
\coordinate (g) at ($(o) + (-0.3,0.8)$);
\coordinate (h) at ($(g) + (-0.4, 0.4)$);
\coordinate (i) at ($(g) + (0.4, 0.4)$);
\coordinate (j) at ($(o) + (0,-0.8)$);
\coordinate (k) at ($(j) + (-0.7, -0.4)$);
\coordinate (l) at ($(j) + (0.7, -0.4)$);
\coordinate (m) at ($(o) + (1.2, 0)$);
\coordinate (n) at ($(m) + (0.7, 0.4)$);
\coordinate (p) at ($(m) + (0.7, -0.4)$);
\coordinate (q) at ($(o) + (-1.2, 0)$);
\coordinate (r) at ($(q) + (-0.7, 0.4)$);
\coordinate (s) at ($(q) + (-0.7, -0.4)$);
\draw (a) -- (b) -- (d) -- (c) -- (a) -- (d);
\draw (c) -- (b);
\draw (o) -- (v2);
\foreach \p in {e, f, g, h, i, j, k, l, m, n, p, q, r, s} \fill[black] (\p) circle (3pt);
\node at ($(c)!0.5!(d) + (0, 0.3)$) {$V_1$};
\node at ($(d)!0.5!(b) + (0.3, 0)$) {$V_3$};
\node at ($(b)!0.5!(a) + (0, -0.3)$) {$V_4$};
\node at ($(c)!0.5!(a) + (-0.3, 0)$) {$V_5$};
\end{tikzpicture}
\caption{\kin{5}}\label{envelope}
\end{center}
\end{figure}

The following result is Proposition 4.3 of \cite{missingaxiom}

\begin{lemma}
Let $\mathbb{K}$ be an infinite field. Then \kin{r} is $\mathbb{K}$-representable for any $r\geq 4$.
\end{lemma}

As $M_{r+1}$ is a tranversal matroid, is it representable over every infinite field by \cite[Proposition 11.2.16]{Oxley}. We obtain \kin{r} by truncating $M_{r+1}$. This is equivalent to freely adding an element to the ground set of $M_{r+1}$ and then contracting it. As the class of representable matroids is closed under free extensions, \kin{r} is also representable.

The following result is proven in \cite[Proposition 4.4]{missingaxiom}.

\begin{lemma}
Let $r\geq 4$ be an integer. Then $V_2\cup V_i$ is a circuit-hyperplane of \kin{r} for any $i\in\{1,3,\ldots,r\}$.
\end{lemma}

Define $\kin{r}^-$ to be the matroid obtained from \kin{r} by relaxing the circuit-hyperplane $V_1\cup V_2$. Also define $\kin{r}_i^=$ to be the matroid obtained from \kin{r} by relaxing the circuit-hyperplanes $V_1\cup V_2$ and $V_2\cup V_i$, for some $i\in\{3,\ldots,r\}$.

The next two results are Proposition 4.5 and Lemma 4.6 of \cite{missingaxiom}.

\begin{lemma}
\label{kin-}
Let $r\geq 4$. The matroid $\kin{r}^-$ is not in $\mathcal{K}_r$, and is therefore not representable over any field.
\end{lemma}

The family of subsets \vect{V}{n} in $\kin{r}^-$ is a bad family relative to $r$, as will be proven in Lemma \ref{gaps}.

\begin{lemma}
\label{kin=}
Let $r\geq 4$ and let $\mathbb{K}$ be an infinite field. The matroid $\kin{r}_i^=$ is $\mathbb{K}$-representable.
\end{lemma}

\section{Kinser classes}

\begin{lemma}
\label{minors}
$\mathcal{K}_n$ is minor-closed for all $n\geq 4$.
\end{lemma}

\begin{proof}
Take some $M\in \mathcal{K}_n$ and $e\in E(M)$ such that $M/e\notin \mathcal{K}_n$. Assume $e$ is not a loop. Assume \vect{X}{n} is a bad family in $M/e$ and let ${X_i}'=X_i\cup e$ for all $i$. Recall that $r_{M/x}(X)=r_M(X\cup x)-r_M(x)$. Thus $r_{M/e}(X_i)=r_{M}(X_i\cup e)-r_M(e)$. When $e$ is not a loop, we have that $r_M({X_i}')=r_{M/e}(X_i)+1$ for all $i$, and $r_M({X_i}'\cup {X_j}'\cup {X_k}')=r_{M/e}(X_i\cup X_j\cup X_k)+1$. Now evaluate inequality $n$ for $X_1',\ldots,X_n'$ in $M$.
\begin{multline*}\sum_{i=3}^n r_M(X_i')+r_M(X_1'\cup X_2')+r_M(X_1'\cup X_3'\cup X_n')+ \sum_{i=4}^n r_M(X_2'\cup X_{i-1}'\cup X_i')\\
\leq r_M(X_1'\cup X_3')+r_M(X_1'\cup X_n')+\sum_{i=3}^n r_M(X_2'\cup X_i')+ 
\sum_{i=4}^n r_M(X_{i-1}'\cup X_i')\end{multline*}
Using the rank equalities calculated above, this is equivalent to
\begin{multline*}\sum_{i=3}^n (r_{M/x}(X_i)+1)+r_{M/x}(X_1\cup X_2)+1\\
+r_{M/x}(X_1\cup X_3\cup X_n)+1+ \sum_{i=4}^n (r_{M/x}(X_2\cup X_{i-1}\cup X_i)+1)\\
\leq r_{M/x}(X_1\cup X_3)+1+r_{M/x}(X_1\cup X_n)+1\\
+\sum_{i=3}^n (r_{M/x}(X_2\cup X_i)+1)+ 
\sum_{i=4}^n (r_{M/x}(X_{i-1}\cup X_i)+1)\end{multline*}
All the constant terms cancel out, leaving inequality $n$ as applied to \vect{X}{n} in $M/e$, contradicting \vect{X}{n} being a bad family in $M/e$. Thus there is no $e$ such that $M/e\notin\mathcal{K}_n$.

Now consider $M\backslash e$. Assume that $M\backslash e$ has a bad family \vect{X}{n}. These subsets are also subsets of $M$ and their rank is unchanged in $M$, so they must form a bad family in $M$ as well, contradicting $M\in\mathcal{K}_n$. 
\end{proof}

\begin{lemma}
\label{indp}
Suppose \vect{X}{n} is a bad family for Kinser inequality $n$. We can assume that each set $X_i$ is independent.
\end{lemma}

\begin{proof}
For all $i\in\{1,\ldots,n\}$, let $I_k\subseteq X_j$ be a basis of $X_j$. We will show that we can replace each set $X_j$ with its basis $I_j$. We have that $r(X_j)=r(I_j)$. Now consider $r(X_j\cup X_k)$. Clearly $r(I_j\cup I_k)\leq r(X_j\cup X_k)$. If $x\in X_j\cup X_k$, then $x$ is either in $X_j$ or it is in $X_j$, so $x\in cl_M(I_j)$ or $x\in cl_M(I_k)$. In either case, $x\in cl_M(I_j\cup I_k)$, so $X_j\cup X_k\subseteq cl_M(I_j\cup I_k)$. Thus
\begin{align*}
r(X_j\cup X_k) &\leq r(cl_M(I_j\cup I_k)) \\
&= r(I_j\cup I_k)
\end{align*}

Thus $r(I_j\cup I_k)\leq r(X_j\cup X_k)\leq r(I_j\cup I_k)$, so $r(X_j\cup X_k)=r(I_j\cup I_k)$. Similarly, $r(X_j\cup X_k\cup X_l)=r(I_j\cup I_k\cup I_l)$. Thus \vect{I}{n} is a bad family for Kinser inequality $n$.
\end{proof}

\begin{lemma}
\label{flats}
Suppose \vect{X}{n} is a bad family for Kinser inequality $n$. We can assume that each set $X_i$ is a flat.
\end{lemma}

\begin{proof}
Recall that a flat is a set $X$ such that $cl(X)=X$. We simply replace each $X_i$ with $cl(X_i)$. First note that $r(cl(X_i))=r(X_i)$. Now consider $r(X_i\cup X_j)=r(cl(X_i\cup X_j))$. We need to show that this is equal to $r(cl(X_i)\cup cl(X_j))$. As $X_i\subseteq cl(X_i)$, we must have that $$r(X_i\cup X_j)\leq r(cl(X_i)\cup cl(X_j))$$ Now note that if $e\in cl(X_i)$, then $e\in cl(X_i\cup X_j)$. This implies that $cl(X_i)\subseteq cl(X_i\cup X_j)$ Likewise, every element in the closure of $X_j$ is also in the closure of $X_i\cup X_j$, so $cl(X_j)\subseteq cl(X_i\cup X_j)$. We thus have that $cl(X_i)\cup cl(X_j)\subseteq cl(X_i\cup X_j)$, so
\begin{align*}
r(cl(X_i)\cup cl(X_j)) &\leq r(cl(X_i\cup X_j))\\
&= r(X_i\cup X_j)
\end{align*}

We thus have that
$$r(cl(X_i)\cup cl(X_j))\leq r(X_i\cup X_j) \leq r(cl(X_i)\cup cl(X_j))$$
so $r(X_i\cup X_j)=r(cl(X_i)\cup cl(X_j))$, and so $cl(X_i),\ldots,cl(X_n)$ is a bad family for inequality $n$ as well.
\end{proof}

\section{Sums}

\begin{dfn}
\label{direct sum}
Let $M=(E,r)$ and $M'=(E',r')$ where $E\cap E'=\varnothing$. The \emph{direct sum} of these matroids is denoted by $M\oplus M'$, and has ground set $E\cup E'$ and rank of $X\subseteq E\cup E'$ equal to $r(X\cap E)+r'(X\cap E')$.
\end{dfn}

\begin{lemma} 
\label{ds closed}
$\mathcal{K}_n$ is closed under direct sum for all $n$.
\end{lemma}

\begin{proof}
Take two matroid $M=(E,r)$ and $M'=(E',r')$ which are contained in $\mathcal{K}_n$. Take the direct sum $M\oplus M'$. We wish to show that for any family \vect{X}{n} of $E\cup E'$ the following inequality holds:
\begin{align*}
\sum_{i=3}^n r_{M\oplus M'}(X_i)+r_{M\oplus M'}(X_1\cup X_2)+r_{M\oplus M'}(X_1\cup X_3\cup X_n)+\sum_{i=4}^n r_{M\oplus M'}(X_2\cup X_{i-1}\cup X_i) \\
\leq r_{M\oplus M'}(X_1\cup X_3)+r_{M\oplus M'}(X_1\cup X_n)+\sum_{i=3}^n r_{M\oplus M'}(X_2\cup X_i)+\sum_{i=4}^n r_{M\oplus M'}(X_{i-1}\cup X_i)
\end{align*}
This is equivalent to
\begin{equation}
\begin{split}
& \sum_{i=3}^n r(X_i\cap E)+r((X_1\cup X_2)\cap E)+r((X_1\cup X_3\cup X_n)\cap E)\\
&\qquad+\sum_{i=4}^n r((X_2\cup X_{i-1}\cup X_i)\cap E) + \sum_{i=3}^n r'(X_i\cap E')+r'((X_1\cup X_2)\cap E')\\
&\qquad+r'((X_1\cup X_3\cup X_n)\cap E')+\sum_{i=4}^n r'((X_2\cup X_{i-1}\cup X_i)\cap E')\\
\leq & \ r((X_1\cup X_3)\cap E)+r((X_1\cup X_n)\cap E)+\sum_{i=3}^n r((X_2\cup X_i)\cap E)\\
&\qquad+\sum_{i=4}^n r((X_{i-1}\cup X_i)\cap E)+r'((X_1\cup X_3)\cap E')+r'((X_1\cup X_n)\cap E')\\
&\qquad+\sum_{i=3}^n r'((X_2\cup X_i)\cap E')+\sum_{i=4}^n r'((X_{i-1}\cup X_i)\cap E')
\end{split}
\end{equation}
As $M\in\mathcal{K}_n$, we have that
\begin{equation*}
\begin{split}
& \sum_{i=3}^n r(X_i\cap E)+r((X_1\cup X_2)\cap E)\\
& \qquad+r((X_1\cup X_3\cup X_n)\cap E)+\sum_{i=4}^n r((X_2\cup X_{i-1}\cup X_i)\cap E) \\ 
& \leq r((X_1\cup X_3)\cap E)+r((X_1\cup X_n)\cap E)\\
&\qquad+\sum_{i=3}^n r((X_2\cup X_i)\cap E)+\sum_{i=4}^n r((X_{i-1}\cup X_i)\cap E)
\end{split}
\end{equation*}
As $M'\in\mathcal{K}_n$, we have that 
\begin{equation*}
\begin{split}
& r'(X_i\cap E')+r'((X_1\cup X_2)\cap E')\\
& \qquad+r'((X_1\cup X_3\cup X_n)\cap E')+\sum_{i=4}^n r'((X_2\cup X_{i-1}\cup X_i)\cap E') \\
& \leq r'((X_1\cup X_3)\cap E')+r'((X_1\cup X_n)\cap E')\\
&\qquad+\sum_{i=3}^n r'((X_2\cup X_i)\cap E')+\sum_{i=4}^n r'((X_{i-1}\cup X_i)\cap E')
\end{split}
\end{equation*}
The values of the terms on the left of inequality (3.0.21.1) are thus bounded by the terms on the right-hand side, and so the inequality holds. 
\end{proof}

\chapter{Kinser Hierarchy}

In this chapter we will investigate how the Kinser classes interact with each other. We will first show that representable matroids are properly contained inside every Kinser class. Next we will show that the classes form a descending chain, and show that the relaxed Kinser matroid of rank $n$ is contained inside $\mathcal{K}_{n-1}$ but not $\mathcal{K}_{n}$. Next we will consider the issue of duality, and prove that the class of matroids which satisfy Kinser inequality $4$ is dual closed. The class of matroids which satisfy Kinser inequality $5$ is, in contrast, not dual closed. The proof of this is given in the next chapter.

\begin{lemma}
\label{containment}
The class of representable matroids is properly contained in $\mathcal{K}_{\infty}$
\end{lemma}

\begin{proof}
Note that as a consequence of Lemma \ref{rep}, the class of representable matroids are contained inside every Kinser class, and so is a subset of $\mathcal{K}_\infty$. We will now show the class of representable matroids forms a proper subset of $\mathcal{K}_\infty$. Recall that $\mathcal{K}_n$ is closed under direct sum for all $n$. Take two matroids in $\mathcal{K}_\infty$. As these matroid is in the intersection of every Kinser class, their direct sum is also contained in every Kinser class, and thus inside $\mathcal{K}_\infty$. Thus $\mathcal{K}_\infty$ is also closed under direct sums.

Define $F_7$ to be the matroid represented over GF(2) by
$$\bordermatrix{\text{} \cr
& 1 & 0 & 0 & 1 & 1 & 0 & 1 \cr
& 0 & 1 & 0 & 1 & 0 & 1 & 1 \cr
& 0 & 0 & 1 & 0 & 1 & 1 & 1 } $$ 
Define $F_7^-$ to be the matroid represented by the same matrix, but over GF(3). By \cite[Proposition 6.4.8]{Oxley}, $F_7$ can be represented over a field only if it has characteristic $2$, while $F_7^-$ can be represented over a field only if it has characteristic different from $2$. Therefore $F_7\oplus F_7^-$ is not representable. However, it is contained in $K_\infty$, since $F_7$ and $F_7^-$ are both representable, and hence in $\mathcal{K}_{\infty}$. 
\end{proof}

\begin{lemma}
\label{supseteq}
$\mathcal{K}_n\supseteq\mathcal{K}_{n+1}$.
\end{lemma}

\begin{proof}
Assume inequality $n+1$ holds for the matroid $M$. Let $X_1,\ldots,X_n$ be arbitrary subsets of $E(M)$. We show inequality $n$ holds for \vect{X}{n}. Let $X_{n+1}=X_n$. We have that
\begin{align*}\sum_{i=3}^{n+1} r(X_i)+r(X_1\cup X_2)+r(X_1\cup X_3\cup X_{n+1})+ \sum_{i=4}^{n+1} r(X_2\cup X_{i-1}\cup X_i) \\
\leq r(X_1\cup X_3)+r(X_1\cup X_{n+1})+\sum_{i=3}^{n+1} r(X_2\cup X_i)+ 
\sum_{i=4}^{n+1} r(X_{i-1}\cup X_i)\end{align*}
Bringing out the last term of each sum, 
\begin{equation*}
\begin{split}
& \sum_{i=3}^n r(X_i)+r(X_{n+1})+r(X_1\cup X_2)+r(X_1\cup X_3\cup X_n)\\
&\qquad+ \sum_{i=4}^n r(X_2\cup X_{i-1}\cup X_i)+r(X_2\cup X_{n+1})\\
& \leq r(X_1\cup X_3)+r(X_1\cup X_n)+\sum_{i=3}^n r(X_2\cup X_i) \\
&\qquad+r(X_2\cup X_{n+1})+ \sum_{i=4}^n r(X_{i-1}\cup X_i)+r(X_n\cup X_{n+1})
\end{split}
\end{equation*}
Now bringing these terms to the start of each side of the inequality and using $X_{n+1}=X_n$, we have that this is the same as
\begin{equation*} 
\begin{split}
& r(X_n)+r(X_2\cup X_n)+\sum_{i=3}^n r(X_i)+r(X_1\cup X_2)\\
&\qquad +r(X_1\cup X_3\cup X_n)+ \sum_{i=4}^n r(X_2\cup X_{i-1}\cup X_i)\\
& \leq r(X_2\cup X_n)+r(X_n)+r(X_1\cup X_3)+r(X_1\cup X_n)\\
&\qquad+\sum_{i=3}^n r(X_2\cup X_i)+ 
\sum_{i=4}^n r(X_{i-1}\cup X_i)
\end{split}
\end{equation*}
The two terms at the start of each side of the inequality cancel out, leaving inequality $n$:
\begin{align*}\sum_{i=3}^n r(X_i)+r(X_1\cup X_2)+r(X_1\cup X_3\cup X_n)+ \sum_{i=4}^n r(X_2\cup X_{i-1}\cup X_i)\\
\leq r(X_1\cup X_3)+r(X_1\cup X_n)+\sum_{i=3}^n r(X_2\cup X_i)+ 
\sum_{i=4}^n r(X_{i-1}\cup X_i).\end{align*}
\end{proof}

We now have the following diagram of the Kinser hierarchy.

\begin{figure}[ht]
\centering
\includegraphics[scale=0.6]{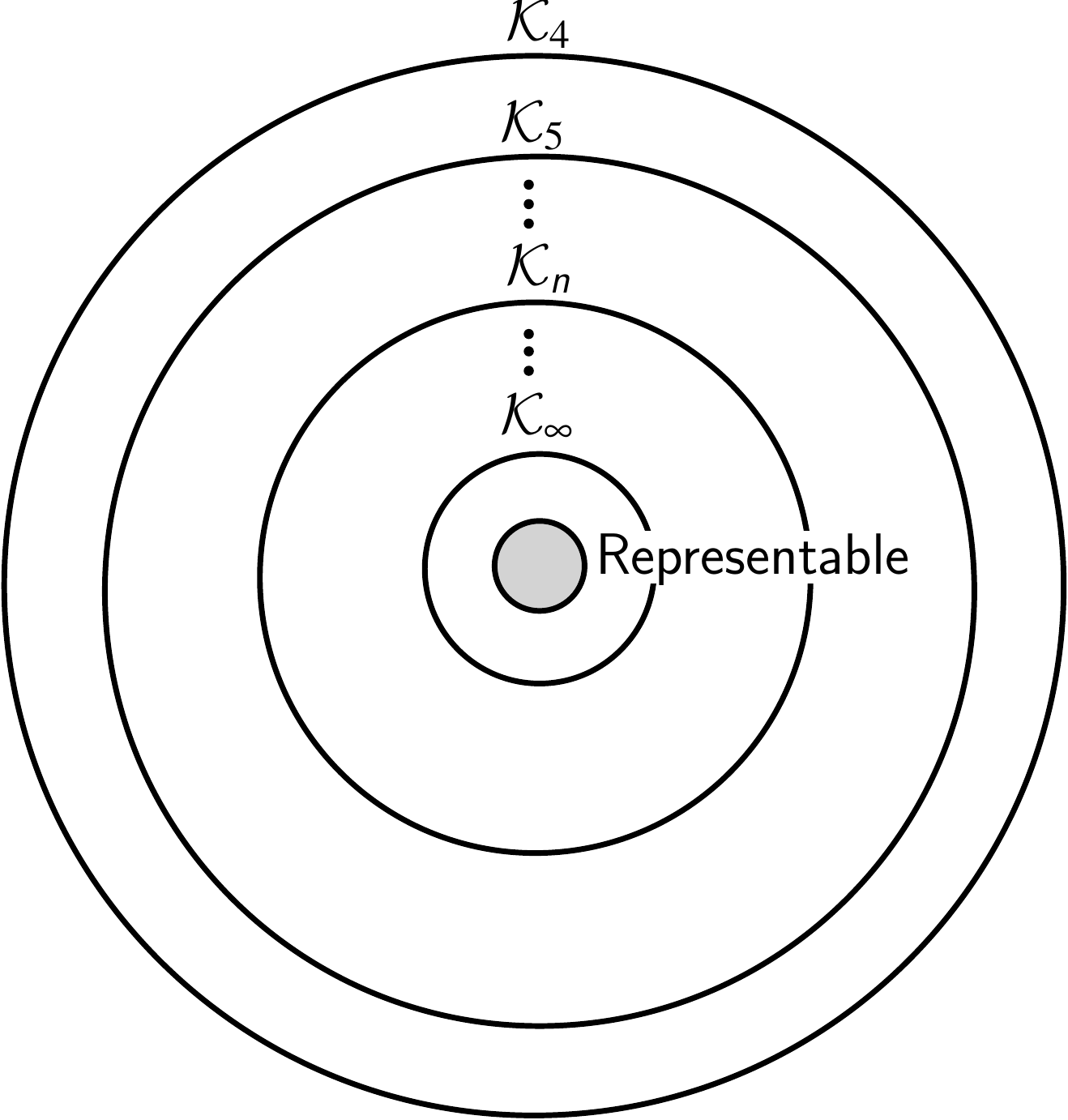}
\caption{Kinser classes}
\end{figure}

Next, we will give an example of a matroid which lies in the gap between two of these classes.

\begin{lemma}
\label{gaps} For $n\geq 5$, \kin{n}$^-\in\mathcal{K}_{n-1}-\mathcal{K}_n$.
\end{lemma}

\begin{proof}
Take $\kin{n}^-$. We will first show that \vect{V}{n} as in the definition of \kin{n} forms a bad family for inequality $n$ -- that is, 

\begin{align*}
\sum_{i=3}^n r(V_i)+r(V_1\cup V_2)+r(V_1\cup V_3\cup V_n)+\sum_{i=4}^n r(V_2\cup V_{i-1}\cup V_i) \\
\nleq r(V_1\cup V_3)+r(V_1\cup V_n)+\sum_{i=3}^n r(V_2\cup V_i)+\sum_{i=4}^n r(V_{i-1}\cup V_i)
\end{align*}

We sketch the proof of \cite[Proposition 4.5]{missingaxiom}.

Recall that $V_1\cup V_2$ is a relaxed circuit-hyperplane, while $V_2\cup V_i$ is a circuit-hyperplane for all $i\in\{3,\ldots,n\}$. Also, $V_i\cup V_{i+1}$ is a hyperplane for all $i\geq 4$, as is $V_1\cup V_3$, while $V_i\cup V_k$ is spanning for inconsecutive $i,k$. Each $V_i$ is indepedent, while the union of any three $V_i$ is spanning. Substituting these results into inequality $n$ gives
$$
\begin{array}{rrcl}
& \sum_{i=3}^n (n-2)+n+n+\sum_{i=4}^n n & \not\leq & (n-1)+n(n-1)\\
& & & \quad+\sum_{i=3}^n (n-1)+\sum_{i=4}^n (n-1)\\
\Leftrightarrow & (n-2)(n-2)+2n+(n-3)n & \not\leq & (2n-3)(n-1)\\
\Leftrightarrow & 2n^2-5n+4 & \not\leq & 2n^2-5n+3
\end{array}
$$

Thus \kin{n}$^-\notin\mathcal{K}_n$. 

Assume \kin{n}$^-$ has a family $X_1,\ldots, X_{n-1}$ which violates inequality $n-1$. Recall that if we take a second hyperplane of the form $V_2\cup V_j$, for an arbitrary $j\in\{3,\ldots,n\}$ and relax it, this yields a representable matroid, $\kin{n}_j^=$, by Lemma \ref{kin=}. As relaxing a circuit-hyperplane causes that subset to become independent, the rank of $V_2\cup V_j$ increases by one. The rank of all other subsets are unchanged. Suppose $V_2\cup V_j$ was not a term in inequality $n-1$ as applied to \vect{X}{n-1}. Relaxing $V_2\cup V_j$ would then have no effect on the value of the inequality, giving that \vect{X}{n-1} is a bad family in \kin{n}$^=$. This contradicts the matroid being representable. Suppose that $V_2\cup V_j$ is a term on the left-hand side of inequality $n-1$. Then, when we relax $V_2\cup V_j$, the left-hand side of the inquality increases by one while the right-hand side remains the same. As the inequality previously did not hold, i.e. the left-hand side was in fact greater than the right-hand side, it still cannot hold. Thus $V_2\cup V_j$ must be a term on the right-hand side of inequality $n$, for all $j$. There are $n-2$ choices of $j$, hence $n-2$ terms on the right-hand side of inequality $n-1$ must be these circuit-hyperplanes, $V_2\cup V_j$. 

Next, note that if we tighten $V_1\cup V_2$, the resulting matroid is representable, by \cite[Lemma 4.6]{missingaxiom}. This decreases the rank of $V_1\cup V_2$ by one and leaves the ranks of all other subsets unchanged. As with $V_2\cup V_j$, $V_1\cup V_2$ must be a term of inequality $n-1$ in order for the inequality to be able to reflect the change in representability. If $V_1\cup V_2$ was a term on the right-hand side, after tightening $V_1\cup V_2$ the right-hand side would decrease by one and the left-hand side would remain the same. As prior to tightening the left-hand side was greater than the right-hand side, this fact is still true, meaning the inequality does not hold, contradicting the matroid being representable. A term on the left-hand side of inequality $n-1$ must thus be equal to $V_1\cup V_2$.

Now consider a graph $G$ on vertices $V=\{X_1,...,X_{n-1}\}$ with adjacency relation $a$ such that $a(X_i)=\{X_j \ | \ X_i\cup X_j \ \mathrm{is} \ \mathrm{a} \ \mathrm{term} \ \mathrm{on} \ \mathrm{the} \ \text{right-hand side} \ \mathrm{of} \ \mathrm{inequality} \ n-1\}$. Take the subgraph $G'$ induced by the edges of $G$ corresponding to the circuit-hyperplanes $V_2\cup V_3,\ldots,V_2\cup V_n$. Suppose $G'$ has a path of length three. Call this path $X_a,e,X_b,f,X_c,g,X_d$ and let the edges $e,f,g$ refer to the circuit-hyperplanes $V_2\cup V_i$, $V_2\cup V_j$, $V_2\cup V_k$ respectively.
\begin{figure}[ht]
\centering
\begin{tikzpicture}[thick]
\coordinate [label=-90:$X_a$] (Xa) at (0,0);
\coordinate [label=135:$X_b$] (Xb) at (1,1);
\coordinate [label=45:$X_c$] (Xc) at (4,1);
\coordinate [label=-90:$X_d$] (Xd) at (5,0);
\draw (Xa) -- (Xb) -- (Xc) -- (Xd);
\foreach \p in {Xa,Xb,Xc,Xd} \filldraw (\p) circle (3pt);
\node at ($(Xa)!0.5!(Xb) + (-0.8,0.2)$) {$V_2 \cup V_i$};
\node at ($(Xb)!0.5!(Xc) + (0,0.3)$) {$V_2 \cup V_j$};
\node at ($(Xc)!0.5!(Xd) + (0.7,0.2)$) {$V_2 \cup V_k$};
\end{tikzpicture}
\end{figure}\\
Consider what elements must lie where; $X_b\cup X_c$ must be equal to $V_2\cup V_j$. $X_b$ cannot contain any elements of $V_j$ as $X_a\cup X_b=V_2\cup V_i$, and $V_i$ and $V_j$ are disjoint. However, $X_c$ cannot contain any of the elements of $V_j$ either, since $X_c\cup X_d=V_2\cup V_k$ does not. This structure can thus not exist, i.e. $G'$ can have no paths of length three. The same reasoning shows $G'$ cannot have a cycle of length three. We will now show that $G'$ does in fact contain a path or cycle of length three. As $V_1\cup V_2$ must be a term on the left-hand side of inequality $n-1$, one of $X_1,...,X_{n-1}$ must contain elements from $V_1$. This subset cannot be incident with any edge representing a circuit-hyperplane $V_2\cup V_i$, so $G'$ has at most $n-2$ vertices. $G'$ must have $n-2$ edges. As trees must have an edge set of size one less than the number of vertices, this shows $G'$ is not a tree, and thus contains a path or cycle of length three. Hence \kin{n}$^-$ cannot have a bad family for inequality $n-1$.
\end{proof}

\pagebreak

\begin{figure}[ht]
\centering
\includegraphics[scale=0.6]{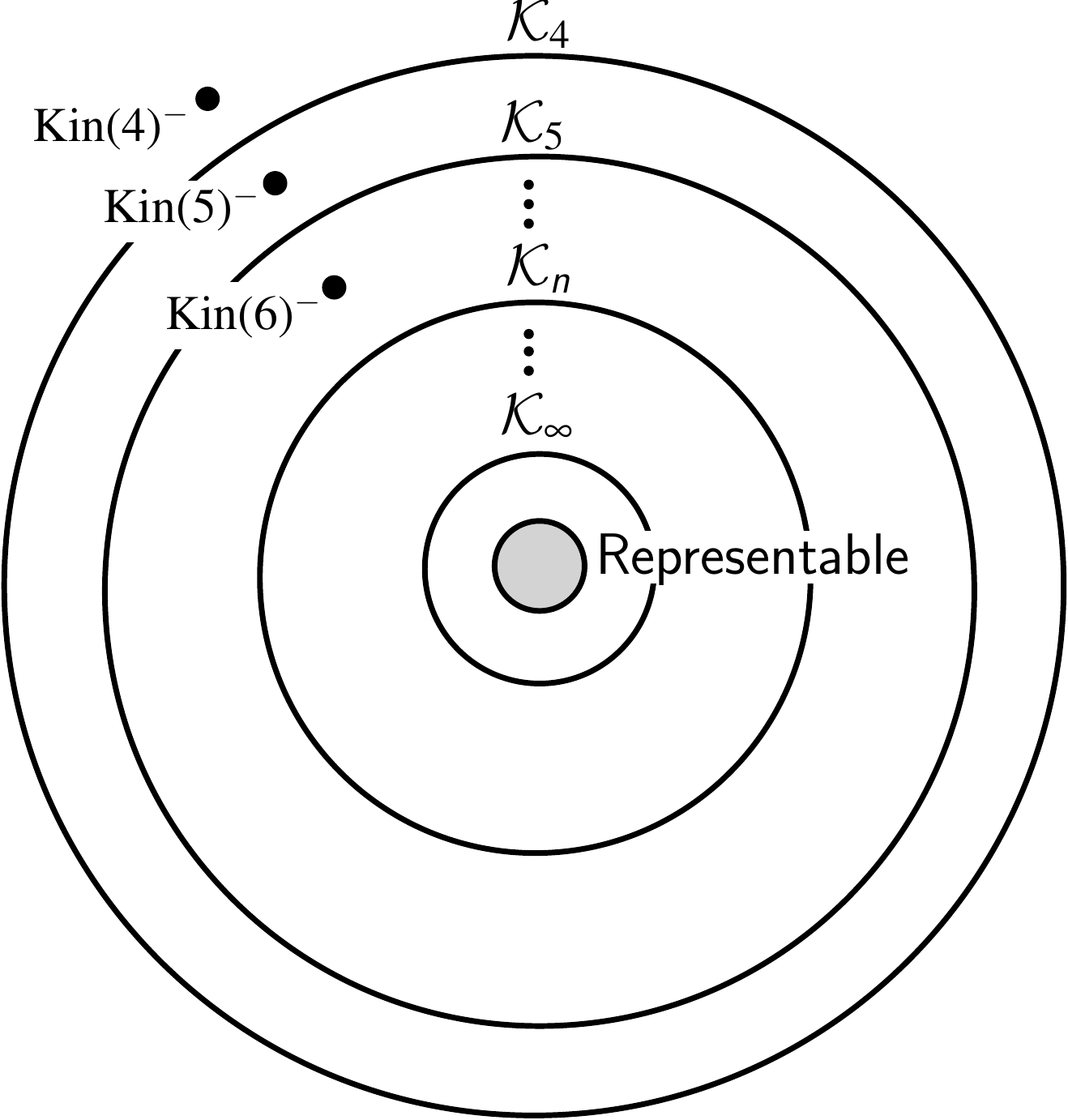}
\caption{Kinser classes (2)}
\end{figure}

Next we show that the first Kinser class, $\mathcal{K}_4$, is dual-closed.

\begin{lemma}
Let $M$ be an excluded minor for the class $\mathcal{K}_4$. Let \vect{X}{4} be a bad family in $M$ for inequality $4$. Then \vect{X}{4} is a partition of $E(M)$.
\end{lemma}

\begin{proof}
Suppose there is an element of $E(M)$ which is not contained in some $X_i$. We could then delete this element and \vect{X}{n} would still form a bad family, contradicting $M$ being minor-minimal with respect to not being in $\mathcal{K}_4$. Thus \vect{X}{n} must cover the entire ground set.

Now assume the sets in the bad family are not disjoint, so there exists an $x$ which is in $X_i\cap X_j$ for some $i,j$. Contract $x$, and for each set $X_k$, let $X'_k=X_k-x$. 

Let $L$ be some union of the $X_i$, so that $r(L)$ appears somewhere in the inequality. Say that $L$ is stable if $r_{M/x}(L-x) = r_{M}(L)$. If none of the five terms on the right-hand side of the inequality are stable, then the sum on the right-hand side of the inequality decreases by exactly five when we contract $x$ and remove $x$ from each term. Since the left-hand side can decrease by at most five, this means that $X_1/x,\ldots,X_n/x$ is a bad family in $M/x$, which is impossible as $M$ is an excluded minor for $\mathcal{K}_4$, and so $M/x\in\mathcal{K}_4$. Therefore there is a stable set on the right-hand side.

As $x\in X_i\cap X_j$, the rank of any term that includes $X_i$ or $X_j$ decreases when we contract $x$ and remove it from the term, and so these terms cannot be stable. There is at most one term on the right-hand side of the inequality that does not involve $X_i$ or $X_j$, so the right-hand side has exactly one stable term. Now this stable term is equal to $X_m\cup X_n$, where either $m$ or $n$ is equal to $3$ or $4$. Note $x\notin cl(X_m)$ and $x\notin cl(X_n)$, or else $X_m\cup X_n$ would not be stable. Therefore either $X_3$ or $X_4$ is stable, so there is a stable term on the left-hand side of the inequality also. This means that $X_1/x,\ldots,X_n/x$ is a bad family in $M/x$, and we get another contradiction.
\end{proof}

\begin{lemma}
$ \mathcal{K}_4=\mathcal{K}_4^*$
\end{lemma}

\begin{proof}
Assume for a contradiction that $M\in\mathcal{K}_4$, $M^*\notin\mathcal{K}_4$. Note that $M^*$ contains a minor-minimal matroid not in $\mathcal{K}_4$. Let $N$ be a minor of $M$ such that $N^*$ is an excluded minor for $\mathcal{K}_4$. Let \vect{X}{4} be a bad family of $N^*$. By the previous lemma, we have that \vect{X}{4} partitions $E(N^*)$. By assumption we have that
\begin{equation}
\begin{split}
&\qquad r^*(X_3)+r^*(X_4)+r^*(X_1\cup X_2)+r^*(X_1\cup X_3\cup X_4)+r^*(X_2\cup X_3\cup X_4)\\
&\nleq r^*(X_1\cup X_3)+r^*(X_1\cup X_4)+r^*(X_2\cup X_3)+r^*(X_2\cup X_4)+r^*(X_3\cup X_4)
\end{split}
\end{equation}
Recall $r^*(X)=|X|+r(E-X)-r(N)$. Use this identity on every term in the inequality. Now we see that (4.4.1) is true if and only if (4.4.2) is true.
\begin{equation}
\begin{split}
& |X_3|+|X_4|+r(\ol{X_3})+r(\ol{X_4})+|X_1\cup X_2|+r(\ol{X_1\cup X_2})\\ 
& \qquad +|X_1\cup X_3\cup X_4|+r(\ol{X_1\cup X_3\cup X_4})\\
& \qquad +|X_2\cup X_3\cup X_4|+r(\ol{X_2\cup X_3\cup X_4})\\
& \nleq |X_1\cup X_3|+r(\ol{X_1\cup X_3})+|X_1\cup X_4|+r(\ol{X_1\cup X_4})\\
& \qquad+|X_2\cup X_3|+r(\ol{X_2\cup X_3})+|X_2\cup X_4|\\
& \qquad +r(\ol{X_2\cup X_4})+|X_3\cup X_4|+r(\ol{X_2\cup X_4})
\end{split}
\end{equation}
where $\overline{X_i}=E-X$.
Note that every $-r(N)$ cancelled out as there is an equal number of terms on each side of the inequality.

Using the identities $|X_i\cup X_j|=|X_i|+|X_j|-|X_1\cap X_j|$ and $|X_i\cup X_j\cup X_k|=|X_i|+|X_j|+|X_k|-|X_i\cap X_j|-|X_i\cap X_k|-|X_j\cap X_k|+|X_i\cap X_j\cap X_k|$, we can simplify this. Fully apply these identities to each cardinality term. We now have that (4.4.2) is true if and only if (4.4.3) is true.
\begin{equation}
\begin{split}
& |X_3|+|X_4|+r(\ol{X_3})+r(\ol{X_4})\\
&\qquad+|X_1|+|X_2|-|X_1\cap X_2|+r(\ol{X_1\cup X_2})\\ 
& \qquad +|X_1|+|X_3|+|X_4|-|X_1\cap X_3|-|X_1\cap X_4|\\
&\qquad\qquad-|X_3\cap X_4|+|X_1\cap X_3\cap X_4|+r(\ol{X_1\cup X_3\cup X_4})\\
& \qquad +|X_2|+|X_3|+|X_4|-|X_2\cap X_3|-|X_2\cap X_4|-|X_3\cap X_4|\\
&\qquad\qquad+|X_2\cap X_3\cap X_4|+r(\ol{X_2\cup X_3\cup X_4})\\
& \nleq |X_1|+|X_3|-|X_1\cap X_3|+r(\ol{X_1\cup X_3})\\
&\qquad+|X_1|+|X_4|-|X_1\cap X_4|+r(\ol{X_1\cup X_4})\\
& \qquad+|X_2|+|X_3|-|X_2\cap X_3|+r(\ol{X_2\cup X_3})\\
&\qquad+|X_2|+|X_4|-|X_2\cap X_4|+r(\ol{X_2\cup X_4})\\
& \qquad +|X_3|+|X_4|-|X_3\cap X_4|+r(\ol{X_2\cup X_4})
\end{split}
\end{equation}

As \vect{X}{4} is a partition of $E(M)$, any term $X_i\cap X_j$ is empty. Cancelling out these terms, and all common terms, this yields
\begin{equation}
\begin{split}
&\qquad r(\overline{X_3})+r(\overline{X_4})+r(\overline{X_1\cup X_2})+r(\overline{X_1\cup X_3\cup X_4})+r(\overline{X_2\cup X_3\cup X_4})\\
& \nleq r(\overline{X_1\cup X_3})+r(\overline{X_1\cup X_4})+r(\overline{X_2\cup X_3})+r(\overline{X_2\cup X_4})+r(\overline{X_3\cup X_4})
\end{split}
\end{equation}
Now let $Y_1=\overline{X_1\cup X_2\cup X_3}$, $Y_2=\overline{X_1\cup X_2\cup X_4}$, $Y_3=\overline{X_1\cup X_3\cup X_4}$, and $Y_4=\overline{X_2\cup X_3\cup X_4}$. 

We have the following equalities
\begin{align*}
Y_1\cup Y_3\cup Y_4 & = \overline{X_1\cup X_2\cup X_3}\cup\overline{X_1\cup X_3\cup X_4}\cup\overline{X_2\cup X_3\cup X_4}\\
& = X_4\cup X_2\cup X_1\\
& = \ol{X_3}
\end{align*}
\begin{align*}
Y_2\cup Y_3\cup Y_4 & = \overline{X_1\cup X_2\cup X_4}\cup\overline{X_1\cup X_3\cup X_4}\cup\overline{X_2\cup X_3\cup X_4}\\
& = X_3\cup X_2\cup X_1\\
& = \ol{X_4}\\
\end{align*}
\begin{align*}
Y_1\cup Y_2 & = \overline{X_1\cup X_2\cup X_3}\cup\overline{X_1\cup X_2\cup X_4}\\
& = X_4\cup X_3\\
& = \ol{X_1\cup X_2}
\end{align*}
\begin{align*}
Y_1\cup Y_3 & = \overline{X_1\cup X_2\cup X_3}\cup\overline{X_1\cup X_3\cup X_4}\\
& = X_4\cup X_2\\
& = \ol{X_1\cup X_3}
\end{align*}
\begin{align*}
Y_2\cup Y_3 & = \overline{X_1\cup X_2\cup X_4}\cup\overline{X_1\cup X_3\cup X_4}\\
& = X_3\cup X_2\\
& = \ol{X_1\cup X_4}
\end{align*}
\begin{align*}
Y_1\cup Y_4 & = \overline{X_1\cup X_2\cup X_3}\cup\overline{X_2\cup X_3\cup X_4}\\
& = X_4\cup X_1\\
& = \ol{X_2\cup X_3}
\end{align*}
\begin{align*}
Y_2\cup Y_4 & = \overline{X_1\cup X_2\cup X_4}\cup\overline{X_2\cup X_3\cup X_4}\\
& = X_3\cup X_1\\
& = \ol{X_2\cup X_4}
\end{align*}
\begin{align*}
Y_3\cup Y_4 & = \overline{X_1\cup X_3\cup X_4}\cup\overline{X_2\cup X_3\cup X_4}\\
& = X_2\cup X_1\\
& = \ol{X_3\cup X_4}
\end{align*}

Inequality (4.5.4) thus becomes
\begin{align*}
r(Y_1\cup Y_3\cup Y_4)+r(Y_2\cup Y_3\cup Y_4)+r(Y_1\cup Y_2)+r(Y_3)+r(Y_4)\\
\nleq r(Y_1\cup Y_3)+r(Y_2\cup Y_3)+r(Y_1\cup Y_4)+r(Y_2\cup Y_4)+r(Y_3\cup Y_4)
\end{align*}
These rank terms are exactly those of inequality $4$, which holds for \vect{Y}{4} as $N\in\mathcal{K}_4$. Thus $N^*$ also satisfies inequality $4$, for all choices of \vect{X}{4}.
\end{proof}

We conjecture that this is the only Kinser class which is dual-closed. We provide a proof of this for only the second class. There appears to be no simple way to verify that a matroid satisfies a particular Kinser inequality, which leads to the difficulty involved in the next chapter. In a later chapter, we will provide a result on the complexity of verifying that a matroid satisfies Kinser inequality $n$.

\begin{thm}
$\mathcal{K}_5\neq\mathcal{K}_5^*$
\end{thm}

The proof of this theorem forms the next chapter. There are now two possibilities of how the Kinser classes sit in the hierarchy, shown in the following diagrams.


\begin{figure}[h]
\centering
\includegraphics[scale=0.6]{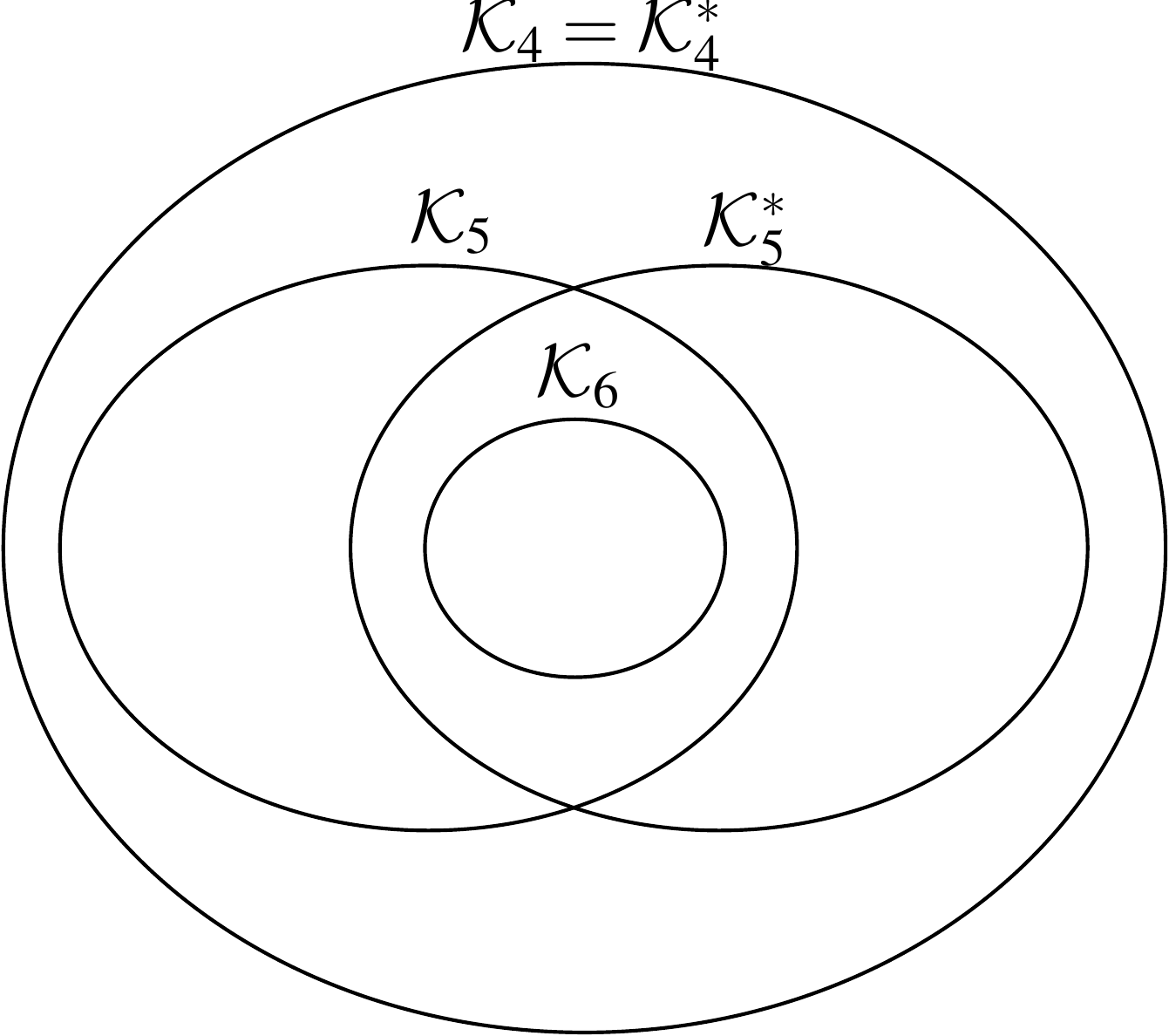}
\caption{Kinser classes (3a)}
\end{figure}

\begin{figure}[h]
\centering
\includegraphics[scale=0.6]{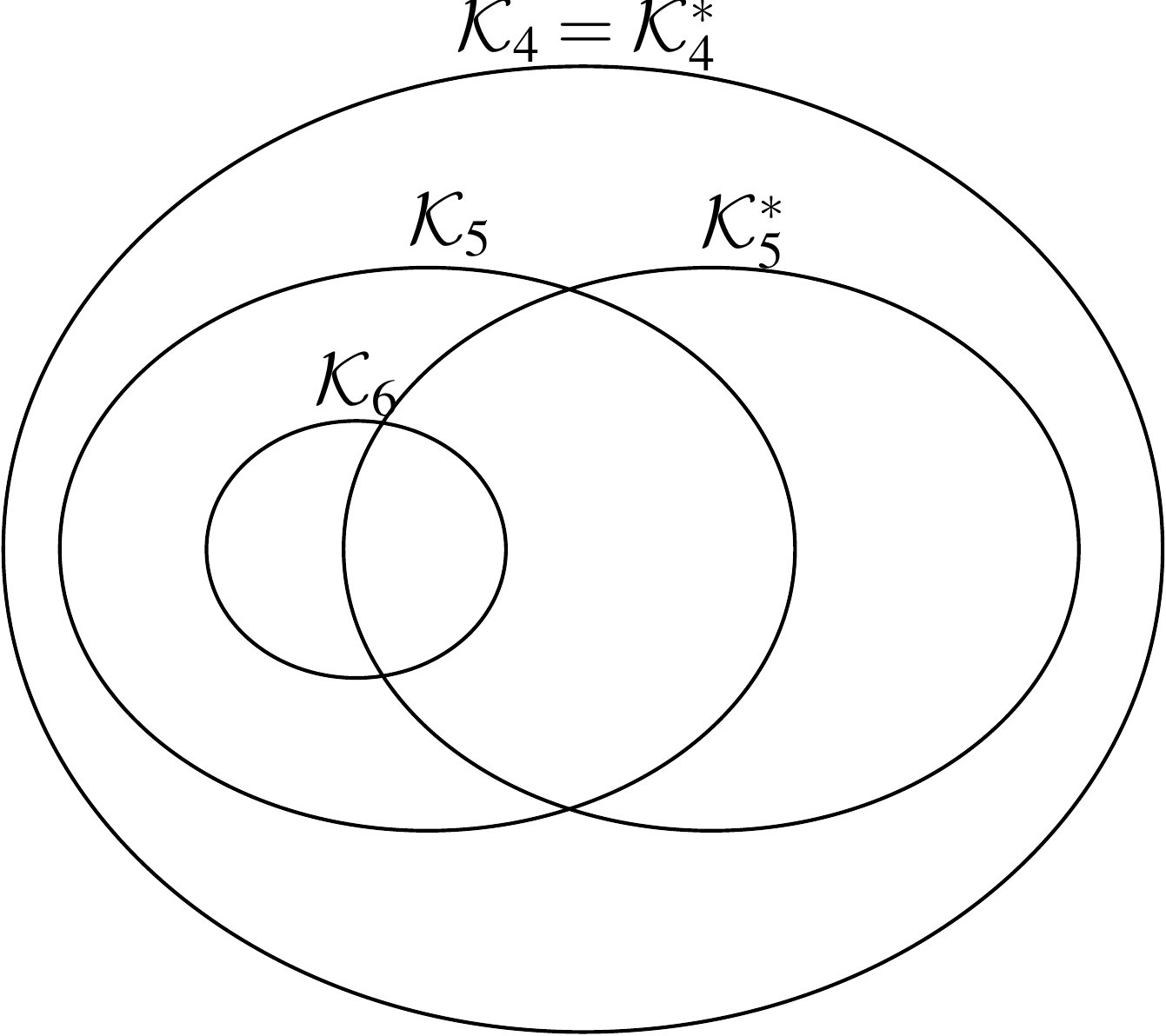}
\caption{Kinser classes (3b)}
\end{figure}

\chapter{\texorpdfstring{$\mathcal{K}_5\neq\mathcal{K}_5^*$}{K5 != K5*}}

\begin{thm}
$\kin{5}^-\in\mathcal{K}_5^*$
\end{thm}

\begin{proof}
Recall that $\kin{5}^-$ is the matroid obtained from the rank $5$ Kinser matroid by relaxing the circuit-hyperplane $V_1\cup V_2$. Assume ($\kin{5}^-)^*$ has at least one bad family $X_{1},\ldots,X_{5}$ violating Kinser inequality $5$, where each set $X_i$ is a flat of $(\kin{5}^-)^*$.

Note the complement of $V_{1}\cup V_{2}$ is a relaxed circuit-hyperplane in $(\kin{5}^-)^*$, and tightening it produces Kin($5$)$^*$, which is representable. This operation affects the rank function by decreasing the rank of exactly one set: $V_{3}\cup V_{4}\cup V_{5}$. Suppose $V_3\cup V_4\cup V_5$ is not a term in inequality $5$ as applied to \vect{X}{5}. Then evaluating the inequality after relaxing the circuit-hyperplane would give the same value as before relaxing, contradicting the change in representability. As the rank of $V_{3}\cup V_{4}\cup V_{5}$ decreases by one when we tighten it, the only way for the inequality to hold only after the tighening is for the left-hand side of the inequality to decrease and the right-hand side to remain the same. So $V_3\cup V_4\cup V_5$ must be a term on the left-hand side of the inequality. 

Next note that $V_2\cup V_3$, $V_2\cup V_4$, and $V_2\cup V_5$ are circuit-hyperplanes on $\kin{5}^-$, so $V_{1}\cup V_{4}\cup V_{5}$, $V_{1}\cup V_{3}\cup V_{5}$, $V_{1}\cup V_{3}\cup V_{4}$ are circuit-hyperplanes of $(\kin{5}^-)^*$. Relaxing any one of these will make $(\kin{5}^-)^*$ representable again. These subsets must thus be terms in the inequality. The rank of each of these subsets increases by one when relaxed, so for the inequality to hold the right-hand side must increase, meaning these three sets must be terms on the right-hand side of inequality $5$. 

Let $e\in V_2$. If $e\notin\vect{X}{5}$, then $(\kin{5}^-)^*\backslash
e=(\kin{5}^-/e)^*$ has a bad family. This means that $\kin{5}^-/e$ is
not representable. We will show that the elements of $V_1$ are freely
placed in $\kin{5}^-/e$ . Pick $z\in V_1$. Let $Z$ be a non-spanning
subset in $\kin{5}^-/e$ such that $z\in cl_{\kin{5}^-/e}(Z)-Z$. This
implies by \cite[Proposition lookitup]{Oxley} $z\in cl_{\kin{5}^-}(Z\cup
\{e\})$. Note that the closure of a set in \kin{5} only differs from
that in $\kin{5}^-$ in that it may contain additional elements in
\kin{5}. Thus $z\in cl_{\kin{5}}(Z\cup \{e\})$. 

Recall that $M_6$ is
the transversal matroid whose truncation is defined to be \kin{5}. As
$Z\cup \{e\}$ does not span \kin{5}, truncation does not affect the rank
of this subset, and so $z\in cl_{M_6}(Z\cup \{e\})$. Now consider the
neighbours of $Z\cup \{e\}$ in the transversal system $\mathcal{A}$. The
neighbours of $z$ must be contained in this set, as otherwise the rank
of $Z\cup \{e,f\}$ would be greater than the rank of $Z\cup \{e\}$,
contradicting $z\in cl(Z\cup \{e\})$. Recall that in $M_{r+1}$, $V_i$ is
incident with $A_0,\ldots,A_{i-1},A_{i+2},\ldots,A_r$. $Z\cup\{e\}$ is
thus neighbours with $A_0,A_1,A_3$, and $A_4$.

Recall $Z$ is
non-spanning in $\kin{5}^-/e$. $\kin{5}^-/e$ has rank $4$, so
$r_{\kin{5}^-/e}(Z)\leq 3$, and $r_{\kin{5}^-}(Z\cup\{e\})\leq 4$. This
rank cannot change after tightening $V_1\cup V_2$, so
$r_{\kin{5}}(Z\cup\{e\})\leq 4$. As $Z\cup\{e\}$ is hence non-spanning in
\kin{5}, $r_{M_6}(Z\cup\{e\}\leq 4$. This means that $Z\cup\{e\}$ has
exactly $A_0,A_1,A_3,$ and $A_4$ as neighbours. Thus
$Z\cup\{e\}\subseteq V_1\cup\{e,f\}$. This implies that $Z\cup\{e,z\}$
is independent in $\kin{5}^-$, and so $Z\cup\{z\}$ is independent in
$\kin{5}^-/e$. This is a contradiction, as $z\in cl_{\kin{5}^-/e}(Z)-Z$.
Thus $z$ is free in $\kin{5}^-/e$ and $\kin{5}^-/e$ is a free-extension
of $\kin{5}^-/e\backslash z=\kin{5}/e\backslash z$ which is
$\mathbb{K}$-representable. This contradicts $(\kin{5}^-)^*\backslash
e=(\kin{5}^-/e)^*$ having a bad family, and so all elements of $V_2$
must be contained in some set $X_i$.

Suppose the elements of $V_2$ are contained in different sets in the bad family. That is, suppose $e\in X_i$ and $f\in X_j$ for some $i,j$.The three circuit-hyperplanes $V_1\cup V_3\cup V_4$, $V_1\cup V_4\cup V_5$, $V_1\cup V_3\cup V_5$ and the relaxed circuit-hyperplane $V_3\cup V_4\cup V_5$ do not include elements of $V_2$, and thus cannot be equal to terms in the inequality which use $X_i$ or $X_j$. Removing these terms from the inequality leaves at most three possible terms for these circuit-hyperplanes, regardless of the values of $i$ and $j$. Thus both elements of $V_2$, $e$ and $f$, must be contained in the same set $X_i$ of the bad family. 

Considering each possible location for the elements of $V_2$ gives us five cases -- $V_2\cap X_1\neq\varnothing, V_2\cap X_2\neq\varnothing, V_2\cap X_3\neq\varnothing, V_2\cap X_4\neq\varnothing, V_2\cap X_5\neq\varnothing$. Take the case $X_5\cap V_2\neq\varnothing$.
Swapping $X_3$ and $X_5$ gives the inequality, and this case is thus identical to the case where $V_2\cap X_3\neq\varnothing.$ We thus reduce to the four cases examined below.

\newtheorem*{case1}{Case 1}
\begin{case1}
$X_3\cap V_2\neq\varnothing$\\
Consider which term on the left-hand side of the inequality must be $V_3\cup V_4\cup V_5$. As this term does not contain any elements of $V_2$ and these elements are contained in $X_3$, the term cannot involve $X_3$. In the following diagram, the ovals represent the sets \vect{X}{5}. An edge between two sets represents the union of those two sets. The edges shown below are on the right-hand side of the inequality, and thus possible locations for the necessary circuit-hyperplanes $ V_{1}\cup V_{3}\cup V_{5}$, $V_{1}\cup V_{4}\cup V_{5}$, $V_{1}\cup V_{3}\cup V_{4}$. Note that $X_1\cup X_2$, indicated by the dashed line, appears on the left-hand side of the inequality. The edges coming from $X_3$ have been left off as, since the elements of $V_2$ are contained in $X_3$, none of these edges could represent the three circuit-hyperplanes.

\begin{figure}[ht]
\centering
\begin{tikzpicture}[thin,line join=round]
\node[ellipse, draw] (x1) at (0,3.5) {$X_1$};
\node[ellipse, draw] (x2) at (0,0) {$X_2$};
\node[ellipse, draw] (x3) at (3.5,0) {$X_3$};
\node[ellipse, draw] (x4) at (0,-3.5) {$X_4$};
\node[ellipse, draw] (x5) at (-3.5,0) {$X_5$};
\draw (x1) -- (x5) -- (x4) -- (x2) -- (x5);
\draw[dashed] (x1) -- (x2);
\end{tikzpicture}
\end{figure}

If $V_3\cup V_4\cup V_5$ is equal to $X_4$ or $X_5$ then any term involving these sets cannot be one of the circuit-hyperplanes $V_{1}\cup V_{3}\cup V_{5}$, $V_{1}\cup V_{4}\cup V_{5}$, $V_{1}\cup V_{3}\cup V_{4}$ which must appear on the right-hand side of the inequality. This will not leave us with enough terms on the right-hand side which could be these three circuit-hyperplanes -- if $V_3\cup V_4\cup V_5=X_4$, then the terms available for the circuit-hyperplanes are $X_1\cup X_5$ and $X_2\cup X_5$, while if $V_3\cup V_4\cup V_5=X_5$, there are again only two terms available, $X_1\cup X_2$ and $X_2\cup X_4$. If $V_{3}\cup V_{4}\cup V_{5}=X_{2}\cup X_{4}\cup X_{5}$, then elements of $V_{1}$ can only be in $X_1$. Note that all of the three circuit-hyperplanes which must appear on the right-hand side include $V_1$. However, as $X_3\cap V_2\neq\varnothing$, the only free term using $X_1$ is $X_1\cup X_5$, so we are unable to have the three needed circuit-hyperplanes on the right-hand side. The only remaining possibility is that $V_3\cup V_4\cup V_5=X_1\cup X_2$. 

We now have four possibilities of where the circuit-hyperplanes on the right-hand side of the inequality lie, and we will consider various subcases. Let $\{i,j,k\}=\{3,4,5\}$. The following diagrams show these subcases, as indicated by the bold lines.

\pagebreak

\begin{figure}
\centering
\subfloat[Subcase 1]{
\begin{tikzpicture}[thin,line join=round]
\node[ellipse, draw] (x1) at (0,2.5) {$X_1$};
\node[ellipse, draw] (x2) at (0,0) {$X_2$};
\node[ellipse, draw] (x3) at (2.5,0) {$X_3$};
\node[ellipse, draw] (x4) at (0,-2.5) {$X_4$};
\node[ellipse, draw] (x5) at (-2.5,0) {$X_5$};
\draw[ultra thick] (x1) -- (x5) -- (x2) -- (x4);
\draw (x4) -- (x5);
\draw[dashed] (x1) -- (x2);
\end{tikzpicture}}
\qquad
\subfloat[Subcase 2]{
\begin{tikzpicture}[thin,line join=round]
\node[ellipse, draw] (x1) at (0,2.5) {$X_1$};
\node[ellipse, draw] (x2) at (0,0) {$X_2$};
\node[ellipse, draw] (x3) at (2.5,0) {$X_3$};
\node[ellipse, draw] (x4) at (0,-2.5) {$X_4$};
\node[ellipse, draw] (x5) at (-2.5,0) {$X_5$};
\draw[ultra thick] (x1) -- (x5) -- (x2);
\draw[ultra thick] (x4) -- (x5);
\draw (x2) -- (x4);
\draw[dashed] (x1) -- (x2);
\end{tikzpicture}}

\subfloat[Subcase 3]{
\begin{tikzpicture}[thin,line join=round]
\node[ellipse, draw] (x1) at (0,2.5) {$X_1$};
\node[ellipse, draw] (x2) at (0,0) {$X_2$};
\node[ellipse, draw] (x3) at (2.5,0) {$X_3$};
\node[ellipse, draw] (x4) at (0,-2.5) {$X_4$};
\node[ellipse, draw] (x5) at (-2.5,0) {$X_5$};
\draw[ultra thick] (x4) -- (x5) -- (x2) -- (x4);
\draw (x1) -- (x5);
\draw[dashed] (x1) -- (x2);
\end{tikzpicture}}
\qquad
\subfloat[Subcase 4]{
\begin{tikzpicture}[thin,line join=round]
\node[ellipse, draw] (x1) at (0,2.5) {$X_1$};
\node[ellipse, draw] (x2) at (0,0) {$X_2$};
\node[ellipse, draw] (x3) at (2.5,0) {$X_3$};
\node[ellipse, draw] (x4) at (0,-2.5) {$X_4$};
\node[ellipse, draw] (x5) at (-2.5,0) {$X_5$};
\draw[ultra thick] (x1) -- (x5) -- (x4) -- (x2);
\draw (x2) -- (x5);
\draw[dashed] (x1) -- (x2);
\end{tikzpicture}}
\end{figure}

\pagebreak

In each subcase examined below, we are considering the circuit-hyperplanes to lie as indicated by the diagram at the start of each subcase. In all of these, there are multiple options for which elements are in what sets, as described below.

\pagebreak

{\bf Subcase 1.1.}

\begin{figure}[ht]
\centering
\begin{tikzpicture}[thin,line join=round]
\node[ellipse, draw] (x1) at (0,3.5) {$X_1$};
\node[ellipse, draw] (x2) at (0,0) {$X_2$};
\node[ellipse, draw] (x3) at (3.5,0) {$X_3$};
\node[ellipse, draw] (x4) at (0,-3.5) {$X_4$};
\node[ellipse, draw] (x5) at (-3.5,0) {$X_5$};
\draw (x1) -- node[rotate=45, above] {$V_1 \cup V_j \cup V_k$} (x5) -- (x4) -- node[rotate=-90, above] {$V_1 \cup V_i \cup V_j$} (x2) -- node[above] {$V_1 \cup V_i \cup V_k$} (x5);
\draw[dashed] (x1) -- (x2);
\end{tikzpicture}
\end{figure}

Here, we have assumed that $V_2\subseteq X_3$. $X_3$ may contain other elements of the ground set of $M$, thus any terms in the inequality involving $X_3$ will not be immediately evaluated. Recall that $X_1\cup X_2=V_3\cup V_4\cup V_5$. This means that $V_1$ can only be a subset of $X_4$ or $X_5$. In order for the three circuit-hyperplanes to lie as indicated, both $X_4$ and $X_5$ must contain $V_1$. $X_2$ must be equal to $V_i$, as this forms the common values of the two circuit-hyperplanes involving $X_2$, excluding $V_1$. Finally, $X_5$ must be equal to $V_1\cup V_k$, while $X_1$ must contain $V_j\cup V_k$. As we assumed the sets $X_i$ to be flats, $X_1$ could also include one or two elements of $V_i$. It cannot contain all three elements, as $V_i\cup V_j\cup V_k$ is a relaxed circuit-hyperplane, and thus has closure equal to the entire ground set. This gives the following three subcases, which differ only in elements of $X_1$. In each subcase, we will begin by evaluating terms of the left-hand and right-hand sides of inequality $5$, then show that the left-hand side must be lower in value, meaning that the inequality holds.

\begin{itemize}[Subcase 1.1a:] 
\item $ X_1=V_j\cup V_k$ \\
$X_2=V_i$\\
$X_3=V_2\cup Z$ where $Z\subseteq E(M)$\\
$X_4=V_1\cup V_j$\\
$X_5= V_1\cup V_k$
\end{itemize}

Note that $\ol{X_1\cup X_2}=V_1\cup V_2$ and that $\ol{X_2\cup X_4\cup X_5}=\ol{V_1\cup V_i\cup V_j\cup V_k}=V_2$.

We will use $M$ to refer to $\kin{5}^-$.

\begin{equation*}
\begin{split}
LHS & = r^*(X_3)+r^*(X_4)+r^*(X_5)+r^*(X_1\cup X_2)+r^*(X_1\cup X_3\cup X_5)\\&\qquad+r^*(X_2\cup X_3\cup X_4)+r^*(X_2\cup X_4\cup X_5) \\
& = r^*(X_3)+r^*(X_1\cup X_3\cup X_5)+r^*(X_2\cup X_3\cup X_4)\\
&\qquad+|X_4|+|X_5|+|X_1\cup X_2| +|X_2\cup X_4\cup X_5|\\
&\qquad +r(\ol{X_4})+r(\ol{X_5}) + r(\ol{X_1 \cup X_2})+ r(\ol{X_2 \cup X_4\cup X_5})-4r(M) \\
& = r^*(X_3)+r^*(X_1\cup X_3\cup X_5)+r^*(X_2\cup X_3\cup X_4)-4r(M)\\
&\qquad+6+6+9+12+5+5+5+2 \\
& = r^*(X_3)+r^*(X_1\cup X_3\cup X_5)+r^*(X_2\cup X_3\cup X_4)-4r(M)+50\\
\end{split}
\end{equation*}
\begin{equation*}
\begin{split}
RHS & = r^*(X_1\cup X_3)+r^*(X_1\cup X_5)+r^*(X_2\cup X_3)+r^*(X_2\cup X_4)\\
&\qquad+r^*(X_2\cup X_5)+r^*(X_3\cup X_4)+r^*(X_4\cup X_5) \\
& = r^*(X_1\cup X_3)+r^*(X_2\cup X_3)+r^*(X_3\cup X_4)+|X_1\cup X_5| \\
&\qquad+|X_2\cup X_4|+|X_2\cup X_5|+|X_4\cup X_5|+r(\ol{X_1\cup X_5})\\
&\qquad+r(\ol{X_2\cup X_4})+r(\ol{X_2\cup X_5})+r(\ol{X_4\cup X_5})-4r(M) \\
& \geq r^*(X_1\cup X_3)+r^*(X_2\cup X_3)+r^*(X_3\cup X_4)-4r(M)\\
&\qquad +9+9+9+9+4+4+4+4 \\
& = r^*(X_1\cup X_3)+r^*(X_2\cup X_3)+r^*(X_3\cup X_4)-4r(M)+52 \\
\end{split}
\end{equation*}

We thus must have that
\begin{equation*}
\begin{split}
&\qquad r^*(X_3)+r^*(X_1\cup X_3\cup X_4)+r^*(X_2\cup X_3\cup X_5)\\
& > r^*(X_1\cup X_3)+r^*(X_2\cup X_3)+r^*(X_3\cup X_4)+2
\end{split}
\end{equation*}
in order for \vect{X}{5} to be a bad family.

Suppose $Z=\varnothing$, i.e. $X_3=V_2$. We then have $2+9+9>8+5+8+2$ which is untrue. Suppose we increase the size of $Z$. If the cardinality of one of these unions does not change, then neither does the rank of the complement, and thus the corank is unchanged. Note that $X_1\cup X_3\cup X_5=V_1\cup V_2\cup V_j\cup V_k\cup Z$. If the cardinality of this set is greater when $Z$ is non-empty, then $Z\subseteq V_i$. As $\ol{X_1\cup X_3\cup X_5}=V_i-Z$, this will cause the rank of $\ol{X_1\cup X_3\cup X_5}$ to decrease. The corank is thus unchanged from when $Z=\varnothing$. A similar argument shows that the corank of $X_2\cup X_3\cup X_4=V_1\cup V_2\cup V_i\cup V_j\cup Z$ must be unchanged as its complement $V_k-Z$ is independent. Also, 
\begin{flalign*}
r^*(X_3\cup X_4)+2 & \geq r^*(X_3)+r^*(X_4)-r^*(X_3\cap X_4)+2 &\\
& \geq r^*(X_3)+r^*(X_4)-r^*(Z\cap X_4)+2 &\\
& \geq r^*(X_3)+2 &\\
& \geq r^*(X_3) &
\end{flalign*}
Thus increasing the size of $Z$ can only cause the left-hand side to decrease by more than the right-hand side, and so we cannot have a bad family for any choice of $Z$.

\begin{itemize}[Subcase 1.1b:] 
\item $ X_1=V_j\cup V_k\cup\{a\}$ where $a\in V_i$\\
$X_2=V_i$\\
$X_3=V_2\cup Z$ where $Z\subseteq E(M)$\\
$X_4=V_1\cup V_j$\\
$X_5= V_1\cup V_k$
\end{itemize}
The only terms whose value changes from before is $|X_1\cup X_5|$ which increases by one. So we must have that
\begin{equation*}
\begin{split}
& \qquad r^*(X_3)+r^*(X_1\cup X_3\cup X_4)+r^*(X_2\cup X_3\cup X_4)\\
& > r^*(X_1\cup X_3)+r^*(X_2\cup X_3)+r^*(X_3\cup X_4)+3
\end{split}
\end{equation*}
The same argument as before shows that we do not have a bad family, for any choice of $Z$.
\begin{itemize}[Subcase 1c:] 
\item $ X_1=V_j\cup V_k\cup\{a,b\}$ where $a,b\in V_i$\\
$X_2=V_i$\\
$X_3=V_2\cup Z$ where $Z\subseteq E(M)$\\
$X_4=V_1\cup V_j$\\
$X_5= V_1\cup V_k$
\end{itemize}
Here, $|X_1\cup X_5|$ is one higher than in the previous subcase, while $r(\ol{X_1\cup X_5})$ is one lower, and the inequality is thus identical to subcase 1b.

\pagebreak

{\bf Subcase 1.2.}

\begin{figure}[ht]
\centering
\begin{tikzpicture}[thin,line join=round]
\node[ellipse, draw] (x1) at (0,3.5) {$X_1$};
\node[ellipse, draw] (x2) at (0,0) {$X_2$};
\node[ellipse, draw] (x3) at (3.5,0) {$X_3$};
\node[ellipse, draw] (x4) at (0,-3.5) {$X_4$};
\node[ellipse, draw] (x5) at (-3.5,0) {$X_5$};
\draw (x1) -- node[rotate=45, above] {$V_1 \cup V_i \cup V_j$} (x5) -- node[rotate=-45, below] {$V_1 \cup V_j \cup V_k$} (x4) -- (x2) -- node[above] {$V_1 \cup V_i \cup V_k$} (x5);
\draw[dashed] (x1) -- (x2);
\end{tikzpicture}
\end{figure}

Again, we have by assumption that $V_2\subseteq X_3$. $X_3$ may contain other elements of the ground set of $M$, thus any terms in the inequality involving $X_3$ will not be immediately evaluated. $X_5$ must be equal to $V_1$, as these are the only common elements of the three circuit-hyperplanes. This forces $X_1$ to be equal to $V_i\cup V_j$, as $X_1\cup X_2=V_3\cup V_4\cup V_5$ and thus $X_1$ cannot contain elements of $V_1$. Likewise, $X_2$ must be equal to $V_i\cup V_k$. $X_4$ must contain $V_j\cup V_k$. As we assumed the sets $X_i$ to be flats, $X_4$ could also include one or three elements of $V_1$. It cannot contain exactly two elements of $V_1$, as $V_1\cup V_j\cup V_k$ is a circuit-hyperplane. This gives the following three subcases, which differ only in elements of $X_4$. As before, in each subcase we will begin by evaluating terms of the left-hand and right-hand sides of inequality $5$, then show that the left-hand side must be lower in value, meaning that the inequality holds.

\begin{itemize}[Subcase 1.2a:] 
\item $ X_1=V_i\cup V_j$ \\
$X_2=V_i\cup V_k$\\
$X_3=V_2\cup Z$ where $Z\subseteq E(M)$\\
$X_4=V_j\cup V_k$\\
$X_5= V_1$
\end{itemize}

Note that $\ol{X_1\cup X_2}=V_1\cup V_2$ and $\ol{X_2\cup X_4\cup X_5}=\ol{V_1\cup V_i\cup V_j\cup V_k}=V_2$.
\begin{equation*}
\begin{split}
LHS & = r^*(X_3)+r^*(X_4)+r^*(X_5)+r^*(X_1\cup X_2)+r^*(X_1\cup X_3\cup X_5)\\&\qquad+r^*(X_2\cup X_3\cup X_4)+r^*(X_2\cup X_4\cup X_5) \\
& = r^*(X_3)+r^*(X_1\cup X_3\cup X_5)+r^*(X_2\cup X_3\cup X_4)\\
&\qquad+|X_4| +|X_5|+|X_1\cup X_2|+|X_2\cup X_4\cup X_5| \\
&\qquad + r(\ol{X_4})+r(\ol{X_5}) + r(\ol{X_1 \cup X_2})+ r(\ol{X_2 \cup X_4\cup X_5})-4r(M) \\
& = r^*(X_3)+r^*(X_1\cup X_3\cup X_5)+r^*(X_2\cup X_3\cup X_4)-4r(M)\\
&\qquad+6+3+9+12+5+5+5+2\\
& = r^*(X_3)+r^*(X_1\cup X_3\cup X_5)+r^*(X_2\cup X_3\cup X_4)-4r(M)+47\\
\end{split}
\end{equation*}
\begin{equation*}
\begin{split}
RHS & = r^*(X_1\cup X_3)+r^*(X_1\cup X_5)+r^*(X_2\cup X_3)+r^*(X_2\cup X_4)\\
&\qquad+r^*(X_2\cup X_5)+r^*(X_3\cup X_4)+r^*(X_4\cup X_5) \\
& = r^*(X_1\cup X_3)+r^*(X_2\cup X_3)+r^*(X_3\cup X_4)+|X_1\cup X_5| \\
&\qquad+|X_2\cup X_4|+|X_2\cup X_5|+|X_4\cup X_5|+r(\ol{X_1\cup X_5})\\
&\qquad+r(\ol{X_2\cup X_4})+r(\ol{X_2\cup X_5})+r(\ol{X_4\cup X_5})-4r(M) \\
& \geq r^*(X_1\cup X_3)+r^*(X_2\cup X_3)+r^*(X_3\cup X_4)-4r(M)\\
&\qquad+ 9+9+9+9+4+5+4+4 \\
& = r^*(X_1\cup X_3)+r^*(X_2\cup X_3)+r^*(X_3\cup X_4)-4r(M)+53 \\
\end{split}
\end{equation*}

We thus must have that 
\begin{equation*}
\begin{split}
& \qquad r^*(X_3)+r^*(X_1\cup X_3\cup X_5)+r^*(X_2\cup X_3\cup X_4)\\
& > r^*(X_1\cup X_3)+r^*(X_2\cup X_3)+r^*(X_3\cup X_4)+6
\end{split}
\end{equation*}
Suppose $Z=\varnothing$. Then $2+9+9>8+8+8+6$ which is untrue. The same argument as in subcase 1a shows that if $Z$ is non-empty, we still cannot have a bad family.

\begin{itemize}[Subcase 1.2b:] 
\item $ X_1=V_i\cup V_j$ \\
$X_2=V_i\cup V_k$\\
$X_3=V_2\cup Z$ where $Z\subseteq E(M)$\\
$X_4=V_j\cup V_k\cup\{a\}$ where $a\in V_1$\\
$X_5= V_1$
\end{itemize}

Compared to subcase 2a, $|X_4|$ increases by one on the left-hand side. On the right-hand side, $|X_2\cup X_4|$ increases by one. The overall inequality is unchanged, so the same argument as in subcase 2a holds.

\begin{itemize}[Subcase 1.2c:] 
\item $ X_1=V_i\cup V_j$ \\
$X_2=V_i\cup V_k$\\
$X_3=V_2\cup Z$ where $Z\subseteq E(M)$\\
$X_4=V_j\cup V_k\cup V_i$\\
$X_5= V_1$
\end{itemize}

Compared to subcase 2a, both $|X_4|$ and $|X_2\cup X_4|$ increase by three, leaving the inequality unchanged.

{\bf Subcase 1.3.}

\begin{figure}[ht]
\centering
\begin{tikzpicture}[thin,line join=round]
\node[ellipse, draw] (x1) at (0,3.5) {$X_1$};
\node[ellipse, draw] (x2) at (0,0) {$X_2$};
\node[ellipse, draw] (x3) at (3.5,0) {$X_3$};
\node[ellipse, draw] (x4) at (0,-3.5) {$X_4$};
\node[ellipse, draw] (x5) at (-3.5,0) {$X_5$};
\draw (x1) -- (x5) -- node[rotate=-45, below] {$V_1 \cup V_i \cup V_k$} (x4) -- node[rotate=-90, above] {$V_1 \cup V_j \cup V_k$} (x2) -- node[above] {$V_1 \cup V_i \cup V_j$} (x5);
\draw[dashed] (x1) -- (x2);
\end{tikzpicture}
\end{figure}

Again, we have by assumption that $V_2\subseteq X_3$. $X_3$ may contain other elements of the ground set of $M$, thus any terms in the inequality involving $X_3$ will not be immediately evaluated. Recall that $X_1\cup X_2=V_3\cup V_4\cup V_5$. This means that $V_1$ can only be a subset of $X_4$ or $X_5$. In order for the three circuit-hyperplanes to lie as indicated, both $X_4$ and $X_5$ must contain $V_1$. $X_2$ must be equal to $V_j$, as this forms the common values of the two circuit-hyperplanes involving $X_2$, excluding $V_1$. Finally, $X_5$ must be equal to $V_1\cup V_i$, while $X_4$ must be equal to $V_i\cup V_k$. As we assumed the sets $X_i$ to be flats, $X_1$ could also include one or two elements of $V_j$. It cannot contain three elements of $V_j$, as $V_i\cup V_j\cup V_k$ is a relaxed circuit-hyperplane in $(\kin{5}^-)^*$ and thus has closure equal to the entire ground set. This gives the following three subcases, which differ only in elements of $X_1$. As before, in each subcase we will begin by evaluating terms of the left-hand and right-hand sides of inequality $5$, then show that the left-hand side must be lower in value, meaning that the inequality holds.

\begin{itemize}[Subcase 1.3a:] 
\item $ X_1=V_i\cup V_k$\\
$X_2=V_j$\\
$X_3=V_2\cup Z$ where $Z\subseteq E(M)$\\
$X_4=V_1\cup V_k$\\
$X_5= V_1\cup V_i$
\end{itemize}

Note that $\ol{X_1\cup X_2}=V_1\cup V_2$ and $\ol{X_2\cup X_4\cup X_5}=\ol{V_1\cup V_i\cup V_j\cup V_k}=V_2$.
\begin{equation*}
\begin{split}
LHS & = r^*(X_3)+r^*(X_4)+r^*(X_5)+r^*(X_1\cup X_2)+r^*(X_1\cup X_3\cup X_5)\\&\qquad+r^*(X_2\cup X_3\cup X_4)+r^*(X_2\cup X_4\cup X_5) \\
& = r^*(X_3)+r^*(X_1\cup X_3\cup X_5)+r^*(X_2\cup X_3\cup X_4) \\
&\qquad +|X_4|+|X_5|+|X_1\cup X_2|+|X_2\cup X_4\cup X_5|\\
&\qquad+ r(\ol{X_4})+r(\ol{X_5}) + r(\ol{X_1 \cup X_2})+ r(\ol{X_2 \cup X_4\cup X_5})-4r(M)\\
& = r^*(X_3)+r^*(X_1\cup X_3\cup X_5)+r^*(X_2\cup X_3\cup X_4)-4r(M)\\
&\qquad+6+6+9+12+5+5+5+2\\
& = r^*(X_3)+r^*(X_1\cup X_3\cup X_5)+r^*(X_2\cup X_3\cup X_4)-4r(M)+50\\
\end{split}
\end{equation*}
\begin{equation*}
\begin{split}
RHS & = r^*(X_1\cup X_3)+r^*(X_1\cup X_5)+r^*(X_2\cup X_3)+r^*(X_2\cup X_4)\\
&\qquad+r^*(X_2\cup X_5)+r^*(X_3\cup X_4)+r^*(X_4\cup X_5) \\
& = r^*(X_1\cup X_3)+r^*(X_2\cup X_3)+r^*(X_3\cup X_4)+|X_1\cup X_5|\\
&\qquad+|X_2\cup X_4|+|X_2\cup X_5|+|X_4\cup X_5|+r(\ol{X_1\cup X_5}) \\
&\qquad+r(\ol{X_2\cup X_4})+r(\ol{X_2\cup X_5})+r(\ol{X_4\cup X_5}) -4r(M)\\
& \geq r^*(X_1\cup X_3)+r^*(X_2\cup X_3)+r^*(X_3\cup X_4)-4r(M)\\
&\qquad+9+9+9+9+4+4+4+4 \\
& = r^*(X_1\cup X_3)+r^*(X_2\cup X_3)+r^*(X_3\cup X_4)-4r(M)+52\\
\end{split}
\end{equation*}

We thus must have that
\begin{equation*}
\begin{split}
&\qquad r^*(X_3)+r^*(X_1\cup X_3\cup X_5)+r^*(X_2\cup X_3\cup X_4)\\
&> r^*(X_1\cup X_3)+r^*(X_2\cup X_3)+r^*(X_3\cup X_4)+2
\end{split}
\end{equation*}
Suppose $Z=\varnothing$. Then we have $2+9+9>8+5+8+2$ which is untrue. If we increase the size of $Z$, the same argument as in subcase 1a shows that we still cannot have a bad family.

\begin{itemize}[Subcase 1.3b:] 
\item $ X_1=V_i\cup V_k\cup\{a\}$ where $a\in V_j$\\
$X_2=V_j$\\
$X_3=V_2\cup Z$ where $Z\subseteq E(M)$\\
$X_4=V_1\cup V_k$\\
$X_5= V_1\cup V_i$
\end{itemize}

Compared to subcase 3a, $|X_1\cup X_5|$ increases by one, and the same argument still holds.

\begin{itemize}[Subcase 1.3c:] 
\item $ X_1=V_i\cup V_k\cup\{a,b\}$ where $a,b\in V_j$\\
$X_2=V_j$\\
$X_3=V_2\cup Z$ where $Z\subseteq E(M)$\\
$X_4=V_1\cup V_k$\\
$X_5= V_1\cup V_i$
\end{itemize}

Compared to subcase 3a, $|X_1\cup X_5|$ increases by two, while $r(\ol{X_1\cup X_5})$ falls by one. The right-hand side of the inequality increases by one, and the same argument as before holds.

\pagebreak

{\bf Subcase 1.4.}

\begin{figure}[ht]
\centering
\begin{tikzpicture}[thin,line join=round]
\node[ellipse, draw] (x1) at (0,3.5) {$X_1$};
\node[ellipse, draw] (x2) at (0,0) {$X_2$};
\node[ellipse, draw] (x3) at (3.5,0) {$X_3$};
\node[ellipse, draw] (x4) at (0,-3.5) {$X_4$};
\node[ellipse, draw] (x5) at (-3.5,0) {$X_5$};
\draw (x1) -- node[rotate=45, above] {$V_1 \cup V_i \cup V_k$} (x5) -- node[rotate=-45, below] {$V_1 \cup V_i \cup V_j$} (x4) -- node[rotate=-90, above] {$V_1 \cup V_j \cup V_k$} (x2) -- (x5);
\draw[dashed] (x1) -- (x2);
\end{tikzpicture}
\end{figure}

Again, we have by assumption that $V_2\subseteq X_3$. $X_3$ may contain other elements of the ground set of $M$, thus any terms in the inequality involving $X_3$ will not be immediately evaluated. Recall that $X_1\cup X_2=V_3\cup V_4\cup V_5$. This means that $V_1$ can only be a subset of $X_4$ or $X_5$. In order for the three circuit-hyperplanes to lie as indicated, both $X_4$ and $X_5$ must contain $V_1$. As $X_2\cup X_4$ cannot contain $V_i$, $X_5$ must be equal to $V_1\cup V_i$. Likewise, $X_2$ must contain $V_k$. This forces $X_4$ to contain $V_j$, and $X_1$ to contain $V_i$. Finally, in order for $X_1\cup X_2=V_3\cup V_4\cup V_5$, it must be that $X_2$ is equal to $V_j\cup V_k$.

$X_2$ must be equal to $V_j\cup V_k$, as this forms the common values of the two circuit-hyperplanes involving $X_2$, excluding $V_1$. Finally, $X_5$ must be equal to $V_1\cup V_i$, while $X_4$ must contain $V_k$. As we assumed the sets $X_i$ to be flats, $X_4$ could also include one or two elements of $V_j$. It cannot contain all three elements, as $V_i\cup V_j\cup V_k$ is a relaxed circuit-hyperplane, and thus has closure equal to the entire ground set. This gives the following three subcases, which differ only in elements of $X_4$. As before, in each subcase we will begin by evaluating terms of the left-hand and right-hand sides of inequality $5$, then show that the left-hand side must be lower in value, meaning that the inequality holds.

\begin{itemize}[\hspace{20 mm}]
\item $ X_1=V_i\cup V_k$\\
$X_2=V_j\cup V_k$\\
$X_3=V_2\cup Z$ where $Z\subseteq E(M)$\\
$X_4=V_1\cup V_j$\\
$X_5= V_1\cup Vi$
\end{itemize}

Note that $\ol{X_1\cup X_2}=V_1\cup V_2$ and $\ol{X_2\cup X_4\cup X_5}=\ol{V_1\cup V_i\cup V_j\cup V_k}=V_2$.
\begin{equation*}
\begin{split}
LHS & = r^*(X_3)+r^*(X_4)+r^*(X_5)+r^*(X_1\cup X_2)+r^*(X_1\cup X_3\cup X_5)\\&\qquad+r^*(X_2\cup X_3\cup X_4)+r^*(X_2\cup X_4\cup X_5) \\
& = r^*(X_3)+r^*(X_1\cup X_3\cup X_5)+r^*(X_2\cup X_3\cup X_4) \\
&\qquad+|X_4|+|X_5|+|X_1\cup X_2|+|X_2\cup X_4\cup X_5|\\
&\qquad + r(\ol{X_4})+r(\ol{X_5}) + r(\ol{X_1 \cup X_2})+ r(\ol{X_2 \cup X_4\cup X_5})-4r(M)\\
& = r^*(X_3)+r^*(X_1\cup X_3\cup X_5)+r^*(X_2\cup X_3\cup X_4)-4r(M)\\
&\qquad+6+6+9+12+5+5+5+2 \\
& = r^*(X_3)+r^*(X_1\cup X_3\cup X_5)+r^*(X_2\cup X_3\cup X_4)-4r(M)+50\\
\end{split}
\end{equation*}
\begin{equation*}
\begin{split}
RHS & = r^*(X_1\cup X_3)+r^*(X_1\cup X_5)+r^*(X_2\cup X_3)+r^*(X_2\cup X_4)\\
&\qquad+r^*(X_2\cup X_5)+r^*(X_3\cup X_4)+r^*(X_4\cup X_5) \\
& = r^*(X_1\cup X_3)+r^*(X_2\cup X_3)+r^*(X_3\cup X_4)+|X_1\cup X_5| \\
&\qquad+|X_2\cup X_4|+|X_2\cup X_5|+|X_4\cup X_5|+r(\ol{X_1\cup X_5})\\
&\qquad+r(\ol{X_2\cup X_4})+r(\ol{X_2\cup X_5})+r(\ol{X_4\cup X_5})-4r(M) \\
& \geq r^*(X_1\cup X_3)+r^*(X_2\cup X_3)+r^*(X_3\cup X_4)-4r(M)\\
&\qquad+9+9+9+9+4+4+2+4\\
& = r^*(X_1\cup X_3)+r^*(X_2\cup X_3)+r^*(X_3\cup X_4)-4r(M)+50\\
\end{split}
\end{equation*}

We thus must have that 
\begin{equation*}
\begin{split}
& \qquad r^*(X_3)+r^*(X_1\cup X_3\cup X_5)+r^*(X_2\cup X_3\cup X_4)\\
& > r^*(X_1\cup X_3)+r^*(X_2\cup X_3)+r^*(X_3\cup X_4)
\end{split}
\end{equation*}
Suppose $Z=\varnothing$. Then we have $2+9+9>8+8+8$ which is untrue. If we increase the size of $Z$, the same argument as in subcase 1a shows that we still cannot have a bad family.
\end{case1}

\newtheorem*{case3}{Case 2}
\begin{case3}
$X_4\cap V_2\neq\varnothing$\\
As in Case 1, we will consider which term on the left-hand side dual inequality must be $V_3\cup V_4\cup V_5$. As this term does not contain any elements of $V_2$ and these elements are contained in $X_4$, the term cannot involve $X_4$. We now have the following representation of the locations remaining as possibilities for all necessary circuit-hyperplanes, where the dashed line again indicates a term which falls on the left-hand side of the inequality.

\begin{figure}[ht]
\centering
\begin{tikzpicture}[thin,line join=round]
\node[ellipse, draw] (x1) at (0,3.5) {$X_1$};
\node[ellipse, draw] (x2) at (0,0) {$X_2$};
\node[ellipse, draw] (x3) at (3.5,0) {$X_3$};
\node[ellipse, draw] (x4) at (0,-3.5) {$X_4$};
\node[ellipse, draw] (x5) at (-3.5,0) {$X_5$};
\draw (x1) -- (x3) -- (x2) -- (x5) -- (x1);
\draw[dashed] (x1) -- (x2);
\end{tikzpicture}
\end{figure}

If it $V_3\cup V_4\cup V_5$ is equal to $X_3$ or $X_5$ then any term involving these sets cannot be one of the circuit-hyperplanes $V_1\cup X_4\cup V_5$, $V_1\cup V_3\cup V_5$. When $V_3\cup V_4\cup V_5=X_3$, the possible terms left are $X_1\cup X_5$ and $X_2\cup X_5$. When $V_3\cup V_4\cup V_5=X_5$, the possible terms left are $X_1\cup X_3$ and $X_2\cup X_3$. In both cases there are not enough terms on the right-hand side for the three circuit-hyperplanes. Thus $V_3\cup V_4\cup V_5=X_1\cup X_2$. As $X_3$ and $X_5$ can be switched with no effect on the inequality, we now have, up to symmetry, two possibilities of where the circuit-hyperplanes lie, indicated by the bold lines in the following diagrams.

\pagebreak

\begin{figure}
\centering
\subfloat[Subcase 1]{
\begin{tikzpicture}[thin,line join=round]
\node[ellipse, draw] (x1) at (0,2.5) {$X_1$};
\node[ellipse, draw] (x2) at (0,0) {$X_2$};
\node[ellipse, draw] (x3) at (2.5,0) {$X_3$};
\node[ellipse, draw] (x4) at (0,-2.5) {$X_4$};
\node[ellipse, draw] (x5) at (-2.5,0) {$X_5$};
\draw[ultra thick] (x1) -- (x3) -- (x2) -- (x5);
\draw (x1) -- (x5);
\draw[dashed] (x1) -- (x2);
\end{tikzpicture}
}
\qquad
\subfloat[Subcase 2]{
\begin{tikzpicture}[thin,line join=round]
\node[ellipse, draw] (x1) at (0,2.5) {$X_1$};
\node[ellipse, draw] (x2) at (0,0) {$X_2$};
\node[ellipse, draw] (x3) at (2.5,0) {$X_3$};
\node[ellipse, draw] (x4) at (0,-2.5) {$X_4$};
\node[ellipse, draw] (x5) at (-2.5,0) {$X_5$};
\draw[ultra thick] (x5) -- (x1) -- (x3) -- (x2);
\draw (x2) -- (x5);
\draw[dashed] (x1) -- (x2);
\end{tikzpicture}
}
\end{figure}

We will again consider both of these subcases, with circuit-hyperplanes lying as shown in the diagrams, which give multiple options within each subcase for the choice of sets $X_1,\ldots,X_5$.

{\bf Subcase 2.1.}

\begin{figure}[ht]
\centering
\begin{tikzpicture}[thin,line join=round]
\node[ellipse, draw] (x1) at (0,3.5) {$X_1$};
\node[ellipse, draw] (x2) at (0,0) {$X_2$};
\node[ellipse, draw] (x3) at (3.5,0) {$X_3$};
\node[ellipse, draw] (x4) at (0,-3.5) {$X_4$};
\node[ellipse, draw] (x5) at (-3.5,0) {$X_5$};
\draw (x1) -- node[rotate=-45,above] {$V_1 \cup V_j \cup V_k$} (x3) -- node[above] {$V_1 \cup V_i \cup V_j$} (x2) -- node[above] {$V_1 \cup V_i \cup V_k$} (x5) -- node[rotate=45,above] {} (x1);
\draw[dashed] (x1) -- (x2);
\end{tikzpicture}
\end{figure}

Here, we have assumed that $V_2\subseteq X_4$. $X_4$ may contain other elements of the ground set of $M$, thus any terms in the inequality involving $X_4$ will not be immediately evaluated. Recall that $X_1\cup X_2=V_3\cup V_4\cup V_5$. This means that $V_1$ can only be a subset of $X_3$ or $X_5$. In order for the three circuit-hyperplanes to lie as indicated, both $X_3$ and $X_5$ must contain $V_1$. $X_2$ must be equal to $V_i$, as this forms the common values of the two circuit-hyperplanes involving $X_2$, excluding $V_1$. This forces $X_1$ to be equal to $V_j\cup V_k$. Note that $X_1$ cannot contain any elements of $V_i$ as $X_1\cup X_3$ does not. Also, $X_3$ must contain $V_j$ and $X_5$ must contain $V_k$. $X_3$ must be equal to $V_1\cup V_j$. As we assumed the sets $X_i$ to be flats, $X_5$ could also include one or three elements of $V_i$. $X_5$ cannot contain exactly two elements of $V_i$ as $V_1\cup V_k\cup V_i$ is a circuit-hyperplane in the dual.This gives the following three subcases, which differ only in elements of $X_5$.

\begin{itemize}[Subcase 2.1a:] 
\item $ X_1=V_j\cup V_k$ \\
$X_2=V_i$\\
$X_3=V_1\cup V_j$\\
$X_4=V_2\cup Z$ where $Z\subseteq E(M)$\\
$X_5= V_1\cup V_k$
\end{itemize}

Note that $\ol{X_1\cup X_2}=V_1\cup V_2$ and $\ol{X_1\cup X_3\cup X_5}=\ol{V_1\cup V_j\cup V_k}=V_2\cup V_i$.
\begin{equation*}
\begin{split}
LHS & = r^*(X_3)+r^*(X_4)+r^*(X_5)+r^*(X_1\cup X_2)+r^*(X_1\cup X_3\cup X_5)\\
&\qquad+r^*(X_2\cup X_3\cup X_4)+r^*(X_2\cup X_4\cup X_5) \\
& = r^*(X_4)+r^*(X_2\cup X_3\cup X_4)+r^*(X_2\cup X_4\cup X_5) \\
&\qquad +|X_3|+|X_5|+|X_1\cup X_2|+|X_1\cup X_3\cup X_5|\\
&\qquad+r(\ol{X_3})+r(\ol{X_5})+r(\ol{X_1\cup X_2})+r(\ol{X_1\cup X_3\cup X_5})-4r(M) \\
& = r^*(X_4)+r^*(X_2\cup X_3\cup X_4)+r^*(X_2\cup X_4\cup X_5)-4r(M)\\
&\qquad+6+6+9+9+5+5+5+4\\
& = r^*(X_4)+r^*(X_2\cup X_3\cup X_4)+r^*(X_2\cup X_4\cup X_5)-4r(M)+49\\
\end{split}
\end{equation*}
\begin{equation*}
\begin{split}
RHS & = r^*(X_1\cup X_3)+r^*(X_1\cup X_5)+r^*(X_2\cup X_3)+r^*(X_2\cup X_4)\\
&\qquad+r^*(X_2\cup X_5)+r^*(X_3\cup X_4)+r^*(X_4\cup X_5) \\
& = r^*(X_2\cup X_4)+r^*(X_3\cup X_4)+r^*(X_4\cup X_5)+|X_1\cup X_3| \\
&\qquad+|X_1\cup X_5|+|X_2\cup X_3|+|X_2\cup X_5|+r(\ol{X_1\cup X_3})\\
&\qquad +r(\ol{X_1\cup X_5})+r(\ol{X_2\cup X_3})+r(\ol{X_2\cup X_5})-4r(M) \\
& = r^*(X_2\cup X_4)+r^*(X_3\cup X_4)+r^*(X_4\cup X_5)-4r(M)\\
&\qquad+9+9+9+9+4+4+4+4 \\
& = r^*(X_2\cup X_4)+r^*(X_3\cup X_4)+r^*(X_4\cup X_5)-4r(M)+52\\
\end{split}
\end{equation*}

We thus must have that 
\begin{equation*}
\begin{split}
& \qquad r^*(X_4)+r^*(X_2\cup X_3\cup X_4)+r^*(X_2\cup X_4\cup X_5)\\
&> r^*(X_2\cup X_4)+r^*(X_3\cup X_4)+r^*(X_4\cup X_5)+3
\end{split}
\end{equation*}
Suppose $Z=\varnothing$. Then we have that $2+8+8>5+8+8+3$ which is untrue. Now suppose $Z$ is non-empty. $X_2\cup X_3\cup X_4=V_1\cup V_2\cup V_i\cup V_j\cup Z$ and $\ol{X_2\cup X_3\cup X_4}=V_k-Z$, so if $|X_2\cup X_3\cup X_4|$ increases, then $Z\subseteq V_k$, and $r(\ol{X_2\cup X_3\cup X_4})$ must decrease by the same amount. Thus $r^*(X_2\cup X_3\cup X_4)$ remains unchanged. Likewise, $r^*(X_2\cup X_4\cup X_5)$ cannot change, as $X_2\cup X_4\cup X_5=V_1\cup V_2\cup V_i\cup V_j\cup Z$ and has complement $V_k-Z$. Finally, $r^*(X_4)$ must be no greater than $r^*(X_4\cup X_5)$. We thus do not have a bad family.

\begin{itemize}[Subcase 2.1b:] 
\item $ X_1=V_j\cup V_k$ \\
$X_2=V_i$\\
$X_3=V_1\cup V_j$\\
$X_4=V_2\cup Z$ where $Z\subseteq E(M)$\\
$X_5= V_1\cup V_k\cup\{a\}$ where $a\in V_i$
\end{itemize}

On the left-hand side, $|X_5|$ and $|X_1\cup X_3\cup X_5|$ both increase in size by two. On the right-hand side, $|X_1\cup X_5|$ increases in size by one. The same argument as in Subcase 1a thus shows there can be no bad family regardless of $Z$.

\begin{itemize}[Subcase 2.1c:] 
\item $ X_1=V_j\cup V_k$ \\
$X_2=V_i$\\
$X_3=V_1\cup V_j$\\
$X_4=V_2\cup Z$ where $Z\subseteq E(M)$\\
$X_5= V_1\cup V_k\cup\{a,b\}$ where $a,b\in V_i$
\end{itemize}

On the left-hand side, $|X_5|$ and $|X_1\cup X_3\cup X_5|$ both increase in size by two, while $r(X_1\cup X_3\cup X_5)$ falls by one. On the right-hand side, $|X_1\cup X_5|$ increases in size by two while $r(X_1\cup X_5)$ falls by one. We still have no bad family.

\pagebreak

{\bf Subcase 2.2.}

\begin{figure}[ht]
\centering
\begin{tikzpicture}[thin,line join=round]
\node[ellipse, draw] (x1) at (0,3.5) {$X_1$};
\node[ellipse, draw] (x2) at (0,0) {$X_2$};
\node[ellipse, draw] (x3) at (3.5,0) {$X_3$};
\node[ellipse, draw] (x4) at (0,-3.5) {$X_4$};
\node[ellipse, draw] (x5) at (-3.5,0) {$X_5$};
\draw (x1) -- node[rotate=-45,above] {$V_1 \cup V_i \cup V_j$} (x3) -- node[above] {$V_1 \cup V_j \cup V_k$} (x2) -- node[above] {} (x5) -- node[rotate=45,above] {$V_1 \cup V_i \cup V_k$} (x1);
\draw[dashed] (x1) -- (x2);
\end{tikzpicture}
\end{figure}

Again, we have assumed that $V_2\subseteq X_4$. $X_4$ may contain other elements of the ground set of $M$, thus any terms in the inequality involving $X_4$ will not be immediately evaluated. Recall that $X_1\cup X_2=V_3\cup V_4\cup V_5$. This means that $V_1$ can only be a subset of $X_3$ or $X_5$. In order for the three circuit-hyperplanes to lie as indicated, both $X_3$ and $X_5$ must contain $V_1$. $X_1$ must be equal to $V_i$, as this forms the common values of the two circuit-hyperplanes involving $X_2$, excluding $V_1$. This forces $X_2$ to be equal to $V_j\cup V_k$, forces $X_3$ to be equal to $V_1\cup V_j$, and forces $X_5$ to contain $V_1\cup V_k$. As we assumed the sets $X_i$ to be flats, $X_5$ could also include one or three elements of $V_i$. $X_5$ cannot contain exactly two elements of $V_i$ as $V_1\cup V_k\cup V_i$ is a circuit-hyperplane in the dual.This gives the following three subcases, which differ only in elements of $X_5$.

\begin{itemize}[Subcase 2.2a:] 
\item $ X_1=V_i$ \\
$X_2=V_j\cup V_k$\\
$X_3=V_1\cup V_j$\\
$X_4=V_2\cup Z$ where $Z\subseteq E(M)$\\
$X_5= V_1\cup V_k$
\end{itemize}

Note that $\ol{X_1\cup X_2}=V_1\cup V_2$ and $\ol{X_1\cup X_3\cup X_5}=\ol{V_1\cup V_i\cup V_j\cup V_k}=V_2$.
\begin{equation*}
\begin{split}
LHS & = r^*(X_3)+r^*(X_4)+r^*(X_5)+r^*(X_1\cup X_2)+r^*(X_1\cup X_3\cup X_5)\\&
\qquad+r^*(X_2\cup X_3\cup X_4)+r^*(X_2\cup X_4\cup X_5) \\
& = r^*(X_4)+r^*(X_2\cup X_3\cup X_4)+r^*(X_2\cup X_4\cup X_5) \\
&\qquad +|X_3|+|X_5|+|X_1\cup X_2|+|X_1\cup X_3\cup X_5|\\
&\qquad+r(\ol{X_3})+r(\ol{X_5})+r(\ol{X_1\cup X_2})+r(\ol{X_1\cup X_3\cup X_5})-4r(M) \\
& = r^*(X_4)+r^*(X_2\cup X_3\cup X_4)+r^*(X_2\cup X_4\cup X_5)-4r(M)\\
&\qquad+6+6+9+12+5+5+5+2\\
& = r^*(X_4)+r^*(X_2\cup X_3\cup X_4)+r^*(X_2\cup X_4\cup X_5)-4r(M)+50\\
\end{split}
\end{equation*}
\begin{equation*}
\begin{split}
RHS & = r^*(X_1\cup X_3)+r^*(X_1\cup X_5)+r^*(X_2\cup X_3)+r^*(X_2\cup X_4)\\
&\qquad+r^*(X_2\cup X_5)+r^*(X_3\cup X_4)+r^*(X_4\cup X_5) \\
& = r^*(X_2\cup X_4)+r^*(X_3\cup X_4)+r^*(X_4\cup X_5)+|X_1\cup X_3| \\
&\qquad +|X_1\cup X_5|+|X_2\cup X_3|+|X_2\cup X_5|+r(\ol{X_1\cup X_3})\\
&\qquad+r(\ol{X_1\cup X_5})+r(\ol{X_2\cup X_3})+r(\ol{X_2\cup X_5})-4r(M) \\
& = r^*(X_2\cup X_4)+r^*(X_3\cup X_4)+r^*(X_4\cup X_5)-4r(M)\\
&\qquad+9+9+9+9+4+4+4+4 \\
& = r^*(X_2\cup X_4)+r^*(X_3\cup X_4)+r^*(X_4\cup X_5)-4r(M)+52\\
\end{split}
\end{equation*}

We thus must have that 
\begin{equation*}
\begin{split}
&\qquad r^*(X_4)+r^*(X_2\cup X_3\cup X_4)+r^*(X_2\cup X_4\cup X_5)\\
& > r^*(X_2\cup X_4)+r^*(X_3\cup X_4)+r^*(X_4\cup X_5)+2
\end{split}
\end{equation*}
Suppose $Z=\varnothing$. Then we have that $2+8+8>8+8+8+2$ which is untrue. Now suppose $Z$ is non-empty. The same argument from Subcase 1a shows we still cannot have a bad family.

\begin{itemize}[Subcase 2.2b:] 
\item $ X_1=V_i$ \\
$X_2=V_j\cup V_k$\\
$X_3=V_1\cup V_j$\\
$X_4=V_2\cup Z$ where $Z\subseteq E(M)$\\
$X_5= V_1\cup V_k\cup \{a\}$ where $a\in V_i$
\end{itemize}

On the left-hand side, $|X_5|$ increases by one. On the right-hand side, $|X_2\cup X_5|$ increases by one. We still have no bad family.

\begin{itemize}[Subcase 2.2c:] 
\item $ X_1=V_i$ \\
$X_2=V_j\cup V_k$\\
$X_3=V_1\cup V_j$\\\
$X_4=V_2\cup Z$ where $Z\subseteq E(M)$\\
$X_5= V_1\cup V_k\cup \{a,b\}$ where $a,b\in V_i$
\end{itemize}

Both $|X_5|$ and $|X_2\cup X_5|$ increase by two, while $r(\ol{X_2\cup X_5})$ falls by one. We still have no bad family.

\end{case3}

\newtheorem*{case4}{Case 3} 
\begin{case4}
$X_2\cap V_2\neq\varnothing$\\
Again consider which term on the left-hand side dual inequality must be $V_3\cup V_4\cup V_5$. As this term does not contain any elements of $V_2$ and these elements are contained in $X_2$, the term cannot involve $X_2$. We have the following representation of remaining possible locations for the necessary circuit-hyperplanes.

\begin{figure}[ht]
\centering
\begin{tikzpicture}[thin,line join=round]
\node[ellipse, draw] (x1) at (0,3.5) {$X_1$};
\node[ellipse, draw] (x2) at (0,0) {$X_2$};
\node[ellipse, draw] (x3) at (3.5,0) {$X_3$};
\node[ellipse, draw] (x4) at (0,-3.5) {$X_4$};
\node[ellipse, draw] (x5) at (-3.5,0) {$X_5$};
\draw (x1) -- (x3) -- (x4) -- (x5) -- (x1);
\end{tikzpicture}
\end{figure}

If $V_3\cup V_4\cup V_5$ is equal to $X_3$, $X_4$, or $X_5$ then any term involving these sets cannot be one of the circuit-hyperplanes $V_1\cup V_3\cup V_4$, $V_1\cup X_4\cup V_5$, $V_1\cup V_3\cup V_5$. This will not leave us with two terms on the right-hand side which could be the circuit-hyperplanes, but three are needed. If it is equal to $X_1\cup X_3\cup X_5$, then only $X_4$ can contain elements of $V_1$. As $V_1$ appears in all of the three circuit-hyperplanes needed on the right-hand side of the inequality, and there are only two terms available using $X_4$, we do not have enough terms left for the circuit-hyperplanes. This covers all terms on the left-hand side, and thus $X_2\cup V_2=\varnothing.$
\end{case4}

\newtheorem*{case5}{Case 4}
\begin{case5}
$X_1\cap V_2\neq\varnothing$\\
Consider which term on the left-hand side dual inquality must be $V_3\cup V_4\cup V_5$. It cannot be any of the terms involving $X_1$ . We have the following representation of the remaining possible locations for the necessary circuit-hyperplanes.

\begin{figure}[ht]
\centering
\begin{tikzpicture}[thin,line join=round]
\node[ellipse, draw] (x1) at (0,3.5) {$X_1$};
\node[ellipse, draw] (x2) at (0,0) {$X_2$};
\node[ellipse, draw] (x3) at (3.5,0) {$X_3$};
\node[ellipse, draw] (x4) at (0,-3.5) {$X_4$};
\node[ellipse, draw] (x5) at (-3.5,0) {$X_5$};
\draw (x4) -- (x5) -- (x2) -- (x3) -- (x4) -- (x2);
\end{tikzpicture}
\end{figure}

If $V_3\cup V_4\cup V_5$ is $X_1\cup X_3\cup X_5$, $X_2\cup X_3\cup X_4$ or $X_2\cup X_4\cup X_5$, only two sets could include elements of $V_1$. We will not have enough terms left on the right-hand side to be the circuit-hyperplanes $V_1\cup V_3\cup V_4$, $V_1\cup X_4\cup V_5$, $V_1\cup V_3\cup V_5$. If $V_3\cup V_4\cup V_5=X_4$, the circuit-hyerplanes cannot be terms using $X_4$, so the only possible edges remaining are $X_2\cup X_3$ and $X_2\cup X_5$ -- one edge less than is necessary. Thus $V_3\cup V_4\cup V_5$ must be either $X_3$ or $X_5$. These two possibilities are symmetric, so assume $V_3\cup V_4\cup V_5=X_3$. The necessary circuit-hyperplanes now cannot use $X_3$ or $X_1$. This gives only one possible assignment to the terms in the inequality, as shown in the following diagram. 

\pagebreak

\begin{figure}[ht]
\centering
\begin{tikzpicture}[thin,line join=round]
\node[ellipse, draw] (x1) at (0,3.5) {$X_1$};
\node[ellipse, draw] (x2) at (0,0) {$X_2$};
\node[ellipse, draw] (x3) at (3.5,0) {$X_3$};
\node[ellipse, draw] (x4) at (0,-3.5) {$X_4$};
\node[ellipse, draw] (x5) at (-3.5,0) {$X_5$};
\draw (x4) -- node[rotate=-45,below] {$V_1 \cup V_j \cup V_k$} (x5) -- node[above] {$V_1 \cup V_i \cup V_j$} (x2) -- node[above] {} (x3) -- node[rotate=45,below] {} (x4) -- node[rotate=-90,above] {$V_1 \cup V_i \cup V_k$} (x2);
\end{tikzpicture}
\end{figure}

We have that $X_1=V_2\cup Z$ where $Z$ is a possibly empty subset of the ground set and that $X_3=V_3\cup V_4\cup V_5$. $X_2$ must contain $V_i$ but cannot contain $V_j$ or $V_k$, $X_4$ must contain $V_k$ but cannot contain $V_i$ or $V_j$, and $X_5$ must contain $V_j$ but cannot contain $V_i$ or $V_k$. All three of these sets could also contain one of three elements from $V_1$, such that whenever we take the union of two of the sets, the union contains $V_1$. A set cannot contain two elements of $V_1$ as this would mean it was not a flat.

Note that if one of the sets $X_2$, $X_4$, or $X_5$ contains no elements of $V_1$, both the other two sets must contain all elements of $V_1$. If one of the sets contains one element of $V_1$, the other two sets must again contain all elements of $V_1$. The third possibility is that all three sets contain $V_1$.

\begin{itemize}[Subcase 4.1a:] 
\item $ X_1=V_2\cup Z$ where $Z\subseteq E(M)$ \\
$X_2= V_i\cup V_1$\\
$X_3= V_3\cup V_4\cup V_5$\\
$X_4= V_k$\\
$X_5= V_j\cup V_1$\\
\end{itemize}
Note that $\ol{X_2\cup X_3\cup X_4}=\ol{X_2\cup X_4\cup X_5}=\ol{V_1\cup V_i\cup V_j\cup V_k}=V_2$.
\begin{equation*}
\begin{split}
LHS & = r^*(X_3)+r^*(X_4)+r^*(X_5)+r^*(X_1\cup X_2)+r^*(X_1\cup X_3\cup X_5)\\
&\qquad+r^*(X_2\cup X_3\cup X_4)+r^*(X_2\cup X_4\cup X_5) \\
& = r^*(X_1\cup X_2)+r^*(X_1\cup X_3\cup X_5)+ |X_3|+|X_4|+|X_5| \\
&\qquad +|X_2\cup X_3\cup X_4|+|X_2\cup X_4\cup X_5|+ r(\ol{X_3})+r(\ol{X_4})\\
&\qquad+r(\ol{X_5})+r(\ol{X_2\cup X_3\cup X_4})+r(\ol{X_2\cup X_4\cup X_5})-5r(M) \\
& = r^*(X_1\cup X_2)+r^*(X_1\cup X_3\cup X_5)-5r(M)+9+3+6\\
&\qquad+12+12+5+5+5+2+2\\
& = r^*(X_1\cup X_2)+r^*(X_1\cup X_3\cup X_5)-5r(M)+61\\
\end{split}
\end{equation*}
\begin{equation*}
\begin{split}
RHS & = r^*(X_1\cup X_3)+r^*(X_1\cup X_5)+r^*(X_2\cup X_3)+r^*(X_2\cup X_4)\\
&\qquad+r^*(X_2\cup X_5)+r^*(X_3\cup X_4)+r^*(X_4\cup X_5) \\
& = r^*(X_1\cup X_3)+r^*(X_1\cup X_5)+|X_2\cup X_3|+|X_2\cup X_4|+|X_2\cup X_5|\\
&\qquad +|X_3\cup X_4|+|X_4\cup X_5| +r(\ol{X_2\cup X_3})+r(\ol{X_2\cup X_4})\\
&\qquad+r(\ol{X_2\cup X_5})+r(\ol{X_3\cup X_4})+r(\ol{X_4\cup X_5})-5r(M) \\
& = r^*(X_1\cup X_3)+r^*(X_1\cup X_5)-5r(M)+ 12+9+9\\
&\qquad+9+9+2+4+4+5+4\\
& =r^*(X_1\cup X_3)+r^*(X_1\cup X_5)-5r(M)+ 67\\
\end{split}
\end{equation*}

We thus must have that 
\begin{equation*}
\begin{split}
&\qquad r^*(X_1\cup X_2)+r^*(X_1\cup X_3\cup X_5)\\
&>r^*(X_1\cup X_3)+r^*(X_1\cup X_5)+6
\end{split}
\end{equation*}
in order for this to be a bad family.

Suppose $Z=\varnothing$. Then we have that $7+9>9+8+6$ which is untrue. Now suppose $Z$ is non-empty. $X_1\cup X_3\cup X_5$ is equal to the entire ground set, so changing $Z$ has no effect on this term and it is thus still spanning. Now take $r^*(X_1\cup X_3)=r^*(V_2\cup V_3\cup V_4\cup V_5\cup Z)$. As $\ol{X_1\cup X_3}=V_1-Z$ is coindependent for any choice of $Z$, $X_1\cup X_3$ is spanning for any choice of $Z$.

Now note that $X_1\cup X_2=V_1\cup V_2\cup V_i\cup Z$ and $\ol{X_1\cup X_2}=(V_j\cup V_k)-Z$. Note that $r(V_j\cup V_k)=4$.
If $Z=\varnothing$, $r(\ol{X_1\cup X_2})=4$. If $Z$ is equal to one element in $V_j\cup V_k$, the cardinality of $X_1\cup X_2$ will increase by one but the rank of $\ol{X_1\cup X_2}$ will be unchanged. This increases $r^*(X_1\cup X_2)$ by one. Likewise, if $Z$ is equal to two elements of $V_j\cup V_k$, $r^*(X_1\cup X_2)$ increases by two. If $Z$ has cardinality greater than or equal to two, $(V_j\cup V_k)-Z$ will be coindependent, making $X_1\cup X_2$ spanning for all such $Z$. This means the left-hand side of the inequality can increase by at most two.

Finally, $|X_1\cup X_5|=|V_1\cup V_2\cup V_j\cup Z|$ and $\ol{X_1\cup X_5}=(V_i\cup V_k)-Z$. Note that $r(V_i\cup V_k)=5$. If $Z$ is equal to one element in $V_i\cup V_k$, the cardinality of $X_1\cup X_5$ will increase by one but the rank of $\ol{X_1\cup X_5}$ will be unchanged. This increases $r^*(X_1\cup X_5)$ by one. When $Z$ is equal to one or more elements in $V_i\cup V_k$, $(V_i\cup V_k)-Z$ is coindependent. This means that $X_1\cup X_5$ is spanning for all such $Z$, and increasing $Z$ further can have no effect. This means the right-hand side of the inequality can only increase in value by at most one with a non-empty $Z$. 

The inequality will thus still not be satisfied for any choice of $Z$.

\begin{itemize}[Subcase 4.1b:] 
\item $ X_1=V_2\cup Z$ where $Z\subseteq E(M)$ \\
$X_2= V_i\cup V_1$\\
$X_3= V_3\cup V_4\cup V_5$\\
$X_4= V_k\cup V_1$\\
$X_5= V_j$
\end{itemize}

Compare this to Subcase 4.1a. $X_4$ now contains $V_1$ and $X_5$ does not, while in 4.1a this was the opposite way around. First consider the left-hand side of the inequality. $r^*(X_4)+r^*(X_5)$ is unchanged, as is $r^*(X_2\cup X_4\cup X_5)$. As $V_1$ is contained in $X_2$, $r^*(X_2\cup X_3\cup X_4)$ and $r^*(X_2\cup X_4\cup X_5)$ are also unchanged. The only possible change from 4.1a is thus in $r^*(X_1\cup X_3\cup X_5)$. Note that $X_1\cup X_3\cup X_5=V_2\cup V_i\cup V_j\cup V_k\cup Z$ and $\ol{X_1\cup X_3\cup X_5}=V_1-Z$. As $V_1-Z$ is coindependent for any value of $Z$, $X_1\cup X_3\cup X_5$ is spanning, as in Subcase 4.1a. The left-hand side of the inequality is thus unchanged from Subcase 4.1a.

Now consider the right-hand side of the inequality. $r^*(X_2\cup X_4)+r^*(X_2\cup X_5)$ remains the same, as does $r^*(X_4\cup X_5)$. This leaves $r^*(X_1\cup X_5)$ and $r^*(X_3\cup X_4)$. First note that $X_3\cup X_4=V_1\cup V_3\cup V_4\cup V_5$ and $\ol{X_3\cup X_4}=V_2$, which is coindependent, and so $X_3\cup X_4$ is still spanning. Now take $X_1\cup X_5=V_2\cup V_j\cup Z$, where $\ol{X_1\cup X_5}=(V_1\cup V_i\cup V_k)-Z$. Suppose $Z=\varnothing$. As $V_i\cup V_k$ is spanning, $r^*(X_1\cup X_5)$ will fall by three in comparision to Subcase 4.1a. As $V_1\cup V_i\cup V_k$ is a dependent set of rank $5$, we can remove at most three elements from it without affecting the rank. Thus we can increase $|X_1\cup X_5|$ by three without affecting $r^*(X_1\cup X_5)$, but, after that, any change in $|X_1\cup X_5|$ is matched by a decrease in $r(\ol{X_1\cup X_5})$, causing $r^*(X_1\cup X_5)$ to remain the same. 

Thus, in comparision to Subcase 4.1a, the inequality is at worst three lower on the right-hand side. As we showed in that subcase that the left-hand side can increase by at most two with a non-empty choice of $Z$, the inequality still cannot be satisfied.

\begin{itemize}[Subcase 4.1c:] 
\item $ X_1=V_2\cup Z$ where $Z\subseteq E(M)$ \\
$X_2= V_i$\\
$X_3= V_3\cup V_4\cup V_5$\\
$X_4= V_k\cup V_1$\\
$X_5= V_j\cup V_1$
\end{itemize}
Compared to subcase 4.1a, $X_5$ now contains $V_1$ and $X_2$ does not. 

Consider the left-hand side of the inequality. Any terms which do not involve $X_2$ or $X_5$ will be unchanged from Subcase 4.1a. $|X_5|$ increases by three, while $r(\ol{X_5})$ remains the same, causing $r^*(X_5)$ to increase by three. $r^*(X_2\cup X_3\cup X_4)$ and $r^*(X_2\cup X_4\cup X_5)$ remain the same, as $V_1\subset X_4$. Finally, take $X_1\cup X_2=V_2\cup V_i\cup Z$. We have that $r^*(V_2\cup V_i)=5$, which is two less than in Subcase 4.1a. As $\ol{X_1\cup X_2}=(V_1\cup V_j\cup V_k)-Z$ is a dependent set of rank $5$, we can increase $|X_1\cup X_2|$ by three without changing $r(\ol{X_1\cup X_2})$. If the cardinality of $Z\subseteq V_1\cup V_j\cup V_k$ is any greater, $\ol{X_1\cup X_2}$ is coindependent, and so $X_1\cup X_2$ is spanning for all such $Z$. Thus, in total, the left-hand side increases in value by at most one.

Now take the right-hand side of the inequality. Note that $X_2\cup X_3$ is still spanning, as $\ol{X_2\cup X_3}=V_2$ is coindependent, and note that $r^*(X_2\cup X_4)$ is unchanged. Also, $r^*(X_2\cup X_5)$ and $r^*(X_4\cup X_5)$ remain the same, as $V_1\subset X_5$. The right-hand side thus does not change in value.

The inequality still does not hold, and we thus have no bad family for any choice of $Z$.

\begin{itemize}[Subcase 4.2a:] 
\item $ X_1=V_2\cup Z$ where $Z\subseteq E(M)$ \\
$X_2= V_i\cup V_1$\\
$X_3= V_3\cup V_4\cup V_5$\\
$X_4= V_k\cup V_1$\\
$X_5= V_j\cup\{a\}$ where $a\in V_1$
\end{itemize}

Note that $\ol{X_2\cup X_3\cup X_4}=\ol{X_2\cup X_4\cup X_5}=\ol{V_1\cup V_i\cup V_j\cup V_k}=V_2$.
\begin{equation*}
\begin{split}
LHS & = r^*(X_3)+r^*(X_4)+r^*(X_5)+r^*(X_1\cup X_2)+r^*(X_1\cup X_3\cup X_5)\\
&\qquad+r^*(X_2\cup X_3\cup X_4)+r^*(X_2\cup X_4\cup X_5) \\
& = r^*(X_1\cup X_2)+r^*(X_1\cup X_3\cup X_5)+ |X_3|+|X_4|+|X_5| \\
&\qquad +|X_2\cup X_3\cup X_4|+|X_2\cup X_4\cup X_5|+ r(\ol{X_3})+r(\ol{X_4})\\
&\qquad+r(\ol{X_5})+r(\ol{X_2\cup X_3\cup X_4})+r(\ol{X_2\cup X_4\cup X_5})-5r(M) \\
& = r^*(X_1\cup X_2)+r^*(X_1\cup X_3\cup X_5)-5r(M)+9+6+4\\
&\qquad+12+12+5+5+5+2+2\\
& = r^*(X_1\cup X_2)+r^*(X_1\cup X_3\cup X_5)-5r(M)+62\\
\end{split}
\end{equation*}
\begin{equation*}
\begin{split}
RHS & = r^*(X_1\cup X_3)+r^*(X_1\cup X_5)+r^*(X_2\cup X_3)+r^*(X_2\cup X_4)\\
&\qquad+r^*(X_2\cup X_5)+r^*(X_3\cup X_4)+r^*(X_4\cup X_5) \\
& = r^*(X_1\cup X_3)+r^*(X_1\cup X_5)+|X_2\cup X_3|+|X_2\cup X_4|+|X_2\cup X_5|\\
&\qquad +|X_3\cup X_4|+|X_4\cup X_5| +r(\ol{X_2\cup X_3})+r(\ol{X_2\cup X_4})\\
&\qquad+r(\ol{X_2\cup X_5})+r(\ol{X_3\cup X_4})+r(\ol{X_4\cup X_5})-5r(M) \\
& = r^*(X_1\cup X_3)+r^*(X_1\cup X_5)-5r(M)+ 12+9+9\\
&\qquad+12+9+2+4+4+2+4\\
& =r^*(X_1\cup X_3)+r^*(X_1\cup X_5)-5r(M)+ 67
\end{split}
\end{equation*}

We thus must have that 
\begin{equation*}
\begin{split}
&\qquad r^*(X_1\cup X_2)+r^*(X_1\cup X_3\cup X_5)\\
&>r^*(X_1\cup X_3)+r^*(X_1\cup X_5)+5
\end{split}
\end{equation*}
in order for this to be a bad family.

Suppose $Z=\varnothing$. We have that $7+9>9+6+5$ which is untrue. Now suppose $Z\neq\varnothing$. We have shown in Subcase 4.1a that $r^*(X_1\cup X_3)$ cannot change, and that $r^*(X_1\cup X_2)$ can increase by at most two. The sets \vect{X}{3} are the same in the current subcase and hence the same facts apply. Note that $X_1\cup X_3\cup X_5$ is equal to the entire ground set, and thus changing $Z$ will have no effect on this term, as in Subcase 4.1a. Finally, note that $X_1\cup X_5=\{a\}\cup V_2\cup V_j\cup Z$ and $\ol{X_1\cup X_5}=(\{b,c\}\cup V_i\cup V_k)-Z$. As $\{b,c\}\cup V_i\cup V_k$ is a dependent set of rank $5$, we can remove at most three elements from it without affecting the rank. Thus we can increase $|X_1\cup X_5|$ by three without affecting $r^*(X_1\cup X_5)$, but, after that, $\ol{X_1\cup X_5}$ is coindependent, causing $r^*(X_1\cup X_5)$ to be spanning for all such $Z$. Thus the right-hand side of the inequality can increase by at most three when we make $Z$ non-empty. The left-hand side can increase by at most two, and so the inequality still cannot be satisfied.

\begin{itemize}[Subcase 4.2b:] 
\item $ X_1=V_2\cup Z$ where $Z\subseteq E(M)$ \\
$X_2= V_i\cup\{a\}$ where $a\in V_1$\\
$X_3= V_3\cup V_4\cup V_5$\\
$X_4= V_k\cup V_1$\\
$X_5= V_j\cup V_1$\\
\end{itemize}

Compared to Subcase 4.2a, $X_2$ contains two less elements of $V_1$ while $X_5$ contains two more. 

Consider the left-hand side of the inequality. Any terms not involving $X_2$ or $X_5$ are unchanged from Subcase 4.2a. $|X_5|$ increases by two in comparision to Subcase 4.2a, while $r(\ol{X_5})$ remains the same, causing $r^*(X_5)$ to increase by two. As $V_1\in X_4$, $r^*(X_2\cup X_3\cup X_4)$ and $r^*(X_2\cup X_4\cup X_5)$ remain the same. Note that $X_1\cup X_3\cup X_5$ is equal to the entire ground set and thus must be spanning, as in Subcase 4.2a. Now note that $X_1\cup X_2=\{a\}\cup V_2\cup V_i\cup Z$ and $\ol{X_1\cup X_2}=(\{b,c\}\cup V_i\cup V_k)-Z$. As $\{b,c\}\cup V_i\cup V_k$ is a dependent set of rank $5$, we can remove at most three elements from it without affecting the rank. Thus we can increase $|X_1\cup X_2|$ by three without affecting $r^*(X_1\cup X_2)$, but, after that, $\ol{X_1\cup X_2}$ is coindependent, causing $r^*(X_1\cup X_2)$ to be spanning for all such $Z$.. The left-hand side can thus increase by at most five in comparision to Subcase 4.2a.

Take the right-hand side of the inequality. $r^*(X_2\cup X_5)$ is unchanged. As $V_1\subset X_4$, $r^*(X_2\cup X_4)$ and $r^*(X_4\cup X_5)$ are also unchanged. $|X_2\cup X_3|$ decreases by two, but $r(\ol{X_2\cup X_3})$ increases by two, meaning that $r^*(X_2\cup X_3)$ is unchanged. Finally, take $X_1\cup X_5=V_1\cup V_2\cup V_j\cup Z$, where $\ol{X_1\cup X_5}=(V_i\cup V_k)-Z$. Suppose that, to begin with, $Z=\varnothing$. In Subcase 4.2a, $r^*(X_1\cup X_5)=6$. As $\ol{X_1\cup X_5}$ is still spanning in the current subcase, $r^*(X_1\cup X_5)$ increases by two due to the increase in $|X_1\cup X_2|$. Now suppose $Z$ is non-empty. As $V_i\cup V_k$ has rank $5$, we can increase $|X_1\cup X_5|$ by one without decreasing $r(\ol{X_1\cup X_5})$. For any $Z\subseteq V_i\cup V_k$ with a cardinality greater than or equal to one, $\ol{X_1\cup X_5}$ is coindependent. This means $X_1\cup X_5$ is spanning for all such $Z$, and so $r^*(X_1\cup X_5)$ can increase by at most one with a non-empty $Z$. Thus, in sum, the right-hand side of the inequality can increase by at most three. We cannot have a bad family, for any choice of $Z$.

\begin{itemize}[Subcase 4.2c:] 
\item $ X_1=V_2\cup Z$ where $Z\subseteq E(M)$ \\
$X_2= V_i\cup V_1$\\
$X_3= V_3\cup V_4\cup V_5$\\
$X_4= V_k\cup \{a\}$ where $a\in V_1$\\
$X_5= V_j\cup V_1$
\end{itemize}

Compared to Subcase 4.2a, $X_4$ contains two less elements of $V_1$ while $X_5$ contains two more. 

On the left-hand side of the inequality, there is no change in value compared to Subcase 4.2a. $r^*(X_4)+r^*(X_5)$ remains the same, as does $r^*(X_2\cup X_4\cup X_5)$. In Subcase 4.2a, $X_1\cup X_3\cup X_5$ was spanning and the same is true after adding additional elements to it. As $V_1\subset X_2$, $r^*(X_2\cup X_4\cup X_4)$ is also unchanged.

Now take the right-hand side of the inequality. $r^*(X_4\cup X_5)$ is unchanged, as are $r^*(X_2\cup X_4)$ and $r^*(X_2\cup X_5)$ since $V_1\subset X_2$. As in Subcase 4.2b, $r^*(X_1\cup X_5)$ can increase by at most three. Finally, take $X_3\cup X_4=\{a\}\cup V_3\cup V_4\cup V_5$. This set is spanning, as in Subcase 4.2a.

We have that, in comparision to Subcase 4.2a, the left-hand side remains the same while the right-hand side increases by at most three. We still have no bad family, for any choice of $Z$.

\begin{itemize}[Subcase 4.3:] 
\item $ X_1=V_2\cup Z$ where $Z\subseteq E(M)$ \\
$X_2= V_i\cup V_1$\\
$X_3= V_3\cup V_4\cup V_5$\\
$X_4= V_k\cup V_1$\\
$X_5= V_j\cup V_1$
\end{itemize}

Note that $\ol{X_2\cup X_3\cup X_4}=\ol{X_2\cup X_4\cup X_5}=\ol{V_1\cup V_i\cup V_j\cup V_k}=V_2$.
\begin{equation*}
\begin{split}
LHS & = r^*(X_3)+r^*(X_4)+r^*(X_5)+r^*(X_1\cup X_2)+r^*(X_1\cup X_3\cup X_5)\\
&\qquad+r^*(X_2\cup X_3\cup X_4)+r^*(X_2\cup X_4\cup X_5) \\
& = r^*(X_1\cup X_2)+r^*(X_1\cup X_3\cup X_5)+ |X_3|+|X_4|+|X_5| \\
&\qquad +|X_2\cup X_3\cup X_4|+|X_2\cup X_4\cup X_5|+ r(\ol{X_3})+r(\ol{X_4})\\
&\qquad+r(\ol{X_5})+r(\ol{X_2\cup X_3\cup X_4})+r(\ol{X_2\cup X_4\cup X_5})-5r(M) \\
& = r^*(X_1\cup X_2)+r^*(X_1\cup X_3\cup X_5)-5r(M)+9+6+6\\
&\qquad+12+12+5+5+5+2+2\\
& = r^*(X_1\cup X_2)+r^*(X_1\cup X_3\cup X_5)-5r(M)+64\\
\end{split}
\end{equation*}
\begin{equation*}
\begin{split}
RHS & = r^*(X_1\cup X_3)+r^*(X_1\cup X_5)+r^*(X_2\cup X_3)+r^*(X_2\cup X_4)\\
&\qquad+r^*(X_2\cup X_5)+r^*(X_3\cup X_4)+r^*(X_4\cup X_5) \\
& = r^*(X_1\cup X_3)+r^*(X_1\cup X_5)+|X_2\cup X_3|+|X_2\cup X_4|+|X_2\cup X_5|\\
&\qquad +|X_3\cup X_4|+|X_4\cup X_5| +r(\ol{X_2\cup X_3})+r(\ol{X_2\cup X_4})\\
&\qquad+r(\ol{X_2\cup X_5})+r(\ol{X_3\cup X_4})+r(\ol{X_4\cup X_5})-5r(M) \\
& = r^*(X_1\cup X_3)+r^*(X_1\cup X_5)-5r(M)+ 12+12+9\\
&\qquad+12+9+2+4+4+2+4\\
& =r^*(X_1\cup X_3)+r^*(X_1\cup X_5)-5r(M)+ 70
\end{split}
\end{equation*}

We thus must have that 
\begin{equation*}
\begin{split}
&\qquad r^*(X_1\cup X_2)+r^*(X_1\cup X_3\cup X_5)\\
&>r^*(X_1\cup X_3)+r^*(X_1\cup X_5)+6
\end{split}
\end{equation*}
in order for this to be a bad family.

Suppose $Z=\varnothing$. Then we have that $7+9>9+8+6$ which is untrue. We have shown in Subcase 4.1a that $r^*(X_1\cup X_3)$ cannot change when $Z$ is non-empty, and that $r^*(X_1\cup X_2)$ can increase by at most two. The sets \vect{X}{3} are the same in the current subcase and hence the same facts apply. Note that $X_1\cup X_3\cup X_5$ is equal to the entire ground set, and thus changing $Z$ will have no effect on this term. Finally, $X_1\cup X_5=V_1\cup V_2\cup V_j\cup Z$ and $\ol{X_1\cup X_5}=(V_i\cup V_k)-Z$. As $r(V_i\cup V_k)=5$, we can remove one element from it without decreasing the rank. This increases $|X_1\cup X_5|$ by one and therefore $r^*(X_1\cup X_5)$. For any $Z\subseteq V_i\cup V_k$ with cardinality one or higher, $\ol{X_1\cup X_5}$ is coindependent, so $X_1\cup X_5$ is spanning for all such $Z$. We thus have that the left-hand side can increase by at most two, while the right-hand side can increase by at most one, giving us no possible bad family.

\end{case5}
\end{proof}

\chapter{A Complexity Theorem}

As yet no method of testing whether a matroid satisfies a Kinser equality has presented itself other than brute force. This leads to the question of whether it is possible to do this in polynomial time. Given the increasing number of terms in each inequality and the lack of bounds on a matroid's possible ground set, this is an important question in terms of the results it is feasible to get -- in particular, whether it would be feasible to construct a matroid similar to that used in Theorem 4.4 and test whether it satisfies inequality $n$ for $n\geq 5$, in order to show that the higher Kinser classes are not dual closed. We give a proof that it would in fact be impossible to test these in polynomial time.

An \emph{oracle machine} consists of a Turing machine with a black box attached, which is referred to as the oracle. Given some question about a particular matroid, inputs are fed to the oracle, which then gives an output answering the question. The time the machine takes to produce an output is given as a function of the number on inputs necessary to answer the question. We wish to know the time an oracle machine would take to answer whether a matroid satisfies Kinser inequality $n$.

\begin{dfn}
\label{spike}
Let $r\geq 4$ and take two distinct $r$-element sets $A=\{a_1,\ldots,a_r\},B=\{b_1,\ldots,b_r\}$. We will define the circuit-hyperplanes of the \emph{rank-$r$ binary spike}, denoted by $Z_r$, on ground set $E=A\cup B$ by its set of circuits. First, define the set of circuit-hyperplanes to be the subsets $\{z_1,\ldots,z_r\}$, where $z_i\in\{a_i,b_i\}$, such that $|\{z_1,\ldots,z_r\}\cap \{b_1,\ldots,b_r\}|$ is even. The non-spanning circuits of $Z_r$ consist of the circuit-hyperplanes as defined above and subsets of $E$ of the form $\{a_i,b_i,a_k,b_k\}$.
\end{dfn}

$Z_r$ can be represented by the following matrix:

\begin{center}
\includegraphics{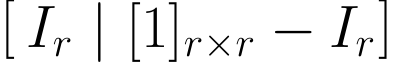}
\end{center}

\begin{lemma}Take an arbitrary rank $r$ binary spike where $r$ is even. If we relax any circuit-hyperplane other than $A$, the resulting matroid violates the Ingleton condition.
\end{lemma}
\begin{proof}
Take a binary spike $Z_r$. Take one of the circuit-hyperplanes of $Z_r$ and call it $Z$, where $Z$ is chosen so that $Z\cap A$ and $Z\cap B$ are non-empty. Define $I\subseteq\{1,\ldots,r \}$ such that $i\in I$ if and only if $a_i\in Z$, and define $J\subseteq\{1,\ldots,r\}$ such that $j\in J$ if and only if $b_i\in Z$. Now let $X_1=\{a_i \ | \ i\in I\}$, $X_2=\{b_j \ | \ j\in J\}$, $X_3=\{b_i \ | \ i\in I\}$, and $X_4=\{a_j \ | \ j\in J\}$. In other words, $X_1$ and $X_2$ consist of the elements in the circuit-hyperplane $Z$ contained in $A$ and $B$ respectively, while $X_3$ and $X_4$ consist of all the remaining elements in $B$ and $A$. Note that $X_2$ contains an even number of elements from $B$ and that $|X_1\cup X_2|=r$, making it a circuit-hyperplane. Relax $X_1\cup X_2$ to get the matroid $Z^-_r$ and evaluate the Ingleton condition:
\begin{align*}
r(X_3)+r(X_4)+r(X_1\cup X_2)+r(X_1\cup X_3\cup X_4)+r(X_2\cup X_3\cup X_4)\\
\leq r(X_1\cup X_3)+r(X_1\cup X_4)+r(X_2\cup X_3)+r(X_2\cup X_4)+r(X_3\cup X_4)
\end{align*} 
The set of non-spanning circuits of $Z_r$ consists of the circuit-hyperplanes as defined above and subsets of $E$ of the form $\{a_i,b_i,a_k,b_k\}$. $X_3$ and $X_4$ do not fit into this category and are thus independent. The ground set of $Z_r^-$ has size $2r$ and $Z$ has size $r$, so $X_3$ and $X_4$ have ranks which sum to $r$. Recall that a leg is a subset $\{a_k,b_k\}$ of the ground set for some $k$. A proper subset of the legs has rank one greater than the number of legs. $X_1\cup X_3=\{a_i\cup b_i \ | \ i\in I\}$ and $X_2\cup X_4=\{a_j\cup b_j \ | \ j\in J \}$ are both collections of legs, the former having $|X_1|=|X_3|$ legs and the latter having $|X_2|=|X_4|$ legs. Thus $r(X_1\cup X_3)=|X_1|+1$ and $r(X_2\cup X_4)=|X_2|+1$. As $Z$ is a circuit-hyperplane, $|X_2|$ is even by definition. This means $X_2\cup X_3$ is a circuit-hyperplane. Recall that $X_2$ and $X_3$ partition $B$. As $r$ is even, $|X_3|=r-|X_2|$ must be even as well. Thus $X_3\cup X_4$ is a circuit-hyperplane, as is $X_1\cup X_4=A$. tNow consider $X_1\cup X_3\cup X_4$. This set properly contains the circuit-hyperplane $X_3\cup X_4$, and all the sets $X_i$ are non-empty. Thus $X_1\cup X_3\cup X_4$ is spanning. $X_2\cup X_3\cup X_4$ also properly contains a circuit-hyperplane, and so is also spanning. Using these calculations we can now evaluate the inequality. 

$$4r\leq (|X_1|+1)+(r-1)+(r-1)+(|X_2|+1)+(r-1)$$ 

Since $|X_1|+|X_2|=r$, this simplifies to $4r\leq 4r-1$ which is untrue. \\ \vect{X}{4} therefore form a bad family. \\
\end{proof}

\begin{thm}
Let $n\geq 4$. There does not exist a polynomial time oracle machine testing Kinser inequality $n$ or its dual.
\end{thm}

\begin{proof}
As proven above, each binary spike $Z_r$ of even rank is representable, therefore satisfies the inequality, while its relaxation $Z^-_r$ does not. This means that in order to test whether a matroid satisfies Kinser inequality $n$ or its dual, the oracle machine must distinguish between each $Z_r$ and $Z^-_r$. Recall $Z^-_r$ can be constructed by relaxing any circuit-hyperplane, which consists of an $r$ element subset $\{z_1,...,z_r\}$ of the ground set $A\cup B$ where $z_i\in\{a_i,b_i\}$ and $|\{z_1,...,z_r\}\cap \{b_1,...,b_r\}|$ is even. Suppose the oracle did not check the rank of the relaxed circuit-hyperplane. This would mean it yields the same result as before the circuit-hyperplane was relaxed, as that is the only subset which changes in rank. Thus the oracle must check the rank of each possible circuit-hyperplane. There are $2^r$ $r$-element sets using one element from each leg, and half of these contain an even number of elements from $\{b_1,...,b_r\}$. The algorithm hence takes at least $2^{r-1}=2^{\frac{E}{2}-1}$ checks, and therefore is exponential in the size of the ground set. As the class of spikes is dual-closed, testing whether the dual of a matroid satisfies Kinser inequality $n$ is also exponential in the size of the ground set.
\end{proof}

\chapter{Excluded minors}

The following theorem was proved by Mayhew, Newman, and Whittle in 2008 \cite{eminors}, settling a conjecture by J. Geelen.

\begin{thm}\label{MNW}
For any infinite field $\mathbb{K}$ and any matroid $N$ representable over $\mathbb{K}$, there is an excluded minor for $\mathbb{K}$-representability that has $N$ as a minor.
\end{thm}
The proof of Theorem \ref{MNW} constructed an excluded minor which contained $N$ and which was not contained in $\mathcal{K}_4$, and thus was not contained inside any Kinser class. In this chapter we will give a strengthening of this result, which states that the excluded minors can actually be contained inside any layer of the hierarchy.

\begin{lemma}
\label{free}
Let $r\geq 3$ be an integer. Let $P$ be the projective geometry PG$(r-1,\mathcal{K})$, where $\mathcal{K}$ is an infinite field, and let \vect{S}{t} be a finite collection of proper subspaces of $P$. If $S$ is a subspace of $P$ that is not contained in any of \vect{S}{t}, then $S$ is not contained in $S_1\cup\ldots\cup S_t$.
\end{lemma}

This is Proposition 4.2 of \cite{missingaxiom} and we will make frequent reference to it throughout this chapter. Whenever we add points freely to a subspace, it is justified by this result.

\begin{thm}
Let $n\geq 5$ be an integer. Let $\mathbb{K}$ be a infinite field and let $M$ be a $\mathbb{K}$-representable matroid.
Then $M$ is contained in an excluded minor for $\mathcal{K}_{n+1}$ which is in $\mathcal{K}_{n}$.
\end{thm}

\begin{proof}
As we can add coloops as desired, we can asssume $M$ has rank $r$ where $r\geq n$. By \cite[Lemma 2.2]{eminors}, we can assume that $M$ is partitioned into two independent hyperplanes. Call these $H_0$ and $H_{n-1}$. Let $\mathbb{K}$ be an infinite field. Imbed $M$ in the projective geometry $P=PG(r,\mathbb{K})$, so that the elements in the ground set of $M$ are identified with points in $P$. Note that this geometry has rank $r+1$, so $M$ spans a hyperplane of $P$. If $X$ is any set of points in $P$, let $\langle X\rangle$ denote the closure of $X$ in $P$. We will now extend $M$ to get $N$, an excluded minor for $\mathcal{K}_{n+1}$. First we will choose points which will not be added to the ground set of $M$, but will enable us to freely place points within $M$. 
Begin by arbitrarily choosing $x_0$ in $P-\langle E(M)\rangle$. Next freely place $x_{n-1}$ with respect to $\langle H_0\rangle$ -- i.e., choose $x_{n-1}$ in $\langle H_0\rangle$ so that $x_{n-1}$ is not spanned by any subset of $E(M)\cup\{x_0\}$ that doesn't span $H_0$. We are able to do this using \cite[Proposition 4.2]{missingaxiom}.

Choose $x_1$ in $\langle H_{n-1}\rangle$ so that it is not spanned by any subset of $E(M)\cup\{x_0,x_{n-1}\}$ that doesn't span $H_{n-1}$. Now choose $x_2$ in $\langle H_0\rangle\cap\langle H_{n-1}\rangle$ so that it is not spanned by any supset of $E(M)\cup\{x_0,x_1,x_{n-1}\}$ unless that subset spans $\langle H_0\rangle\cap\langle H_{n-1}\rangle$. Choose $x_3$ in $\langle H_0\rangle\cap\langle H_{n-1}\rangle$ so that it is not spanned by any subset of $E(M)\cup\{x_0,x_1,x_2,x_{n-1}\}$ unless that subset spans $\langle H_0\rangle\cap\langle H_{n-1}\rangle$. Continue in this way until $x_0,x_1,\ldots,x_{n-1}$ have been chosen. Now choose $r-n+1$ points in the same space, $\langle H_0\rangle\cap\langle H_{n-1}\rangle$ using the same technique. Call this set of points $X$, and note that $X\cup \{x_2,\ldots,x_{n-2}\}$ is an independent set that spans $\langle H_0\rangle\cap\langle H_{n-1}\rangle$.

The points chosen so far, $X\cup\{x_0,\ldots,x_{n-1}\}$, will act as guides for adding points to the ground set of $N$. Add a point $e_1$ to $\langle (X\cup\{x_0,\ldots,x_{n-1}\})-\{x_1,x_2\}\rangle$ so that it is not spanned by any subset of $E(M)\cup X\cup\{x_0,\ldots,x_{n-1}\}$ unless that subset spans $(X\cup\{x_0,\ldots,x_{n-1}\}-\{x_1,x_2\}$. Now add another point to the same space so that it is not spanned by any subset of $E(M)\cup X\cup\{x_0,\ldots,x_{n-1},e_1\}$ unless that subset spans $X\cup\{x_0,\ldots,x_{n-1}\}-\{x_1,x_2\}$. Contine in this way until $r-1$ points have been added to the space. Call this set of $r-1$ points $H_1$. Follow this same method to create $r-1$ points to form the set $H_2$, this time adding the points to the space $\langle(X\cup\{x_0,\ldots,x_{n-1}\})-\{x_2,x_3\}\rangle$. In this way we create $H_1,\ldots,H_{n-2}$ -- i.e. for $i\in\{1,\ldots,n-2\}$, create $H_i$ by freely placing $r-1$ points in the space $\langle(X\cup\{x_0,\ldots,x_{n-1}\})-\{x_i,x_{i+1}\}\rangle$. Note that the points of $H_0$ are in $\langle (X\cup\{x_0,\ldots,x_{n-1}\})-\{x_0,x_1\}\rangle$ and the points of $H_{n-1}$ are in $\langle (X\cup\{x_0,\ldots,x_{n-1}\})-\{x_{n-1},x_0\}\rangle$.

Finally, add a point $p$ freely to $\langle X\rangle$, then add another point $p'$ freely to $P$. Freely place a point $e$ on the line spanned by $p$ and $p'$, $\langle\{p,p'\}\rangle$, then do the same with another point $f$. Let $N$ be the matroid consisting of the points $H_0\cup\ldots H_{n-1}\cup\{e,f\}$.

\begin{lemma}
$N$ is $\mathbb{K}$-representable.
\end{lemma}

This lemma is true by construction.

\begin{lemma}
$H_i\cup\{e,f\}$ is a circuit-hyperplane of $N$ for every $i\in\{0,\ldots,n-1\}$.
\end{lemma}

\begin{proof}
$H_i\cup\{e,f\}$ has $r+1$ points, and by construction is contained in $\langle(X\cup\{p',x_0,\ldots,x_{n-1}\})-\{x_i,x_{i+1}\}\rangle$. This is a rank $r$ space and so $H_i\cup\{e,f\}$ must be dependent. Suppose $H_i$ is dependent for some $i$. Then at some point in constructing $N$, we would have added a point $g$ to already chosen elements of $H_i$ so that the point was contained in $cl(H_i-g)$, i.e. contained in $\langle (X\cup\{x_0,\ldots,x_{n-1}\})-\{x_i,x_{i+1}\}\rangle$. This contradicts every point of $H_i$ being freely placed in the space $\langle (X\cup\{x_o,\ldots,x_{n-1}\})-\{x_i,x_{i+1}\}\rangle$. Thus $H_i$ is independent. Now suppose $H_i\cup\{e\}$ is dependent. Then $e\in cl(H_i)$. That is, $e\in\langle (X\cup\{x_0,\ldots,x_{n-1}\}-\{x_i,x_{i+1}\}\rangle$. This contradicts $e$ being a point on the line spanned by $p$ and $p'$.
Likewise, $H_i\cup \{f\}$ is also independent. We have shown that every subset of $H_i\cup\{e,f\}$ is independent, meaning that $H_i\cup\{e,f\}$ must be a circuit. 

Now suppose $H_i\cup\{e,f\}$ is not a flat. There must be some element $g\in E(N)-(H_i\cup\{e,f\})$ such that $r(H_i\cup\{e,f,g\})=r(H_i\cup\{e,f\})$ -- that is, $g\in cl(H_i\cup\{e,f\})$. This implies $g\in cl(H_i\cup\{e\})$. Let $g\in H_j$ for some $j$. Assume $g\in cl(H_i)$. Then we have that $g\in\langle (X\cup\{x_0,\ldots,x_{n-1}\})-\{x_i,_{i+1}\}\rangle$. This is a contradiction, as $(X\cup\{x_0,\ldots,x_{n-1}\})-\{x_i,x_{i+1}\}$ does not span $(X\cup\{x_0,\ldots,x_{n-1}\})-\{x_j,x_{j+1}\}$.

Now suppose $g\notin cl(H_i)$. By the fourth closure axiom, $e\in cl(H_i\cup\{g\})$. Recall that $p\in cl(H_i)$. As $p$ and $e$ form a line, we must have that $\langle\{p,e\}\rangle\subseteq cl(H_i\cup\{g\})$. Thus $p'\in cl(H_i\cup g)$. As $p'$ was added freely to the projective geometry, the only way this is possible is if $H_i\cup g$ is spanning, which is a contradiction.
\end{proof}

\begin{lemma}
Relaxing $H_0\cup\{e,f\}$ produces a matroid not in $\mathbb{K}_{n+1}$.
\end{lemma}

\begin{proof}
We will show that $(X_1,X_2,\ldots,X_{n+1})=(H_0,\{e,f\},H_1,\ldots,H_{n-1})$ violates inequality $n+1$, i.e.
\begin{align*}
\sum_{i=3}^{n+1} r(X_i)+r(X_1\cup X_2)+r(X_1\cup X_3\cup X_{n+1})+\sum_{i=4}^{n+1} r(X_2\cup X_{i-1}\cup X_i) \\
> r(X_1\cup X_3)+r(X_1\cup X_{n+1})+\sum_{i=3}^{n+1} r(X_2\cup X_i)+\sum_{i=4}^{n+1} r(X_{i-1}\cup X_i)
\end{align*}
$X_i$ is independent by construction, as proven in the previous result, with rank $r-1$. Recall $X_i\cup X_2$ is a circuit-hyperplane for all $i$ as proven in the previous lemma. Note that $X_i\subseteq\langle X\cup\{x_0,\ldots,x_{n-1}\}-\{x_i,x_{i+1}\}\rangle$ for all $i\neq 2$, and that the points were chosen so as to make it an independent set of rank $r-1$.
Take two consecutive sets $X_i, X_j$, where $i,j\neq 2$. 
\begin{equation*}
\begin{split}
r(&X_i\cup X_j) \\
&\leq\quad r(\langle X\cup\{x_0,\ldots,x_{n-1}\}-\{x_i,x_{i+1}\}\rangle \cup\langle X\cup\{x_0,\ldots,x_{n-1}\}-\{x_{i+1},x_j\}\rangle)\\
&= \quad r(\langle X\cup\{x_0,\ldots,x_{n-1}\}-\{x_i,x_{i+1}\}\rangle )+r( \langle X\cup\{x_0,\ldots,x_{n-1}\}-\{x_{i+1},x_j\}\rangle)\\
& \hspace{5 mm}- r(\langle X\cup\{x_0,\ldots,x_{n-1}\}-\{x_i,x_{i+1}\}\rangle \cap\langle X\cup\{x_0,\ldots,x_{n-1}\}-\{x_{i+1},x_j\}\rangle\\
&= \quad (r-1)+(r-1)-(r-2)\\
&= \quad r
\end{split}
\end{equation*}
Now suppose $X_i, X_j$ are inconsecutive. The intersection term will now have rank $r-3$, one less than before, so $r(X_i\cup X_j)=r+1$. Note that these two calculations imply the rank of the union of any three $X_i$ must be $r+1$. We can now show that the inequality above holds:
$$\sum_{i=3}^{n+1} (r-1)+(r+1)+(r+1)+\sum_{i=4}^{n+1} (r+1)
> r+r+\sum_{i=3}^{n+1} r+\sum_{i=4}^{n+1} r$$
Therefore \vect{X}{n+1} is a bad family if and only if
\begin{equation*}
\begin{split}
& (n+1-2)(r-1)+2(r+1)+(n+1-3)(r+1)\\
& \hspace{5 mm}>2r+(n+1-2)r+(n+1-3)r
\end{split}
\end{equation*}
which is true if and only if
$$\begin{array}{ccrcl}
&\quad&(n-1)(r-1)+n(r+1)& >&(2n-1)r\\
\Leftrightarrow&&(2n-1)r+1&>&(2n-1)r\\
\end{array}$$

and this completes the proof.
\end{proof}
Call this relaxation $N'$.
\begin{lemma}
\label{2CH}
Relaxing $H_i\cup\{e,f\}$ in $N$ creates a $\mathbb{K}$-representable matroid.
\end{lemma}

\begin{proof}
Construct $L$ in exactly the same way as $N$, up until the point where $p$ and $p'$ are added. Instead of adding $p$ to $\langle X\rangle$, add it freely to $\langle X\cup\{x_1,\ldots,x_i\}\rangle$. Now add $p'$ freely to $\langle X\cup\{x_{i+1},\ldots,x_{n-1},x_0\}\rangle$. Then add $e$ and $f$ freely to the line $\langle\{p,p'\}\rangle$ as before. This matroid $L$ is $\mathbb{K}$-representable by construction. We will show that it is the same as the matroid obtained from $N'$ by relaxing $H_i\cup\{e,f\}$, referred to as $N''$. 

Note that by \cite[Proposition 3.3.5]{Oxley}, we have that $N\backslash e\backslash f=N'\backslash e\backslash f=N''\backslash e\backslash f$, and also that $N\backslash e\backslash f =L\backslash e\backslash f$ by construction. If $Z\subseteq E(N\backslash e)$ spans $f$, then, as we chose $f$ to be freely placed on the line spanned by $p$ and $p'$, $Z$ must span $\langle\{p,p'\}\rangle$. This implies that $p'\in\langle Z\rangle$. As $p'$ was freely placed in $E(N\backslash e)$, this implies $Z$ is spanning. Thus $N\backslash e$ is a free extension of $N\backslash e\backslash f$ by the element $f$.

Now suppose $Z\subseteq E(L\backslash e)$ spans $f$. Then again we have that $\langle\{p,p'\}\rangle\subseteq\langle Z\rangle$. As this gives that $p\in\langle Z\rangle$, we have that $X\cup\{x_1,\ldots,x_i\}\subseteq\langle Z\rangle$ by the way $p$ was chosen in the construction of $L$. As $p'\in\langle Z\rangle$, we have that $X\cup\{x_{i+1},\ldots,x_{n-1},x_0\}\subseteq\langle Z\rangle$. Putting these together gives $X\cup\{x_0,\ldots,x_{n-1}\}\subseteq\langle Z\rangle$. As $X\cup\{x_0,\ldots,x_{n-1}\}$ was chosen so as to be a basis of $L$, $Z$ must be spanning. This tells us that $f$ is freely placed in $L\backslash e$, so $L\backslash e$ is a free extension of $L\backslash e\backslash f$ by the element $f$. As $L\backslash e\backslash f=N\backslash e\backslash f$, we have that $L\backslash e=N\backslash e$. Note also that $N\backslash e=N'\backslash e=N''\backslash e$, so $L\backslash e=N''\backslash e$. The same argument shows that $L\backslash f=N''\backslash f$.

Suppose $L\neq N''$. There must exist a set $A$ which is a non-spanning circuit in $N''$ and independent in $L$ or vice versa. The previous results tell us that $e,f\in A$, as otherwise $A$ would have the same rank in both matroids. 

Suppose $A$ is a non-spanning circuit in $L$. Say that the points in $E(L)-\{e,f\}$ were added in the order $e_1,....,e_{t}$. Let $e_j$ be the largest element of $A$ according to this ordering, and let $e_j\in H_k$. As $e_j$ was freely placed, $A$ must span $\langle (X\cup\{x_0,\ldots,x_{n-1}\})-\{x_k,x_{k+1}\}\rangle$. This means that $(A-H_k) \cup 
(X\cup \{x_0,...,x_{n-1}\})-\{x_k,x_{k+1}\})$ spans the same set as $A$. Suppose the last element added to $A$ before those in $H_k$ is $e_l\in H_j$ where $j<k$. Then $(A-H_k) \cup (X\cup \{x_0,...,x_{n-1}\})-\{x_k,x_{k+1}\}$ spans an element from $H_j$, and by construction, as every element in the set above was added before $H_j$, we see that this set spans $(X\cup \{x_0,...,x_{n-1}\})-\{x_j,x_{j+1}\}$. Thus $A$ spans both $\langle (X\cup\{x_0,\ldots,x_{n-1}\})-\{x_k,x_{k+1}\}\rangle$ and $\langle (X\cup\{x_0,\ldots,x_{n-1}\})-\{x_j,x_{j+1}\}\rangle$. As shown in the previous lemma, if $H_j$ and $H_k$ are inconsecutive, $A$ will have rank $r+1$ and be spanning. $H_j$ and $H_k$ thus must be consecutive in order for $A$ to be non-spanning. Take a dependent subset of $H_j\cup H_k$ in $L$. As this subset does not include $e$ nor $f$, it has the same rank in $L\backslash e$. Likewise, the rank of the subset in $N''$ has the same rank in $N''\backslash e$. As we have already shown $L\backslash e=N''\backslash e$, we have that any dependent subset of $H_j\cup H_k$ in $L$ is also dependent in $N''$. This contradicts the assumption that $A$ is independent in $N''$. If there is no point contained in a set $H_j$ where $j<k$, in order for $A$ to be a circuit, $A$ must be equal to $H_k\cup\{e,f\}$, where $k\notin\{0,i\}$, as any subset of this is independent in $L$, as proved in the next lemma.

\begin{lemma}
$H_k\cup\{e,f\}$ is a circuit of $L$ for all $k\in\{1,\ldots,i-1,i+1,\ldots,n\}$.
\end{lemma}

\begin{proof}
Consider $H_k\cup\{e,f\}$. Recall that $p$ was added freely to $\langle X\cup\{x_1,\ldots,x_i\}\rangle$ while $p'$ was added freely to $\langle X\cup\{x_{i+1},\ldots,x_{n-1},x_0\}\rangle$. $H_k$ is contained in $\langle (X\cup\{x_0,\ldots,x_{n-1}\}-\{x_k,x_{k+1}\}\rangle$. When $i\leq k$, this subspace spans $\langle X\cup\{x_1,\ldots,x_i\}\rangle$ and so spans $p$. When $i\geq k$, this subspace spans $\langle X\cup\{x_{i+1},\ldots,x_{n-1},x_0\}\rangle$ and so spans $p'$. As $e$ and $f$ were freely placed on the line spanned by $p$ and $p'$, in either case we have that $H_k\cup\{e,f\}\in cl(H_k\cup\{e\})$ and so $H_k\cup\{e,f\}$ is dependent.

Suppose $H_k$ is dependent for some $k$. Then at some point in constructing $L$, we would have added a point $g$ to already chosen elements of $H_k$ so that the point was contained in $cl(H_k-g)$, i.e. contained in $\langle (X\cup\{x_0,\ldots,x_{n-1}\})-\{x_k,x_{k+1}\}\rangle$. This contradicts each of the $r-1$ points of $H_k$ being freely placed in the rank $r-1$ space $\langle (X\cup\{x_o,\ldots,x_{n-1}\})-\{x_k,x_{k+1}\}\rangle$. Thus $H_k$ is independent. 

Now suppose $H_k\cup\{e\}$ is dependent. Then $e\in cl(H_k)$ -- that is, $e\in\langle (X\cup\{x_0,\ldots,x_{n-1}\}-\{x_k,x_{k+1}\}\rangle$. However, $e$ was freely placed on the line spanned by $p$ and $p'$. Thus if $H_k$ spans $e$, it must span this line. As the line itself is free in the matroid, for this to happen, $H_k$ must be spanning -- contradiction.
Likewise, $H_k\cup \{f\}$ is also independent. We have shown that every subset of $H_k\cup\{e,f\}$ is independent, meaning that $H_k\cup\{e,f\}$ must be a circuit. 
\end{proof}

Thus $H_k\cup\{e,f\}$ is dependent in $L$. We have shown that it is also dependent in $N''$, so again have a contradiction to $A$ being independent in $L$. The same argument shows that if $A$ is dependent in $N''$, $A$ is also dependent in $L$. Thus $L=N''$.
\end{proof}
We constructed $N$ to be representable, so $N$ must satisfy every Kinser inequality. In particular, it must be contained inside $\mathcal{K}_{n+1}$. Next we have shown that if we relax a single circuit-hyperplane of $N$, the resulting matroid $N'$ has a bad family for $\mathcal{K}_{n+1}$. We will now show that $N'$ is in fact an excluded minor for $\mathcal{K}_{n+1}$ -- that is, we will show that each proper minor of $N'$ is representable and thus is contained in $\mathcal{K}_{n+1}$.

First suppose that $x\in H_j$ where $j\neq 0$. Let $N''=N'$ with the circuit-hyperplane $H_j\cup\{e,f\}$ relaxed. By \cite[Proposition 3.3.5]{Oxley}, $N''\backslash x=N'\backslash x$. As $N''$ is $\mathbb{K}$-representable by Theorem \ref{2CH}, and representability is preserved under minors, $N'\backslash x$ is $\mathbb{K}$-representable. Say $l\in\{0,\ldots,n-1\}-\{0,j\}$. Now let $N''=N'$ with $H_l\cup\{e,f\}$ relaxed. Also by \cite[Proposition 3.3.5]{Oxley}, we have that $N''/x=N'/x$, and so $N'/x$ is $\mathbb{K}$-representable.

Next, suppose $x\in H_0$. As $N'=N$ with the circuit-hyperplane $H_0\cup\{e,f\}$ relaxed, we have that $N'\backslash x=N\backslash x$, so $N'\backslash x$ is $\mathbb{K}$-representable. Let $N''=N'$ with $H_i\cup\{e,f\}$ relaxed. We have that $N''/x=N'/x$, so $N'/x$ is $\mathbb{K}$-representable.

Now suppose $x$ is equal to $e$. As $e$ and $f$ were freely placed on the line spanned by $p$ and $p'$, the same argument as follows works for $x=f$. We have that $N'\backslash e=N\backslash e$, so $N'\backslash e$ is $\mathbb{K}$-representable. 

Finally, consider $N'/e$. Take some $z\in H_0$. Recall that $N'=N$ with the circuit-hyperplane $H_0\cup\{e,f\}$ relaxed. Note that $N'/e$ is obtained from $N/e$ by relaxing $H_0\cup\{f\}$. This gives us that $N'/e\backslash z=N/e\backslash z$, as deleting $z$ effectively undoes the relaxation. As $N$ is $\mathbb{K}$-representable, and thus $N/e\backslash z$ is $\mathbb{K}$-representable, $N'/e\backslash z$ is also $\mathbb{K}$-representable. Let $Z\subseteq E(N'/e)$ be such that $z\notin Z$ and $z\in cl_{N'/e}(Z)$. $N'/e$ is a relaxation of $N/e$ which can only affect closures in so far as that some may contain additional elements in $N/e$, so $z\in cl_{N/e}(Z)$. This implies that $z\in cl_{N}(Z\cup\{e\})$ by \cite[Proposition 3.1.11]{Oxley}. Due to the way $H_0$ was constructed, we thus have that $\langle Z\cup\{e\}\rangle\supseteq (X\cup\{x_0,\ldots,x_{n-1}\})-\{x_0,x_1\}$. As we have that $z\in cl_{N}(Z\cup\{e\})$ and all elements of $H_0$ are freely placed in the relevant subspace, $Z\cup\{e,f\}$ must thus also span every other element of $H_0$. As $e$ and $f$ were freely placed on the line spanned by $p$ and $p'$, we also have that $f\in cl_{N}(Z\cup\{e\})$. Thus, in $N$, $H_0\cup\{e,f\}$ is contained in $\langle Z\cup\{e\}\rangle$. As $H_0\cup\{e,f\}$ is a circuit-hyperplane, this implies that either $Z\cup\{e\}$ is spanning in $N$ or that $Z\cup\{e\}=H_0\cup\{e,f\}$. If $Z\cup\{e\}=H_0\cup\{e,f\}$, we have a contradiction to the assumption that $z\notin Z$. We thus have that $Z\cup\{e\}$ is spanning in $N$. This means that $Z\cup\{e\}$ is also spanning in $N'$, and, as $r(N'/e)=r(N')-1$, that $Z$ is spanning in $N'/e$. We have that $z$ is only in the closure of a subset of $N'/e$ when that subset spans $N'/e$ -- that is, we have shown that $z$ is freely placed in $N'/e$. Thus $N'/e$ is a free extension of $N'/e\backslash z$ by $z$. As $N'/e\backslash z$ is $\mathbb{K}$-representable and this fact is preserved under free extentions, we have that $N'/e$ is $\mathbb{K}$-representable. 

We have now shown that every minor of $N'$ is $\mathbb{K}$-representable and so contained in $\mathcal{K}_{n+1}$, making $N'$ an excluded minor for $\mathcal{K}_{n+1}$. This completes the proof of Theorem \ref{MNW}.
\end{proof}


\chapter{Conjectures}

Finally, we give some conjectures on the hierachy of the Kinser classes.

\begin{conj}
Let $n>5$. $\mathcal{K}_n\neq\mathcal{K}_n^*$.
\end{conj}

As shown in the previous chapter, verifying that a matroid satisfies a Kinser inequality is very difficult. Given the amount of difficulty involved in proving that the fifth Kinser class is not dual closed, proving this result in general would involve an even greater amount of work. Based on that case, however, we give a strengthening of the above conjecture.

\begin{conj}
$\kin{n}^-\in\mathcal{K}_n^*-\mathcal{K}_n$
\end{conj}

One further question about the structure of the hierarchy is how each dual class sits within the previous Kinser class. There are two possibilities here, and we conjecture that the following is true.

\begin{conj}
\label{conj}
Let $n>4$. $\mathcal{K}_{n+1}^*\subseteq\mathcal{K}_n$.
\end{conj}

The following conjecture is true if Conjecture \ref{conj} is as well. 

\begin{conj}
$\mathcal{K}_\infty^*=\mathcal{K}_\infty$
\end{conj}

To see that this follows from Conjecture \ref{conj}, assume $M\in\mathcal{K}_\infty$, but $M\not\in\mathcal{K}_\infty^*$. Then there exists an integer $n$ such that $M\not\in\mathcal{K}_n^*$. However, this contradicts Conjecture \ref{conj}, which gives that $M\in\mathcal{K}_{n+1}\subseteq\mathcal{K}_n^*$.

Assuming these conjectures to hold true, we have a final diagram of the hierarchy. 

\begin{figure}[h]
\centering
\includegraphics[scale=0.6]{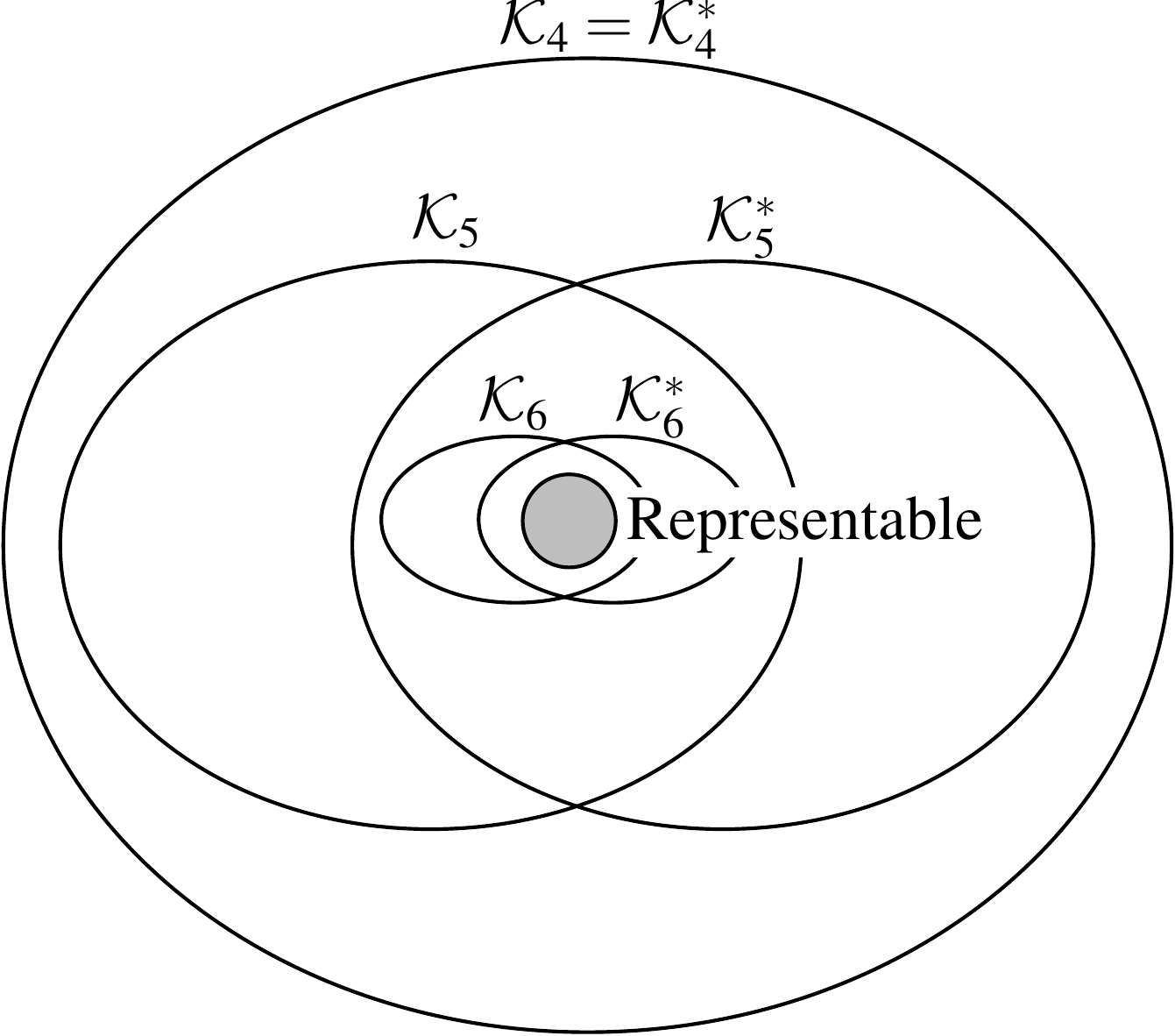}
\caption{Kinser classes (4)}
\end{figure}

\pagebreak

Now we will consider two classes of matroids which we conjecture satisfy every Kinser inequality.

\begin{dfn}
\label{dgeom}
Let $G$ be an abelian group. Take a complete graph and add a loop to very vertex, replace every edge with a parallel class of $|G|$ edges. Call this graph $H$. Orient every edge which is not a loop so that parallel edges have the same direction. Bijectively label each parallel class with the elements of $G$, and label loops with non-identities. 
Let $C$ be a cycle of $H$. Consider the product of group labels in $C$, taken in cyclic order, where if an edge is oriented against the cyclic order we take the inverse of its label instead. If the result is the identity, call $C$ \emph{positive}. Otherwise, call $C$ \emph{negative}.
Take a Dowling geometry $H$. There exists a matroid which has $E(H)$ as its ground set, and set of circuits equal to the positive cycles of $H$ and minimal connected subgraphs that contain two negative cycles. Call this matroid a \emph{Dowling geometry}.
\end{dfn}

Take a field $\mathbb{F}$. Recall that $\mathbb{F}^\times$ is the multiplicative group consisting of the non-zero elements of $\mathbb{F}$.

\begin{lemma}
\cite[Theorem 6.10.10]{Oxley} Take a Dowling geometry of rank $r$ over a finite group $G$. This matroid is representable over a field $\mathbb{F}$ if and only if $G$ is isomorphic to a subgroup of $\mathbb{F}^\times$.
\end{lemma}

In this case, the Dowling matroid satisfies every Kinser inequality. We also have that if $G$ is a finite subgroup of the multiplicative group of a field, then $G$ is cyclic by \cite{group}. The following conjecture is thus open when $G$ is both finite and non-cyclic.

\begin{conj}
A Dowling geometry satisfies every Kinser inequality.
\end{conj}

Now we will consider matroids which are representable over skew partial fields. All of the following definitions and results can be found in \cite{skew}.

\begin{dfn}
\label{spf}
A \emph{skew partial field} is a pair $(R,G)$ where $R$ is a ring, and $G$ is a subgroup of the group of units of $R$, such that $-1\in G$.
\end{dfn}

\begin{dfn}
\label{chain}
Let $R$ be a ring, and let $E$ be a finite set. An \emph{R-chain group} on $E$ is a subset $C\subseteq R^E$ such that, for all $f,g\in C$ and $r\in R$,
\begin{itemize}
\item[i.] $0\in C$,
\item[ii.] $f+g\in C$,
\item[iii.] $rf\in C$
\end{itemize}
\end{dfn}

The elements of $C$ are called \emph{chains}, and the \emph{support} of a chain $c=\{c_1,\ldots,c_e\}\in C$ is $$||c||=\{i\in E \ | \ c_i\neq 0\}$$

\begin{dfn}
A chain $c\in C$ is \emph{elementary} if $c\neq 0$ and there is no $c'\in C-\{0\}$ with $||c'||\subset ||c||$.
\end{dfn}

\begin{dfn}
Let $G$ be a subgroup of the group of units of $R$. A chain $c\in C$ is \emph{G-primitive} if $c\in (G\cup\{0\})^E$.
\end{dfn}

\begin{dfn}
Let $\mathbb{P}=(R,G)$ be a skew partial field, and $E$ a finite set. A \emph{$\mathbb{P}$-chain group} on $E$ is an $R$-chain group $C$ on $E$ such that every elementary chain $c\in C$ can be written as $c=rc'$ for some $G$-primitive chain $c'\in C$ and some $r\in R$.
\end{dfn}

\begin{lemma}
Let $\mathbb{P}=(R,G)$ be a skew partial field, and let $C$ be a $\mathbb{P}$-chain group on $E$. Then $\mathcal{C}^*=\{||c|| \ | \ c\in C, \ c \ \text{is elementary}\}$ is the set of cocircuits of a matroid on $E$.
\end{lemma}

A matroid $M$ is said to be \emph{$\mathbb{P}$-representable} if there exists a $\mathbb{P}$-chain group $C$ such that $M=M(C)$. If a matroid is representable and thus satisfies every Kinser inequality, it is representable over a skew partial field.

\begin{conj}
Take a matroid $M$ which is $\mathbb{P}$-representable. $M$ satisfies every Kinser inequality.
\end{conj}

\bibliographystyle{acm}
\bibliography{library} 
\end{document}